\documentclass[a4paper,11pt]{article}
\usepackage{geometry}
\usepackage{hyperref}
\usepackage[english]{babel}
\usepackage{amsmath,amssymb,textcomp,enumerate,xcolor,latexsym,theorem}
\usepackage{mathrsfs}
\usepackage{graphicx}

\usepackage{pxfonts}
\usepackage[T1]{fontenc}

\catcode`\>=\active \def>{
\fontencoding{T1}\selectfont\symbol{62}\fontencoding{\encodingdefault}}
\newcommand{\assign}{:=}
\newcommand{\cdummy}{\cdot}
\newcommand{\comma}{{,}}
\newcommand{\mathD}{\mathrm{D}}
\newcommand{\mathd}{\mathrm{d}}
\newcommand{\nocomma}{}
\newcommand{\noplus}{}
\newcommand{\nosymbol}{}
\newcommand{\tmSep}{; }
\newcommand{\tmaffiliation}[1]{\thanks{\textit{Affiliation:} #1}}
\newcommand{\tmdummy}{$\mbox{}$}
\newcommand{\tmem}[1]{{\em #1\/}}
\newcommand{\tmemail}[1]{\thanks{\textit{Email:} \texttt{#1}}}

\newcommand{\tmname}[1]{\textsc{#1}}
\newcommand{\tmop}[1]{\ensuremath{\operatorname{#1}}}
\newcommand{\tmscript}[1]{\text{\scriptsize{$#1$}}}
\newcommand{\tmtextbf}[1]{{\bfseries{#1}}}
\newcommand{\tmtextit}[1]{{\itshape{#1}}}
\newcommand{\upl}{+}
\newenvironment{enumeratenumeric}{\begin{enumerate}[1.] }{\end{enumerate}}
\newenvironment{enumerateroman}{\begin{enumerate}[i.] }{\end{enumerate}}
\newenvironment{proof}{\noindent\textbf{Proof\ }}{\hspace*{\fill}$\Box$\medskip}
\newtheorem{corollary}{Corollary}
\newtheorem{definition}{Definition}
{\theorembodyfont{\rmfamily}\newtheorem{example}{Example}}
{\theorembodyfont{\rmfamily\small\it}\newtheorem{exercise}{Exercise}}
\newtheorem{lemma}{Lemma}
{\theorembodyfont{\rmfamily}\newtheorem{remark}{Remark}}
\newtheorem{theorem}{Theorem}
\newcommand{\tmfloatcontents}{}
\newlength{\tmfloatwidth}
\newcommand{\tmfloat}[5]{
  \renewcommand{\tmfloatcontents}{#4}
  \setlength{\tmfloatwidth}{\widthof{\tmfloatcontents}+1in}
  \ifthenelse{\equal{#2}{small}}
    {\ifthenelse{\lengthtest{\tmfloatwidth > \linewidth}}
      {\setlength{\tmfloatwidth}{\linewidth}}{}}
    {\setlength{\tmfloatwidth}{\linewidth}}
  \begin{minipage}[#1]{\tmfloatwidth}
    \begin{center}
      \tmfloatcontents
      \captionof{#3}{#5}
    \end{center}
  \end{minipage}}

\newcommand{\zzone}{\text{\resizebox{.7em}{!}{\includegraphics{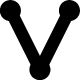}}}}
\newcommand{\zztwo}{\text{\resizebox{.7em}{!}{\includegraphics{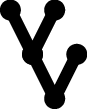}}}}
\newcommand{\zzthree}{\text{\resizebox{.7em}{!}{\includegraphics{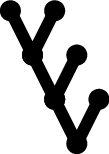}}}}
\newcommand{\zzfour}{\text{\resizebox{1em}{!}{\includegraphics{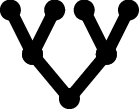}}}}

\newcommand{\LL}{\mathscr{L}}
\newcommand{\CA}{\mathscr{A}}
\newcommand{\CB}{\mathscr{B}}
\newcommand{\CF}{\mathscr{F}}
\newcommand{\CS}{\mathscr{S}}
\newcommand{\CC}{\mathscr{C}}
\newcommand{\CD}{\mathscr{D}}
\newcommand{\DD}{\mathscr{D}}

\newcommand{\R}{\mathbb{R}}

\newcommand{\lpara}{\,\mathord{\prec}\,}
\newcommand{\rpara}{\,\mathord{\succ}\,}
\newcommand{\reso}{\,\mathord{\circ}\,}
\newcommand{\renorm}{\,\mathord{\diamond}\,}

\begin{document}

\title{\bf{Lectures on singular stochastic PDEs}}

\author{
  M. Gubinelli
  \tmaffiliation{CEREMADE \& CNRS UMR 7534{\tmSep}
  Universit{\'e} Paris Dauphine and IUF, France{\tmSep}
  Hausdorff Center of Mathematics \& Institute of Applied Mathematics{\tmSep}
  Universit\"at Bonn, Germany}
  \tmemail{gubinelli@iam.uni-bonn.de}
  \and
  N. Perkowski
  \tmaffiliation{Humboldt--Universit\"at zu Berlin{\tmSep}
  Institut f\"ur Mathematik}
  \tmemail{perkowsk@math.hu-berlin.de}
}

\date{Version 1.10 -- October 2015}

\maketitle

\begin{abstract}
  These are the notes for a course at the 18th Brazilian School of Probability
  held from August 3rd to 9th, 2014 in Mambucaba. The aim of the course is to
  introduce the basic problems of non--linear PDEs with stochastic and
  irregular terms. We explain how it is possible to handle them using two main
  techniques: the notion of energy
  solutions~{\cite{goncalves_universality_2010,gubinelli_regularization_2013}}
  and that of paracontrolled distributions, recently introduced
  in~{\cite{Gubinelli2012}}. In order to maintain a link with physical
  intuitions, we motivate such singular SPDEs via a homogenization result for
  a diffusion in a random potential. \ 
\end{abstract}

{\tableofcontents}

\section{Introduction}

The aim of these lectures is to explain how to apply controlled path
ideas~{\cite{Gubinelli2004}} to solve basic problems in singular stochastic  parabolic
equations. The hope is that the insight gained by doing so can inspire new applications or
the construction of other more powerful tools to analyze a wider class of
problems.

\

To understand the origin
of such singular equations, we have chosen to present the example of a
homogenization problem for a singular potential in a linear parabolic
equation. This point of view has the added benefit that it allows us to track
back the renormalization needed to handle the singularities as effects living
on other scales than those of interest. The basic problem is that of having to
handle effects of the microscopic scales and their interaction through
non--linearities on the macroscopic behaviour of the solution.

\

Mathematically, this problem translates into the attempt of making Schwartz's
theory of distribution coexist with non--linear operations which are
notoriously not continuous in the usual topologies on distributions. This is a
very old problem of analysis and has been widely studied. The additional input
which is not present in the usual approaches is that the singularities which
force us to treat the problem in the setting of Schwartz's distributions are
of a stochastic nature. So we dispose of two handles on the problem: the
analytical one and the probabilistic one. The right mix of the two will
provide an effective solution to a wide class of problems.

\

A first and deep understanding of these problems has been obtained starting
from the late '90s by T. Lyons~{\cite{Lyons1998}}, who introduced a theory of
{\tmem{rough paths}} in order to settle the conflict of topology and
non--linearity in the context of driven differential equations, or more
generally in the context of the non--linear analysis of time--varying signals.
Nowadays there are many expositions of this
theory~{\cite{Lyons2002,Friz2010,Lyons2007,Friz2014}} and we refer the reader
to the literature for more details.

\

In~{\cite{Gubinelli2004,Gubinelli2010}}, the notion of {\tmem{controlled
paths}} has been introduced in order to extend the applicability of the rough
path ideas to a larger class of problems that are not necessarily related to
the integration of ODEs but which still retain the one--dimensional nature of
the directions in which the irregularity manifest itself. The controlled path
approach has been used to make sense of the evolution of irregular objects
such as vortex filaments and certain SPDEs. Later Hairer understood how to
apply these ideas to the long standing problem of the Kardar--Parisi--Zhang
equation~{\cite{Hairer2013b}}, and his insights prompted the researchers to
try more ambitious approaches to extend rough paths to a multidimensional
setting.

\

In~{\cite{Gubinelli2012}}, in collaboration with P. Imkeller, we introduced a
notion of {\tmem{paracontrolled distributions}} which is suitable to
handle a wide class of SPDEs which were well out of reach with previously
known methods. Paracontrolled distributions can be understood as an extension
of controlled paths to a multidimensional setting, and they are based on new
combinations of basic tools from harmonic analysis.

\

At the same time, Hairer managed to devise a vast generalization of the basic
construction of controlled rough paths in the multidimensional and
distributional setting, which he called the theory of {\tmem{regularity
structures}}~{\cite{Hairer2014}} and which subsumes standard analysis based on
H{\"o}lder spaces and controlled rough path theory but goes well beyond that.
Just few days after the lectures in Mambucaba took place, it was announced
that Martin Hairer was awarded a Fields Medal for his work on SPDEs and in
particular for his theory of regularity structures~{\cite{Hairer2014}} as a
tool for dealing with singular SPDEs. This prize witnesses the exciting period
we are experiencing: we now understand sound lines of attack to long standing
problems, and there are countless opportunities to apply similar ideas to new
problems.

\

\

The plan of the lectures is the following. We start by discussing \emph{energy
solutions}~{\cite{goncalves_universality_2010,gubinelli_regularization_2013}} of the stationary stochastic Burgers equation (one of the avatars of the Kardar--Parisi--Zhang equation). Energy solutions have the advantage of being relatively easy to handle and of being based on tools that are familiar to probabilists. On the other side, they only apply in the specific example of the stochastic Burgers equation in equilibrium, and here we will only focus on the existence but not on the uniqueness of energy solutions. Starting our lectures in this way will allow us to introduce the reader to SPDEs in a progressive manner, and
also to introduce Gaussian tools on the way (Wick products, hypercontractivity) and to
present some of the basic phenomena that appear when dealing with singular
SPDEs. Next we set up the analytical tools we need in the rest of the
lectures: Besov spaces and some basic harmonic analysis based on the
Littlewood--Paley decomposition of distributions. In order to motivate the reader and to
provide a physical ground for the intuition to stand on, we then discuss a
homogenization problem for the linear heat equation with random potential 
which describes diffusion in a random environment.
This will allow us to derive the need for the weak topologies we shall use and
for irregular objects like the white noise from first principles and ``concrete''
applications. The homogenization problem also allows us to see that there are naturally appearing
renormalization effects and to keep track of their mathematical
meaning. Starting from these problems we introduce the two--dimensional
parabolic Anderson model, the simplest SPDE in which most of the features of
more difficult problems are already present, and we explain how to use
paraproducts and the paracontrolled ansatz in order to keep the non--linear
effect of the singular data under control. Then we return to the stochastic Burgers equation and show how to apply paracontrolled distribution in order to obtain the existence and uniqueness of solutions also in the non--stationary case.

\paragraph{Acknowledgements.} The authors would like to thank the two anonymous referees for the careful reading and the manifold suggestions which helped up to greatly improve the manuscript. We would also like to thanks the organisers of the Brazilian Summer Schools in Probability for the invitation and the researchers who attended the meeting for the wonderful atmosphere.\

The main part of the research was carried out while N.P. was employed by Universit\'e Paris Dauphine. N.P. was supported by the Fondation Sciences Math\'ematiques de Paris (FSMP) and by a public grant overseen by the French National Research Agency (ANR) as part of the ``Investissements d'Avenir'' program (reference: ANR-10-LABX-0098).

\ \

\paragraph{Conventions and notations.}We write $a \lesssim b$ if there exists
a constant $C > 0$, independent of the variables under consideration, such
that $a \leqslant C b$. Similarly we define $\gtrsim$. We write $a \simeq b$
if $a \lesssim b$ and $b \lesssim a$. If we want to emphasize the dependence
of $C$ on the variable $x$, then we write $a (x) \lesssim_x b (x)$.

If $a$ is a complex number, we write $a^\ast$ for its complex conjugate.

If $i$ and $j$ are index variables of Littlewood--Paley blocks (to be defined
below), then $i \lesssim j$ is to be interpreted as $2^i \lesssim 2^j$, and
similarly for $\simeq$ and $\lesssim$. In other words, $i \lesssim j$ means $i
\leqslant j + N$ for some fixed $N \in \mathbb{N}$ that does not depend on
$i$ or $j$.

We use standard multi-index notation: for ${\mu} \in \mathbb{N}^d_0$ we
write $| {\mu} | ={\mu}_1 + \ldots +{\mu}_d$ and
$\partial^{{\mu}} = \partial^{| {\mu} |} /
\partial^{{\mu}_1}_{x_1} \ldots \partial^{{\mu}_d}_{x_d}$, as well as
$x^{{\mu}} = x^{{\mu}_1}_1 \cdummy \ldots \cdummy x^{{\mu}_d}_d$
for $x \in \mathbb{R}^d$.

For $\alpha > 0$ we write $C^{\alpha}_b$ for the bounded functions $F : \mathbb{R}
\rightarrow \mathbb{R}$ which are $\lfloor \alpha \rfloor$ times continuously
differentiable with bounded and $(\alpha - \lfloor \alpha \rfloor)$--H{\"o}lder continuous
derivatives of order $\lfloor \alpha \rfloor$, equipped with the norm
$$\| F \|_{C^{\alpha}_b} =\sup_{\mu: 0\le |\mu| \le \lfloor \alpha \rfloor}\|\partial^{{\mu}} F\|_{L^\infty} + \sup_{\mu : |\mu| = \lfloor \alpha \rfloor}\sup_{x\neq y} \frac{|\partial^{{\mu}} F(x)-\partial^{{\mu}} F(y)|}{|x-y|^{\alpha - \lfloor \alpha \rfloor}}.$$

If we write $u \in \CC^{\alpha -}$, then that means that $u$ is in
$\CC^{\alpha - \varepsilon}$ for all $\varepsilon > 0$. The $\CC^{\alpha}$
spaces will be defined below.

If $\mathbb{X}$ is a Banach space with norm $\| \cdummy
\|_{\mathbb{X}}$ and if $T > 0$, then we define $C\mathbb{X}$ and $C_T
\mathbb{X}$ as the spaces of continuous functions from $[0, \infty)$
respectively $[0, T]$ to $\mathbb{X}$, and $C_T \mathbb{X}$ is equipped with
the supremum norm $\| \nosymbol \cdummy \|_{C_T \mathbb{X}}$. If $\alpha \in
(0, 1)$ then we write $C^{\alpha} \mathbb{X}$ for the functions in
$C\mathbb{X}$ that are $\alpha$--H{\"o}lder continuous on every interval $[0,
T]$, and we write
\[ \| f \|_{C^{\alpha}_T \mathbb{X}} = \sup_{0 \leqslant s < t \leqslant T}
   \frac{\| f (t) - f (s) \|}{| t - s |^{\alpha}} . \]

\section{Energy solutions}\label{sec:energy}

The first issue one encounters when dealing with singular SPDEs is the
ill--posed character of the equation, even in a weak sense. Typically, the equation features some non--linearity that does not make sense in the natural spaces where solutions live
and one has to provide a suitable smaller space in which it is possible to
give an appropriate interpretation to ``ambiguous quantities'' that appear in the equation.

\

Energy
solutions~{\cite{goncalves_universality_2010,gubinelli_regularization_2013}}
are a relatively simple tool in order to come up with well--defined
non--linearities. Moreover, proving existence of energy solutions or even convergence
to energy solutions is usually a quite simple problem, at least compared to
the other approaches like paracontrolled solutions or regularity structures,
where already existence requires quite a large amount of computations but
where uniqueness can be established quite easily afterwards.
The main drawback is that we lack of general uniqueness results for energy solutions. Only very recently, after the completion of these notes, we were able to prove that energy solutions for the stationary stochastic Burgers equation are unique. This topic will not be touched upon here. The interested reader can find the details in the preprint~\cite{Gubinelli2015b}. 

\subsection{Distributions}

We will need to use distributions defined on the $d$-dimensional torus $\mathbb{T}^d$ where $\mathbb{T}=\mathbb{R}/ (2 \pi \mathbb{Z})$. We collect here some basic results and definitions. The space of distributions $\CS' = \CS' (\mathbb{T}^d)$ is the set of linear
maps $f$ from $\CS = C^{\infty} (\mathbb{T}^d, \mathbb{C})$ to
$\mathbb{C}$, such that there exist $k \in \mathbb{N}$ and $C > 0$ with
\[ | \langle f, \varphi \rangle | : = | f (\varphi) | \leqslant C \sup_{|
   {\mu} | \leqslant k} \| \partial^{{\mu}} \varphi \|_{L^{\infty}
   (\mathbb{T}^d)} \]
for all $\varphi \in \CS$.

\begin{example}
  Clearly $L^p = L^p (\mathbb{T}^d) \subset \CS'$ for all $p \geqslant 1$,
  and more generally the space of finite signed measures on $(\mathbb{T}^d, \mathcal{\CB} (\mathbb{T}^d))$ is contained in
  $\CS'$. Another example of a distribution is $\varphi \mapsto
  \partial^{{\mu}} \varphi (x)$ for ${\mu} \in \mathbb{N}^d_0$ and $x
  \in \mathbb{T}$.
\end{example}

In particular, the Fourier transform $\CF f : \mathbb{Z}^d \rightarrow
\mathbb{C}$,
\[ \CF f (k) = \hat{f} (k) = \langle f, e_k \rangle, \]
with $e_k = e^{- i \langle k, \cdummy \rangle} / (2 \pi)^{d / 2}$, is defined
for all $f \in \CS'$, and it satisfies $| \CF f (k) | \leqslant | P
(k) |$ for a suitable polynomial $P$. Conversely, if $(g (k))_{k \in
\mathbb{Z}^d}$ is at most of polynomial growth, then its inverse Fourier
transform
\[ \CF^{- 1} g = \sum_{k \in \mathbb{Z}^d} g (k) e_k^{\ast} \]
defines a distribution (here $e_k^\ast = e^{i \langle k, \cdummy \rangle} / (2 \pi)^{d / 2}$ is the complex conjugate of $e_k$).

\begin{exercise}
  \label{exo:fourier inversion}Show that the Fourier transform of $\varphi \in
  \CS$ decays faster than any rational function (we say that it is of
  \tmtextit{rapid decay}). Combine this with the fact that $\CF$ defines a
  bijection from $L^2 (\mathbb{T}^d)$ to $\ell^2 (\mathbb{Z}^d)$ with
  inverse $\CF^{- 1}$ to show that $\CF^{- 1} \CF f = f$ for all $f \in \CS'$
  and $\CF \CF^{- 1} g = g$ for all $g$ of polynomial growth. Extend the
  Parseval formula
  \[ \langle f, \varphi^{\ast} \rangle_{L^2 (\mathbb{T}^d)} =
     \int_{\mathbb{T}^d} f (x) \varphi (x)^{\ast} \mathd x = \sum_k \hat{f} (k)
     \hat{\varphi} (k)^{\ast} \]
  from $f, \varphi \in L^2 (\mathbb{T}^d)$ to $f \in \CS'$ and $\varphi \in
  \CS$.
\end{exercise}

\begin{exercise}
  \label{exo:wn}Fix a complete probability space $(\Omega, \mathcal{F},
  \mathbb{P})$. On that space let $\xi$ be a spatial white noise on
  $\mathbb{T}^d$, i.e. $\xi$ is a centered Gaussian process indexed by $L^2
  (\mathbb{T}^d)$, with covariance
  \[ \mathbb{E} [\xi (f) \xi (g)] = \int_{\mathbb{T}^d} f (x) g (x) \mathd
     x. \]
  Show that there exists $\tilde{\xi}$ with $\mathbb{P} (\tilde{\xi} (f) =
  \xi (f)) = 1$ for all $f \in L^2$, such that $\tilde{\xi} (\omega) \in \CS'$
  for all $\omega \in \Omega$.
  
  \tmtextbf{Hint:} Show that $\mathbb{E} [ \sum_{k \in \mathbb{Z}^d}
  \exp (\lambda | \xi (e_k) |^2) / (1 + | k |^{d + 1}) ] < \infty$ for
  some suitable $\lambda > 0$. 
\end{exercise}

Linear maps on $\CS'$ can be defined by duality: if $A : \CS \rightarrow \CS$
is such that for all $k \in \mathbb{N}$ there exists $n \in \mathbb{N}$ and
$C > 0$ with $\sup_{| {\mu} | \leqslant k} \| \partial^{{\mu}} (A
\varphi) \|_{L^{\infty}} \leqslant C \sup_{| {\mu} | \leqslant n} \|
\partial^{{\mu}} \varphi \|_{L^{\infty}}$, then we set $\langle^t A f,
\varphi \rangle = \langle f, A \varphi \rangle$. Differential operators are
defined by $\langle \partial^{{\mu}} f, \varphi \rangle = (- 1)^{|
{\mu} |} \langle f, \partial^{{\mu}} \varphi \rangle$. If $\varphi :
\mathbb{Z}^d \rightarrow \mathbb{C}$ grows at most polynomially, then it
defines a \tmtextit{Fourier multiplier}
\[
   \varphi (\mathD) : \CS' \rightarrow \CS', \hspace{2em} \varphi (\mathD) f = \CF^{- 1} ( \varphi \CF f ).
\]
\begin{exercise}\label{exercise:convolution}
  Use the Fourier inversion formula of
  Exercise~\ref{exo:fourier inversion} to show that for $f \in \CS'$, $\varphi
  \in \CS$ and for $u, v : \mathbb{Z}^d \rightarrow \mathbb{C}$ with $u$ of
  polynomial growth and $v$ of rapid decay
  \[ \CF (f \varphi) (k) = (2 \pi)^{- d / 2} \sum_{\ell} \hat{f} (k - \ell)
     \hat{\varphi} (\ell) \hspace{1em} \tmop{and} \hspace{1em} \CF^{- 1} (u v)
     (x) = (2 \pi)^{d / 2} \langle \CF^{- 1} u, (\CF^{- 1} v) (x - \cdot)
     \rangle . \]
\end{exercise}

\subsection{The Stochastic Burgers equation}

Our aim here is to motivate the ideas at the base of the notion of energy
solutions. We will not insist on a detailed formulation of all the available
results. The reader can always refer to the original
paper~{\cite{gubinelli_regularization_2013}} for missing details. Applications
to the large scale behavior of particle systems are studied
in~{\cite{goncalves_universality_2010,goncalves_kpznew_2014}}.

\

We will study  the case of the stochastic Burgers equation on the
torus $\mathbb{T}$. The solution of the stochastic Burgers equation is the derivative of the solution of the Kardar--Parisi--Zhang equation, a universal model for the fluctuations in random interface growth which has been at the center of several spectacular results of the past years. Excellent surveys on the KPZ equation and related areas are~\cite{CorwinSurvey2012, QuastelSurvey2012, Quastel2015}.

 The unknown $u :
\mathbb{R}_+ \times \mathbb{T} \rightarrow \mathbb{R}$ should satisfy
\[ \partial_t u = \Delta u + \partial_x u^2 + \partial_x \xi, \]
where $\xi : \mathbb{R}_+ \times \mathbb{T} \rightarrow \mathbb{R}$ is a
space--time white noise defined on a given probability space $(\Omega,
\mathcal{F}, \mathbb{P})$ fixed once and for all. That is, $\xi$ is a
centered Gaussian process indexed by $L^2 (\mathbb{R}_+ \times \mathbb{T})$
with covariance
\[ \mathbb{E} [\xi (f) \xi (g)] = \int_{\mathbb{R}_+ \times \mathbb{T}} f
   (t, x) g (t, x) \mathd t \mathd x. \]
The equation has to be understood as a relation for processes which are
distributions in space with sufficiently regular time dependence. In particular, if we test the above
relation with $\varphi \in \CS \assign \CS (\mathbb{T}) : = C^{\infty}
(\mathbb{T})$, denote with $u_t (\varphi)$ the pairing of the distribution $u
(t, \cdot)$ with $\varphi$, and integrate in time over the interval $[0, t]$,
we formally get
\[ u_t (\varphi) = u_0 (\varphi) + \int_0^t u_s (\Delta \varphi) \mathd s -
   \int_0^t \langle u^2_s, \partial_x \varphi \rangle \mathd s - \int_0^t
   \xi_s (\partial_x \varphi) \mathd s. \]
Let us discuss the various terms in this equation. In order to make sense of
$u_t (\varphi)$ and $\int_0^t u_s (\Delta \varphi) \mathd s$, it is enough to
assume that for all $\varphi \in \CS$ the mapping $(t, \omega) \mapsto u_t
(\varphi) (\omega)$ is a stochastic process with continuous trajectories.
Next, if we denote $M_t (\varphi) = -\int_0^t \xi_s (\partial_x \varphi) \mathd
s$ then, at least by a formal computation, we have that $(M_t (\varphi))_{t
\geqslant 0, \varphi \in \CS}$ is a Gaussian random field with covariance
\[ \mathbb{E} [M_t (\varphi) M_s (\psi)] = (t \wedge s) \langle \partial_x
   \varphi, \partial_x \psi \rangle_{L^2 (\mathbb{T})} . \]
In particular, for every $\varphi \in \CS$ the stochastic process $(M_t
(\varphi))_{t \geqslant 0}$ is a Brownian motion with covariance
\[ \| \varphi \|_{H^1 (\mathbb{T})}^2 \assign \langle \partial_x \varphi,
   \partial_x \varphi \rangle_{L^2 (\mathbb{T})} . \]
   We will use this fact to have a rigorous interpretation of the white noise $\xi$ appearing in the equation.
Here we used the notation $M$ in order to stress the fact that $M_t (\varphi)$
is a martingale in its natural filtration and more generally in the filtration
$\mathcal{F}_t = \sigma (M_s (\varphi) : s \leqslant t, \varphi \in H^1
(\mathbb{T}))$, $t \geqslant 0$.

\

The most difficult term is of course the nonlinear one: $\int_0^t \langle
u^2_s, \partial_x \varphi \rangle \mathd s$. In order to define it, we need to
square the distribution $u_t$, an operation which in general can be quite
dangerous. A natural approach would be to define it as the limit of some
regularizations. For example, if $\rho : \mathbb{R} \rightarrow
\mathbb{R}_+$ is a compactly supported $C^{\infty}$ function such that
$\int_{\mathbb{R}} \rho (x) \mathd x = 1$, and we set $\rho_{\varepsilon}
(\cdot) = \rho (\cdot / \varepsilon) / \varepsilon$, then we can set
$\mathcal{N}_{t, \varepsilon} (u) (x) = \int_0^t ((\rho_{\varepsilon} \ast
u_s) (x))^2 \mathd s$ \ and define $\mathcal{N}_t (u) = \lim_{\varepsilon
\rightarrow 0} \mathcal{N}_{t, \varepsilon} (u)$ whenever the limit exists in
$\CS' \assign \CS' (\mathbb{T})$, the space of distributions on
$\mathbb{T}$. Then the question arises which properties $u$ should have for
this convergence to occur.

\subsection{The Ornstein--Uhlenbeck process}

\

Let us simplify the problem and start by studying the linearized equation
obtained by neglecting the non--linear term. Let $X$ be a solution to
\begin{equation}
  X_t (\varphi) = X_0 (\varphi) + \int_0^t X_s (\Delta \varphi) \mathd s + M_t
  (\varphi) \label{eq:ou-weak}
\end{equation}
for all $t \geqslant 0$ and $\varphi \in \CS$. This equation has \tmtextit{at
most} one solution (for fixed $X_0$). Indeed, the difference $D$ between two
solutions should satisfy $D_t (\varphi) = \int_0^t D_s (\Delta \varphi) \mathd
s$, which means that $D$ is a distributional solution to the heat equation.
Taking $\varphi (x) = e_k (x)$, where
\[ e_k (x) \assign \exp (- i kx) / \sqrt{2 \pi}, \hspace{2em} k \in
   \mathbb{Z}, \]
we get $D_t (e_k) = - k^2 \int_0^t D_s (e_k) \mathd s$ and then by Gronwall's
inequality $D_t (e_k) = 0$ for all $t \geqslant 0$. This easily implies that
$D_t = 0$ in $\CS'$ for all $t \geqslant 0$.

To obtain the \tmtextit{existence} of a solution, observe that
\[ X_t (e_k) = X_0 (e_k) - k^2 \int_0^t X_s (e_k) \mathd s + M_t (e_k) \]
and that $M_t (e_0) = 0$, while for all $k \neq 0$ the process $\beta_t (k) =
M_t (e_k) / (- i k)$ is a complex valued Brownian motion (i.e. real and
imaginary part are independent Brownian motions with the same variance). The
covariance of $\beta$ is given by
\[ \mathbb{E} [\beta_t (k) \beta_s (m)] = (t \wedge s) \delta_{k + m = 0} \]
and moreover $\beta_t (k)^{\ast} = \beta_t (- k)$ for all $k \neq 0$ (where
$\cdummy^{\ast}$ denotes complex conjugation), as well as $\beta_t (0) = 0$.
In other words, $(X_t (e_k))$ is a complex--valued Ornstein--Uhlenbeck process~(\cite{karazas}, Example 5.6.8)
which solves a linear one--dimensional SDE and has an explicit representation
given by the variation of constants formula
\[ X_t (e_k) = e^{- k^2 t} X_0 (e_k) - i k \int_0^t e^{- k^2 (t - s)} \mathd_s
   \beta_s (k) . \]
This is enough to determine $X_t (\varphi)$ for all $t \geqslant 0$ and
$\varphi \in \CS$.

\begin{exercise}
  Show that $(X_t (e_k) : t \in \mathbb{R}_+, k \in \mathbb{Z})$ is a
  complex Gaussian random field, that is for all $n \in \mathbb{N}$, for all
  $t_1, \ldots, t_n \in \mathbb{R}_+$, $k_1, \ldots, k_n \in \mathbb{Z}$,
  the vector
  \[ (\tmop{Re} (X_{t_1} (k_1)), \ldots, \tmop{Re} (X_{t_n} (k_n)), \tmop{Im}
     (X_{t_1} (k_1)), \ldots, \tmop{Im} (X_{t_n} (k_n))) \]
  is multivariate Gaussian. Show that $X$ has mean $\mathbb{E} [X_t (e_k)] =
  e^{- k^2 t} X_0 (e_k)$ and covariance
  \[ \mathbb{E} [(X_t (e_k) -\mathbb{E} [X_t (e_k)]) (X_s (e_m) -\mathbb{E}
     [X_s (e_m)])] = k^2 \delta_{k + m = 0} \int_0^{t \wedge s} e^{- k^2 (t -
     r) - k^2 (s - r)} \mathd r \]
  as well as
  \[ \mathbb{E} [(X_t (e_k) -\mathbb{E} [X_t (e_k)]) (X_s (e_m) -\mathbb{E}
     [X_s (e_m)])^{\ast}] = k^2 \delta_{k = m} \int_0^{t \wedge s} e^{- k^2 (t
     - r) - k^2 (s - r)} \mathd r. \]
  In particular,
  \[ \mathbb{E} [| X_t (e_k) -\mathbb{E} [X_t (e_k)] |^2] = \frac{1 - e^{- 2
     k^2 t}}{2} . \]
\end{exercise}

\

Next we examine the Sobolev regularity of $X$. For this purpose, we need the
following definition.

\begin{definition}
  Let $\alpha \in \mathbb{R}$. Then the Sobolev space $H^{\alpha}$ is defined
  as
  \[ H^{\alpha} \assign H^{\alpha} (\mathbb{T}) \assign \left\{ \rho \in \CS'
     : \| \rho \|^2_{H^{\alpha}} \assign \sum_{k \in \mathbb{Z}} (1 + | k
     |^2)^{\alpha} | \rho (e_k) |^2 < \infty \right\} . \]
  We also write $C H^{\alpha}$ for the space of continuous functions from
  $\mathbb{R}_+$ to $H^{\alpha}$.
\end{definition}

\begin{lemma}\label{lemma:ou-regularity}
Let $\gamma \leqslant - 1 / 2$ and assume that
  $X_0 \in H^{\gamma}$. Then almost surely $X \in C H^{\gamma -}$.
\end{lemma}

\begin{proof}
  Let $\alpha = \gamma - \varepsilon$ and consider
  \[ \| X_t - X_s \|_{H^{\alpha}}^2 = \sum_{k \in \mathbb{Z}} (1 + | k
     |^2)^{\alpha} | X_t (e_k) - X_s (e_k) |^2 . \]
  Let us estimate the $L^{2 p} (\Omega)$ norm of this quantity for $p \in
  \mathbb{N}$ by writing
  \[ \mathbb{E} \| X_t - X_s \|_{H^{\alpha}}^{2 p} = \sum_{k_1, \ldots, k_p
     \in \mathbb{Z}} \prod_{i = 1}^p (1 + | k_i |^2)^{\alpha} \mathbb{E}
     \prod_{i = 1}^p | X_t (e_{k_i}) - X_s (e_{k_i}) |^2 . \]
  By H\"older inequality, we get
  \[ \mathbb{E} \| X_t - X_s \|_{H^{\alpha}}^{2 p} \lesssim \sum_{k_1,
     \ldots, k_p \in \mathbb{Z}} \prod_{i = 1}^p (1 + | k_i |^2)^{\alpha}
     \prod_{i = 1}^p (\mathbb{E} | X_t (e_{k_i}) - X_s (e_{k_i}) |^{2 p})^{1
     / p} . \]
  Note now that $X_t (e_{k_i}) - X_s (e_{k_i})$ is a Gaussian random variable,
  so that there exists a universal constant $C_p$ for which
  \[ \mathbb{E} | X_t (e_{k_i}) - X_s (e_{k_i}) |^{2 p} \leqslant C_p
     (\mathbb{E} | X_t (e_{k_i}) - X_s (e_{k_i}) |^2)^p . \]
  Moreover,
  \[ X_t (e_k) - X_s (e_k) = (e^{- k^2 (t - s)} - 1) X_s (e_k) + ik \int_s^t
     e^{- k^2 (t - r)} \mathd_r \beta_r (k), \]
  leading to
  \begin{align*}
     &\mathbb{E} | X_t (e_k) - X_s (e_k) |^2 = (e^{- k^2 (t - s)} - 1)^2 \mathbb{E} | X_s (e_k) |^2 + k^2 \int_s^t e^{- 2 k^2 (t - r)} \mathd r \\
     &\hspace{40pt} = (e^{- k^2 (t - s)} - 1)^2 e^{- 2 k^2 s} | X_0 (e_k) |^2 + (e^{- k^2 (t - s)} - 1)^2 k^2 \int_0^s e^{- 2 k^2 (s - r)} \mathd r  \\
     &\hspace{40pt} \qquad + k^2 \int_s^t e^{- 2 k^2 (t - r)} \mathd r \\
     &\hspace{40pt} = (e^{- k^2 t} - e^{- k^2 s})^2 | X_0 (e_k) |^2 + \frac{1}{2} (e^{- k^2 (t - s)} - 1)^2 (1 - e^{- 2 k^2 s}) + \frac{1}{2} (1 - e^{- 2 k^2 (t - s)}).
  \end{align*}
  For any $\kappa \in [0, 1]$ and $k \neq 0$, we thus have
  \[ \mathbb{E} | X_t (e_k) - X_s (e_k) |^2 \lesssim (k^2 (t - s))^{\kappa}
     (| X_0 (e_k) |^2 + 1), \]
  while for $k = 0$ we have $\mathbb{E} | X_t (e_0) - X_s (e_0) |^2 = 0$. Let us introduce the notation $\mathbb{Z}_0 =\mathbb{Z} \setminus \{ 0 \}$.
  Therefore,
  \begin{align*}
     \mathbb{E} \| X_t - X_s \|_{H^{\alpha}}^{2 p} & \lesssim \sum_{k_1, \ldots, k_p \in \mathbb{Z}_0} \prod_{i = 1}^p (1 + | k_i |^2)^{\alpha} \prod_{i = 1}^p \mathbb{E} | X_t (e_{k_i}) - X_s (e_{k_i}) |^2 \\
     & \lesssim (t - s)^{\kappa p} \sum_{k_1, \ldots, k_p \in \mathbb{Z}_0}  \prod_{i = 1}^p (1 + | k_i |^2)^{\alpha} (k_i^2)^{\kappa} (| X_0 (e_{k_i}) |^2 + 1) \\
     & \lesssim (t - s)^{\kappa p} \Big[ \sum_{k \in \mathbb{Z}_0} (1 + | k
     |^2)^{\alpha} (k^2)^{\kappa} (| X_0 (e_k) |^2 + 1) \Big]^p \\
     & \lesssim (t - s)^{\kappa p} \Big(\| X_0 \|^{2 p}_{H^{\alpha + \kappa} (\mathbb{T})} + \Big[\sum_{k \in \mathbb{Z}_0} (1 + | k |^2)^{\alpha} (k^2)^{\kappa}\Big]^p\Big) .
  \end{align*}
  If $\alpha < - 1 / 2 - \kappa$, the sum on the right hand side is finite and
  we obtain an estimation for the \ modulus of continuity of $t \mapsto X_t$
  in $L^{2 p} (\Omega ; H^{\alpha})$:
  \[ \mathbb{E} \| X_t - X_s \|_{H^{\alpha}}^{2 p} \lesssim (t - s)^{\kappa
     p} [1 + \| X_0 \|^{2 p}_{H^{\alpha + \kappa}}] . \]
  Now Kolmogorov's continuity criterion allows us to conclude that almost
  surely $X \in C H^{\alpha}$ whenever $X_0 \in H^{\alpha + \kappa}$.
\end{proof}

Now note that the regularity of the Ornstein--Uhlenbeck process does not allow
us to form the quantity $X^2_t$ point--wise in time since by Fourier inversion
$X_t = \sum_k X_t (e_k) e_k^{\ast}$, and therefore we should have
\[ X_t^2 (e_k) = (2 \pi)^{- 1 / 2} \sum_{\ell + m = k} X_t (e_{\ell}) X_t
   (e_m) . \]
Of course, at the moment this expression is purely formal since we cannot
guarantee that the infinite sum converges. A reasonable thing to try is to
approximate the square by regularizing the distribution, taking the square,
and then trying to remove the regularization. Let $\Pi_N$ be the projector of
a distribution onto a finite number of Fourier modes:
\[ (\Pi_N \rho) (x) = \sum_{| k | \leqslant N} \rho (e_k) e^{\ast}_k (x) . \]
Then $\Pi_N X_t (x)$ is a smooth function of $x$ and we can consider $(\Pi_N
X_t)^2$ which satisfies
\[ (\Pi_N X_t)^2 (e_k) = (2 \pi)^{- 1 / 2} \sum_{\ell + m = k} \mathbb{I}_{|
   \ell | \leqslant N, | m | \leqslant N} X_t (e_{\ell}) X_t (e_m) . \]
We would then like to take the limit $N \rightarrow + \infty$. For
convenience, we will perform the computations below in the limit $N = +
\infty$, but one has to come back to the case of finite $N$ in order to make
it rigorous.

Then
\begin{align*}
   \mathbb{E} [X_t^2 (e_k)] & = (2 \pi)^{- 1 / 2} \delta_{k = 0} \sum_{m \in \mathbb{Z}_0} \mathbb{E} [X_t (e_{- m}) X_t (e_m)] \\
   & = (2 \pi)^{- 1 / 2} \delta_{k = 0} \sum_{m \in \mathbb{Z}_0} e^{- 2 m^2 t} | X_0 (e_m) |^2 + (2 \pi)^{- 1 / 2} \delta_{k = 0} \sum_{m \in \mathbb{Z}_0} m^2 \int_0^t e^{- 2 m^2 (t - s)} \mathd s
\end{align*}
and
\[ \sum_{m \in \mathbb{Z}_0} m^2 \int_0^t e^{- 2 m^2 (t - s)} \mathd s =
   \frac{1}{2} \sum_{m \in \mathbb{Z}_0} (1 - e^{- 2 m^2 t}) = + \infty . \]
This is not really a problem since in Burgers equation only components of $u^2_t
(e_k)$ with $k \neq 0$ appear (due to the presence of the derivative). However, $X^2_t (e_k)$ is not even a
well--defined random variable. For the remainder of this subsection let us assume that $X_0 = 0$,
which will slightly simplify the computation. If $k \neq 0$, we have
\[ \mathbb{E} [| X^2_t (e_k) |^2] =\mathbb{E} [X^2_t (e_k) X^2_t (e_{- k})]
   = (2 \pi)^{- 1} \sum_{\ell + m = k} \sum_{\ell' + m' = - k} \mathbb{E}
   [X_t (e_{\ell}) X_t (e_m) X_t (e_{\ell'}) X_t (e_{m'})] . \]
By Wick's theorem (see {\cite{Janson1997}}, Theorem~1.28), the expectation can
be computed in terms of the covariances of all possible pairings of the four
Gaussian random variables (3 possible combinations):
\begin{align*}
   \mathbb{E} [X_t (e_{\ell}) X_t (e_m) X_t (e_{\ell'}) X_t (e_{m'})] &=\mathbb{E} [X_t (e_{\ell}) X_t (e_m)] \mathbb{E} [X_t (e_{\ell'}) X_t(e_{m'})] \\
   &\quad + \mathbb{E} [X_t (e_{\ell}) X_t (e_{\ell'})] \mathbb{E} [X_t (e_m) X_t(e_{m'})] \\
   &\quad +\mathbb{E} [X_t (e_{\ell}) X_t (e_{m'})] \mathbb{E} [X_t (e_m) X_t (e_{\ell'})].
\end{align*}
Since $k \neq 0$, we have $\ell + m \neq 0$ and $\ell' + m' \neq 0$ which
allows us to neglect the first term since it is zero. By symmetry of the
summations, the two other give the same contribution and we remain with
\begin{align}\label{eq:OU square variance}
    \mathbb{E} [| X^2_t (e_k) |^2] & = \frac{1}{\pi} \sum_{\ell + m = k} \sum_{\ell' + m' = -k} \mathbb{E} [X_t (e_{\ell}) X_t (e_{\ell'})] \mathbb{E} [X_t (e_m) X_t (e_{m'})] \\ \nonumber
    & =  \frac{1}{\pi} \sum_{\ell + m = k} \mathbb{E} [X_t (e_{\ell}) X_t (e_{- \ell})] \mathbb{E} [X_t (e_m) X_t (e_{- m})] \\ \nonumber
    & = \frac{1}{4\pi} \sum_{\ell + m = k} (1 - e^{- 2 \ell^2 t}) (1 - e^{- 2 m^2 t}) = + \infty.
\end{align}
This shows that even when tested against smooth test functions, $X_t^2$ is not
in $L^2 (\Omega)$. This indicates that there are problems with $X_t^2$ and
indeed one can show that $X_t^2 (e_k)$ does not make sense as a random
variable.

\

To understand this better, observe that the Ornstein--Uhlenbeck process can
be decomposed as
\[ X_t (e_k) = ik \int_{- \infty}^t e^{- k^2 (t - s)} \mathd \beta_s (k) - ik
   e^{- k^2 t} \int_{- \infty}^0 e^{k^2 s} \mathd \beta_s (k), \]
where we extended the Brownian motions $(\beta_s (k))_{s \geqslant 0}$ to two
sided complex Brownian motions by considering independent copies. The interest in this decomposition is in the fact
that it is not
difficult to show that the second term gives rise to a smooth function if $t >
0$, so all the irregularity of $X_t$ is described by the first term which we
call $Y_t (e_k)$ and which is stationary in time. Note that $Y_t (e_k) \sim \mathcal{N}_{\mathbb{C}} (0, 1 /
2)$ for all $k \in \mathbb{Z}_0$ and $t \in \mathbb{R}$, where we write
\[ U \sim \mathcal{N}_{\mathbb{C}} (0, \sigma^2) \]
if $U = V + i W$, where $V$ and $W$ are independent random variables with
distribution $\mathcal{N} (0, \sigma^2 / 2)$. The random distribution $Y_t$
then satisfies $Y_t (\varphi) \sim \mathcal{N} (0, \| \varphi \|_{L^2
(\mathbb{T})}^2 / 2)$, and moreover it is ($1 / \sqrt{2}$ times) the
white noise on $\mathbb{T}$. It is also possible to deduce that the white
noise on $\mathbb{T}$ is indeed the invariant measure of the
Ornstein--Uhlenbeck process, that it is the only one, and that it is
approached quite fast~\cite{karazas}.

\

So we should expect that, at fixed time, the regularity of the
Ornstein--Uhlenbeck process is like that of the space white noise and this is
a way of understanding our difficulties in defining $X^2_t$ since this will
be, modulo smooth terms, the square of the space white noise.

\

A different matter is to make sense of the time--integral of $\partial_x
X^2_t$. Let us give it a name and call it $J_t (\varphi) = \int_0^t \partial_x
X_s^2 (\varphi) \mathd s$. For $J_t (e_k)$, the computation of its variance
gives a quite different result.

\begin{lemma}
  \label{lem:OU square}Almost surely, $J \in C^{1 / 2 -} H^{- 1 / 2 -}$.
\end{lemma}

\begin{proof}
  Proceeding as in~(\ref{eq:OU square variance}), we have now
  \[ \mathbb{E} [| J_t (e_k) |^2] = \frac{1}{\pi} k^2 \int_0^t \int_0^t \sum_{\ell + m =
     k} \mathbb{E} [X_s (e_{\ell}) X_{s'} (e_{- \ell})] \mathbb{E} [X_s
     (e_m) X_{s'} (e_{- m})] \mathd s \mathd s' . \]
  If $s > s'$, we have
  \[ \mathbb{E} [X_s (e_{\ell}) X_{s'} (e_{- \ell})] = \frac{1}{2} e^{-
     \ell^2 (s - s')} (1 - e^{- 2 \ell^2 s'}), \]
  and therefore
  \begin{align*}
     \mathbb{E} [| J_t (e_k) |^2] & = \frac{k^2}{4\pi} \int_0^t \int_0^t \sum_{\ell + m = k} e^{- (\ell^2 + m^2) | s - s' |} (1 - e^{- 2 \ell^2(s' \wedge s)}) (1 - e^{- 2 m^2 (s' \wedge s)}) \mathd s \mathd s' \\
      & \leqslant \frac{k^2}{4\pi} \int_0^t \int_0^t \sum_{\ell + m = k} e^{-(\ell^2 + m^2) | s - s' |} \mathd s \mathd s' \leqslant \frac{1}{2\pi} k^2 t \sum_{\ell + m = k} \int_0^{\infty} e^{- (\ell^2 + m^2) r} \mathd r  \\
      & = \frac{1}{2\pi} k^2 t \sum_{\ell + m = k} \frac{1}{\ell^2 + m^2}.
  \end{align*}
  Now for $k \neq 0$
  \[ \sum_{\ell + m = k} \frac{1}{\ell^2 + m^2} \lesssim \int_{\mathbb{R}}
     \frac{\mathd x}{x^2 + (k - x)^2} \lesssim \frac{1}{| k |} . \]
  So finally $\mathbb{E} [| J_t (e_k) |^2] \lesssim | k | t$. From which is easy to conclude that at 
  fixed $t$ the random field $J_t$ belongs almost surely to $H^{-1/2-}$. Redoing a
  similar computation in the case $J_t (e_k) - J_s (e_k)$, we obtain
  $\mathbb{E} [| J_t (e_k) - J_s (e_k) |^2] \lesssim | k | \times | t - s |$.
  To go from this estimate to a path--wise regularity result of the
  distribution $(J_t)_t$, following the line of reasoning of
  Lemma~\ref{lemma:ou-regularity}, we need to estimate the $p$-th moment of
  $J_t (e_k) - J_s (e_k)$. We already used in the proof of Lemma~\ref{lemma:ou-regularity} that all moments of a Gaussian random variable are comparable. By Gaussian hypercontractivity (see Theorem~3.50
  of~{\cite{Janson1997}}) this also holds for polynomials of Gaussian random variables, so that
  \[ \mathbb{E} [| J_t (e_k) - J_s (e_k) |^{2 p}] \lesssim_p (\mathbb{E} [|
     J_t (e_k) - J_s (e_k) |^2])^p . \]
  From here we easily derive that almost surely $J \in C^{1 / 2 -} H^{- 1 / 2
  -}$ which is the space of $1/2-$-Holder continuous functions with values in $ H^{- 1 / 2
  -}$.
\end{proof}

This shows that $\partial_x X^2_t$ exists as a space--time distribution but
not as a continuous function of time with values in distributions in space.
The key point in the proof of Lemma~\ref{lem:OU square} is the fact that the correlation $\mathbb{E} [X_s (e_{\ell}) X_{s'} (e_{- \ell})]$ of the
Ornstein--Uhlenbeck process decays quite rapidly in time.

\

The construction of the process $J$ does not solve our problem of constructing $\int_0^t \partial_x u_s^2 \mathd s$ since we need
similar properties for the full solution $u$ of the non--linear dynamics (or
for some approximations thereof), and all we have done so far relies on
explicit computations and the specific Gaussian features of the
Ornstein--Uhlenbeck process. But at least this give us a hint that indeed
there could exist a way of making sense of the term $\partial_x u (t, x)^2$,
even if only as a space--time distribution, and that in doing so we should
exploit some decorrelation properties of the dynamics.

\

So when dealing with the full solution $u$, we need a replacement for the
Gaussian computations based on the explicit distribution of $X$ that we used above. This will be provided, in the current
setting, by stochastic calculus along the time direction. Indeed, note that
for each $\varphi \in \CS$ the process $(X_t (\varphi))_{t \geqslant 0}$ is a
semimartingale in the filtration $(\mathcal{F}_t)_{t \geqslant 0}$.

\

Before proceeding with these computations, we need to develop some tools to
describe the It{\^o} formula for functions of the Ornstein--Uhlenbeck process.
This will also serve us as an opportunity to set up some analysis on
Gaussian spaces.

\subsection{Gaussian computations}

For cylindrical functions $F : \CS' \rightarrow \mathbb{R}$ of the form $F
(\rho) = f (\rho (\varphi_1), \ldots, \rho (\varphi_n))$ with $\varphi_1,
\ldots, \varphi_n \in \CS$ and $f : \mathbb{R}^n \rightarrow \mathbb{R}$ at
least $C^2_b$, we have by It{\^o}'s formula
\[ \mathd_t F (X_t) = \sum_{i = 1}^n F_i (X_t) \mathd X_t (\varphi_i) +
   \frac{1}{2} \sum_{i, j = 1}^n F_{i, j} (X_t) \mathd \langle X (\varphi_i),
   X (\varphi_j) \rangle_t, \]
where $\langle \rangle_t$ denotes the quadratic covariation of two continuous
semimartingales and where $F_i (\rho) = \partial_i f (\rho (\varphi_1),
\ldots, \rho (\varphi_n))$ and $F_{i, j} (\rho) = \partial_{i, j}^2 f (\rho
(\varphi_1), \ldots, \rho (\varphi_n))$, with $\partial_i$ denoting the
derivative with respect to the $i$-th argument. Now recall that $\mathd X_t(\varphi_i) = X_t (\Delta \varphi_i) \mathd t + \mathd M_t(\varphi_i)$ is a continuous semimartingale, and therefore
\[ \mathd \langle X (\varphi_i), X (\varphi_j) \rangle_t = \mathd \langle M
   (\varphi_i), M (\varphi_j) \rangle_t = \langle \partial_x \varphi_i,
   \partial_x \varphi_j \rangle_{L^2 (\mathbb{T})} \mathd t, \]
and then
\[ \mathd_t F (X_t) = \sum_{i = 1}^n F_i (X_t) \mathd M_t (\varphi_i) + L_0 F
   (X_t) \mathd t, \]
where $L_0$ is the second--order differential operator defined on cylindrical
functions $F$ as
\begin{equation}
  L_0 F (\rho) = \sum_{i = 1}^n F_i (\rho) \rho (\Delta \varphi_i) + \sum_{i,
  j = 1}^n \frac{1}{2} F_{i, j} (\rho) \langle \partial_x \varphi_i,
  \partial_x \varphi_j \rangle_{L^2 (\mathbb{T})} .
\end{equation}
Another way to describe the generator $L_0$ is to give its value on the
functions $\rho \mapsto \exp (\rho (\psi))$ for $\psi \in \CS$, which is
\[ L_0 e^{\rho (\psi)} = e^{\rho (\psi)} (\rho (\Delta \psi) - \frac{1}{2}
   \langle \psi, \Delta \psi \rangle_{L^2 (\mathbb{T})}) . \]
If $F, G$ are two cylindrical functions (which we can take of the form $F
(\rho) = f (\rho (\varphi_1), \ldots, \rho (\varphi_n))$ and $G (\rho) = g
(\rho (\varphi_1), \ldots, \rho (\varphi_n))$ for the same $\varphi_1, \ldots,
\varphi_n \in \CS$), we can check that
\begin{equation}
  \label{eq:L0 and E} L_0 (F G) = (L_0 F) G + F (L_0 G) +\mathcal{E} (F, G),
\end{equation}
where the quadratic form $\mathcal{E}$ is given by
\begin{equation}
  \mathcal{E} (F, G) (\rho) = \sum_{i, j} F_i (\rho) G_j (\rho) \langle
  \partial_x \varphi_i, \partial_x \varphi_j \rangle_{L^2 (\mathbb{T})} .
\end{equation}
In particular, the quadratic variation of the martingale obtained in the It\^o formula for $F$ is given by
\[
   \mathd \Big\langle \int_0^\cdot \sum_{i = 1}^n F_i (X_s) \mathd M_s (\varphi_i) \Big\rangle_t = \mathcal{E} (F, F) (X_t) \mathd t.
\]
\begin{lemma}
  (Gaussian integration by parts) Let $(Z_i)_{i = 1, \ldots, M}$ be an
  $M$-dimensional Gaussian vector with zero mean and covariance $(C_{i,
  j})_{i, j = 1, \ldots, M}$. Then for all $g \in C^1_b (\mathbb{R}^M)$ we
  have
  \[ \mathbb{E} [Z_k g (Z)] = \sum_{\ell} C_{k, \ell} \mathbb{E} \left[
     \frac{\partial g (Z)}{\partial Z_{\ell}} \right] . \]
\end{lemma}

\begin{proof}
  Use that $\mathbb{E} [e^{i \langle Z, \lambda \rangle}] = e^{- \langle
  \lambda, C \lambda \rangle / 2}$ and moreover that
  \[ \mathbb{E} [Z_k e^{i \langle Z, \lambda \rangle}] = (- i)
     \frac{\partial}{\partial \lambda_k} \mathbb{E} [e^{i \langle Z, \lambda
     \rangle}] = (- i) \frac{\partial}{\partial \lambda_k} e^{- \langle
     \lambda, C \lambda \rangle / 2} = i (C \lambda)_k e^{- \langle \lambda, C
     \lambda \rangle / 2} \]
  \[ = i \sum_{\ell} C_{k, \ell} \lambda_{\ell} \mathbb{E} [e^{i \langle Z,
     \lambda \rangle}] = \sum_{\ell} C_{k, \ell} \mathbb{E}
     [\frac{\partial}{\partial Z_{\ell}} e^{i \langle Z, \lambda \rangle}] .
  \]
  The relation is true for trigonometric functions and taking Fourier
  transforms we see that it holds for all $g \in \CS$. Is then a matter of
  taking limits to show that we can extend it to any $g \in C^1_b
  (\mathbb{R}^M)$.
\end{proof}

As a first application of this formula let us show that $\mathbb{E} [L_0 F
(\eta)] = 0$ for every cylindrical function, where $\eta$ is a space white
noise with mean zero, i.e. $\eta (\varphi) \sim \mathcal{N} (0, \| \varphi \|_{L^2 (\mathbb{T})}^2 / 2)$ for all $\varphi \in L^2_0(\mathbb{T})$, and $\eta(1) = 0$. Here we write $L^2_0(\mathbb{T)}$ for the subspace of all $\varphi \in L^2(\mathbb{T})$ with $\int_{\mathbb{T}} \varphi \mathd x = 0$. Indeed, note that by polarization $\mathbb{E} [\eta (\varphi_i) \eta (\Delta
\varphi_j)] = \frac{1}{2} \langle \varphi_i, \Delta \varphi_j \rangle_{L^2
(\mathbb{T})}$, leading to
\begin{align*}
   \mathbb{E} \sum_{i, j = 1}^n \frac{1}{2} F_{i, j} (\eta) \langle \partial_x \varphi_i, \partial_x \varphi_j \rangle_{L^2 (\mathbb{T})} &= -\mathbb{E} \sum_{i, j = 1}^n \frac{1}{2} F_{i, j} (\eta) \langle \varphi_i, \Delta \varphi_j \rangle_{L^2 (\mathbb{T})} \\
   & = - \frac{1}{2} \sum_{i, j = 1}^n \langle \varphi_i, \Delta \varphi_j \rangle_{L^2 (\mathbb{T})} \mathbb{E} \frac{\partial}{\partial \eta (\varphi_i)} F_j (\eta) \\
   & = - \sum_{j = 1}^n \mathbb{E} [\eta (\Delta \varphi_j) F_j (\eta)],
\end{align*}
so that $\mathbb{E} [L_0 F (\eta)] = 0$ (here we interpreted $\partial_j f$ as a function of $n+1$ variables, with trivial dependence on the $(n+1)$-th one). In combination with It{\^o}'s
formula, this indicates that the white noise law should indeed be a stationary
distribution for $X$ (onvince yourself of it! ). From now on we fix the initial distribution
$X_0 \sim \eta$, which means that $X_t \sim \eta$ for all $t \geqslant 0$.

As another application of the Gaussian integration by parts formula, we get
\begin{align*}
  \frac{1}{2} \mathbb{E} [\mathcal{E} (F, G) (\eta)] & = - \frac{1}{2} \sum_{i, j} \mathbb{E} [F_i (\eta) G_j (\eta)] \langle \varphi_i, \Delta \varphi_j \rangle_{L^2 (\mathbb{T})} . \\
   & = - \frac{1}{2} \sum_{i, j} \mathbb{E} [(F (\eta) G_j (\eta))_i] \langle \varphi_i, \Delta \varphi_j \rangle_{L^2 (\mathbb{T})} \\
   &\qquad + \frac{1}{2}  \sum_{i, j} \mathbb{E} [F (\eta) G_{i j} (\eta)] \langle \varphi_i, \Delta \varphi_j \rangle_{L^2 (\mathbb{T})} \\
   & = - \sum_j \mathbb{E} [F (\eta) G_j (\eta) \eta (\Delta \varphi_j)] + \frac{1}{2} \sum_{i, j} \mathbb{E} [F (\eta) G_{i j} (\eta)] \langle \varphi_i, \Delta \varphi_j \rangle_{L^2 (\mathbb{T})} \\
   & = -\mathbb{E} [(FL_0 G) (\eta)].
\end{align*}
Combining this with~(\ref{eq:L0 and E}) and with $\mathbb{E} [L_0 (F G)
(\eta)] = 0$, we obtain $\mathbb{E} [(FL_0 G) (\eta)] =\mathbb{E} [(GL_0 F)
(\eta)]$. That is, $L_0$ is a symmetric operator with respect to the law of
$\eta$.

\

Consider now the operator $\mathD$, defined on cylindrical functions $F$ by
\begin{equation}
  \mathD F (\rho) = \sum_i F_i (\rho) \varphi_i
\end{equation}
so that $\mathD F$ takes values in $\CS'$, the continuous linear functionals
on $\CS$.

\begin{exercise}
  Show that $\mathD$ is independent of the specific representation of $F$,
  that is if
  \[ F (\rho) = f (\rho (\varphi_1), \ldots, \rho (\varphi_n)) = g (\rho
     (\psi_1), \ldots, \rho (\psi_m)) \]
  for all $\rho \in \CS'$, then
  \[ \sum_i \partial_i f (\rho (\varphi_1), \ldots, \rho (\varphi_n))
     \varphi_i = \sum_j \partial_j g (\rho (\psi_1), \ldots, \rho (\psi_m))
     \psi_m . \]
  \tmtextbf{Hint:} One possible strategy is to show that for all $\theta \in
  \CS$,
  \[ \langle \mathD F (\rho), \theta \rangle = \frac{\mathd}{\mathd
     \varepsilon} F (\rho + \varepsilon \theta) |_{\varepsilon = 0} . \]
\end{exercise}

\

By Gaussian integration by parts we get
\[ \mathbb{E} [F (\eta) \langle \psi, \mathD G (\eta) \rangle] +\mathbb{E}
   [G (\eta) \langle \psi, \mathD F (\eta) \rangle] = \sum_i \mathbb{E} [(F
   G)_i (\eta) \langle \psi, \varphi_i \rangle] = 2\mathbb{E} [\eta (\psi)
   (FG) (\eta)], \]
and therefore
\[
   \mathbb{E} [F (\eta) \langle \psi, \mathD G (\eta) \rangle] =\mathbb{E}[G (\eta) \langle \psi, - \mathD F (\eta) + 2 \rho F (\eta) \rangle].
\]
So if we consider the space $L^2 (\tmop{law} (\eta))$ with inner product $\mathbb{E}[F(\eta) G(\eta)]$, then the
adjoint of $\mathD$ is given by $\mathD^{\ast} F (\rho) = - \mathD F (\rho) + 2 \rho F (\rho)$. Let
$\mathD_{\psi} F(\rho) = \langle \psi, \mathD F(\rho) \rangle$ and similarly for
$\mathD^{\ast}_{\psi} F(\rho) = - \mathD_{\psi} F(\rho) + 2 \rho (\psi) F(\rho)$.

\begin{exercise}
  Let $(e_n)_{n \geqslant 1}$ be an orthonormal basis of $L^2 (\mathbb{T})$.
  Show that
  \[ L_0 = \frac{1}{2} \sum_n \mathD^{\ast}_{e_n} \mathD_{\Delta e_n} . \]
\end{exercise}

\

Recall that the commutator between two operators $A$ and $B$ is defined as $[A,B] \assign AB - BA$. In our case we have
\[ [\mathD_{\theta}, \mathD^{\ast}_{\psi}] F (\rho) = (\mathD_{\theta}
   \mathD^{\ast}_{\psi} - \mathD^{\ast}_{\psi} \mathD_{\theta}) F (\rho) = 2
   \langle \psi, \theta \rangle_{L^2 (\mathbb{T})} F (\rho), \]
whereas $[\mathD^{\ast}_{\theta}, \mathD^{\ast}_{\psi}] = 0$. Therefore,
\[ [L_0, \mathD^{\ast}_{\psi}] = \frac{1}{2} \sum_n [\mathD^{\ast}_{e_n}
   \mathD_{\Delta e_n}, \mathD^{\ast}_{\psi}] = \frac{1}{2} \sum_n
   \mathD^{\ast}_{e_n} [\mathD_{\Delta e_n}, \mathD^{\ast}_{\psi}] +
   \frac{1}{2} \sum_n [\mathD^{\ast}_{e_n}, \mathD^{\ast}_{\psi}]
   \mathD_{\Delta e_n} \]
\[ = \sum_n \mathD^{\ast}_{e_n} \langle \psi, \Delta e_n \rangle_{L^2
   (\mathbb{T})} = \mathD^{\ast}_{\Delta \psi} . \]
So if $\psi$ is an eigenvector of $\Delta$ with eigenvalue $\lambda$, then
$[L_0, D^{\ast}_{\psi}] = \lambda \mathD^{\ast}_{\psi}$. Let now $(\psi_n)_{n
\in \mathbb{N}}$ be an orthonormal eigenbasis for $\Delta$ with eigenvalues
$\Delta \psi_n = \lambda_n \psi_n$ and consider the functions
\[ H (\psi_{i_1}, \ldots, \psi_{i_n}) : \CS' \rightarrow \mathbb{R},
   \hspace{2em} H (\psi_{i_1}, \ldots, \psi_{i_n}) (\rho) = (
   \mathD^{\ast}_{\psi_{i_1}} \cdots \mathD^{\ast}_{\psi_{i_n}} 1 )
   (\rho) . \]
Then
\begin{align}\label{eq:L0-eigenvalues}
   L_0 H (\psi_{i_1}, \ldots, \psi_{i_n}) & = L_0 \mathD^{\ast}_{\psi_{i_1}} \cdots \mathD^{\ast}_{\psi_{i_n}} 1 =  \mathD^{\ast}_{\psi_{i_1}} L_0 \mathD^{\ast}_{\psi_{i_2}} \cdots \mathD^{\ast}_{\psi_{i_n}} 1 + \lambda_{i_1} \mathD^{\ast}_{\psi_{i_1}} \cdots \mathD^{\ast}_{\psi_{i_n}} 1 \\ \nonumber
  & = \cdots = (\lambda_{i_1} + \cdots + \lambda_{i_n}) H (\psi_{i_1}, \ldots,
   \psi_{i_n}),
\end{align}
where we used that $L_0 1 = 0$. So these functions are eigenfunctions for
$L_0$ and the eigenvalues are all the possible combinations of $\lambda_{i_1}
+ \cdots + \lambda_{i_n}$ for $i_1, \ldots, i_n \in \mathbb{N}$. We have
immediately that for different $n$ these functions are orthogonal in $L^2
(\tmop{law} (\eta))$. They are actually orthogonal as soon as the indices $i$
differ since in that case there is an index $j$ which is in one but not in the
other and using the fact that $\mathD^{\ast}_{\psi_j}$ is adjoint to
$D_{\psi_j}$ and that $D_{\psi_j} G = 0$ if $G$ does not depend on $\psi_j$ we
get the orthogonality. The functions $H (\psi_{i_1}, \ldots, \psi_{i_n})$ are
polynomials and they are called \tmtextit{Wick polynomials}.

\begin{lemma}
  \label{lem:exp-D-star}For all $\psi \in \CS$, almost surely
  \[ (e^{\mathD^{\ast}_{\psi}} 1) (\eta) = e^{2 \eta (\psi) - \| \psi \|^2} .
  \]
\end{lemma}

\begin{proof}
  If $F$ is a cylindrical function of the form $F (\rho) = f (\rho
  (\varphi_1), \ldots, \rho (\varphi_m))$ with $f \in \CS (\mathbb{R}^m)$,
  then
  \[ \mathbb{E} [F (\eta) (e^{\mathD^{\ast}_{\psi}} 1) (\eta)] =\mathbb{E}
     [e^{\mathD_{\psi}} F (\eta)] =\mathbb{E} [F (\eta + \psi)] =\mathbb{E}
     [F (\eta) e^{2 \eta (\psi) - \| \psi \|^2}], \]
  where the second step follows from the fact that if we note $\Psi_t (\eta) =
  F (\eta + t \psi)$ (note that every $\psi \in \CS$ can be interpreted as an element of $\CS'$) we have $\partial_t \Psi_t (\eta) = \mathD_{\psi} \Psi_t
  (\eta)$ and $\Psi_0 (\eta) = F (\eta)$ so that $\Psi_t (\eta) = (e^{t
  \mathD_{\psi}} F) (\eta)$ for all $t \geqslant 0$ and in particular for $t =
  1$. The last step is simply a Gaussian change of variables. Indeed if we
  take $\varphi_1 = \psi$ and $\varphi_k \bot \psi$ for $k \geqslant 2$ we
  have
  \[ \mathbb{E} [F (\eta + \psi)] =\mathbb{E} [f (\eta (\psi) + \langle
     \psi, \psi \rangle, \eta (\varphi_2), \ldots, \eta (\varphi_m))] \]
  since $(\eta + \psi) (\varphi_k) = \eta (\varphi_k)$ for $k \geqslant 2$.
  Now observe that $\eta (\psi)$ is independent of $(\eta (\varphi_2), \ldots,
  \eta (\varphi_m))$ so that
  \[ \mathbb{E} [f (\eta (\psi) + \langle \psi, \psi \rangle, \eta
     (\varphi_2), \ldots, \eta (\varphi_m))] = \int_{\mathbb{R}} \frac{e^{-
     z^2 / \| \psi \|^2}}{\sqrt{\pi \| \psi \|^2}} \mathbb{E} [f (z + \langle
     \psi, \psi \rangle, \eta (\varphi_2), \ldots, \eta (\varphi_m))] \]
  \[ = \int_{\mathbb{R}} \frac{e^{- z^2 / \| \psi \|^2}}{\sqrt{\pi \| \psi
     \|^2}} e^{2 z - \| \psi \|^2} \mathbb{E} [f (z, \eta (\varphi_2),
     \ldots, \eta (\varphi_m))] =\mathbb{E} [F (\eta) e^{2 \eta (\psi) - \|
     \psi \|^2}] . \]
  To conclude the proof, it suffices to note that $\mathbb{E} [F (\eta) (e^{\mathD^{\ast}_{\psi}} 1) (\eta)] =\mathbb{E}
     [F (\eta) e^{2 \eta (\psi) - \| \psi \|^2}]$ implies that $(e^{\mathD^{\ast}_{\psi}} 1) (\eta) =e^{2 \eta (\psi) - \| \psi \|^2}$.
\end{proof}

\begin{theorem}
  The Wick polynomials $\{ H (\psi_{i_1}, \ldots, \psi_{i_n}) (\eta) : n
  \geqslant 0, i_1, \ldots, i_n \in \mathbb{N} \}$ form an orthogonal basis
  of $L^2 (\tmop{law} (\eta))$.
\end{theorem}

\begin{proof}
  Taking $\psi = \sum_i \sigma_i \psi_i$ in Lemma~\ref{lem:exp-D-star}, we get
  \begin{align*}
     e^{2 \sum_i \sigma_i \eta (\psi_i) - \sum_i \sigma_i^2 \| \psi_i \|^2} & = (e^{\mathD^{\ast}_{\psi}} 1) (\eta) = \sum_{n \geqslant 0} \frac{((\mathD^{\ast}_{\psi})^n 1) (\eta)}{n!} \\
     & = \sum_{n \geqslant 0} \sum_{i_1, \ldots, i_n} \frac{\sigma_{i_1} \cdots \sigma_{i_n}}{n!} H(\underbrace{_{} \psi_{i_1}, \ldots, \psi_{i_n}}_{n \tmop{times}})(\eta),
  \end{align*}
  which is enough to show that any random variable in $L^2$ can be expanded in
  a series of Wick polynomials showing that the Wick polynomials are an
  orthogonal basis of $L^2 (\tmop{law} (\eta))$ (but they are still not
  normalized). Indeed assume that $Z \in L^2 (\tmop{law} (\eta))$ but $Z \bot
  H (\psi_{i_1}, \ldots, \psi_{i_n}) (\eta)$ for all $n \geqslant 0$, $i_1,
  \ldots, i_n \in \mathbb{N}$, then
  \[ 0 = e^{\sum_i \sigma_i^2 \| \psi_i \|^2} \mathbb{E} [Z
     (e^{\mathD^{\ast}_{\psi}} 1) (\eta)] = e^{\sum_i \sigma_i^2 \| \psi_i
     \|^2} \mathbb{E} [Z e^{2 \sum_i \sigma_i \eta (\psi_i) - \sum_i
     \sigma_i^2 \| \psi_i \|^2}] =\mathbb{E} [Z e^{2 \sum_i \sigma_i \eta
     (\psi_i)}] . \]
  Since the $\sigma_i$ are arbitrary, this means that $Z$ is orthogonal to any
  polynomial in $\eta$ (consider the derivatives in $\sigma \equiv 0$) and then that it is orthogonal also to $\exp ( i
  \sum_i \sigma_i \eta (\psi_i) )$. So let $f \in \CS (\mathbb{R}^M)$
  and $\sigma_i = 0$ for $i > m$, and observe that
  \[ 0 = (2 \pi)^{- m / 2} \int \mathd \sigma_1 \cdots \mathd \sigma_m \CF f
     (\sigma_1, \ldots, \sigma_m) \mathbb{E} [Z e^{i \sum_i \sigma_i \eta
     (\psi_i)}] =\mathbb{E} [Z f (\eta (\psi_1), \ldots, \eta (\psi_M))], \]
  which means that $Z$ is orthogonal to all the random variables in $L^2$
  which are measurable with respect to the $\sigma$--field generated by $(\eta
  (\psi_n))_{n \geqslant 0}$. This implies $Z = 0$. That is, Wick polynomials
  form a basis for $L^2 (\tmop{law} (\eta))$.
\end{proof}

\begin{example}
  The first few (un--normalized) Wick polynomials are
  \[ H (\psi_i) (\rho) = \mathD^{\ast}_{\psi_i} 1(\rho) = 2 \rho (\psi_i), \]
  \[ H (\psi_i, \psi_j) (\rho) = \mathD^{\ast}_{\psi_i} \mathD^{\ast}_{\psi_j}
     1 = 2 \mathD^{\ast}_{\psi_i} \rho (\psi_j) = - 2 \delta_{i = j} + 4 \rho
     (\psi_i) \rho (\psi_j), \]
  and
  \begin{align*}
     H (\psi_i, \psi_j, \psi_k) (\rho) & = \mathD^{\ast}_{\psi_i} (- 2 \delta_{j= k} + 4 \rho (\psi_j) \rho (\psi_k)) \\
     & = - 4 \delta_{j = k} \rho (\psi_i) - 4 \delta_{i = j} \rho (\psi_k) - 4 \delta_{i = k} \rho (\psi_j) + 8 \rho (\psi_i) \rho (\psi_j) \rho (\psi_k).
  \end{align*}
\end{example}

Some other properties of Wick polynomials can be derived using the commutation
relation between $\mathD$ and $\mathD^{\ast}$. By linearity
$\mathD^{\ast}_{\varphi + \psi} = \mathD^{\ast}_{\varphi} +
\mathD^{\ast}_{\psi}$, so that using the symmetry of $H$ we get
\[ H_n (\varphi + \psi) \assign H \underbrace{(\varphi + \psi, \ldots, \varphi
   + \psi)}_n = \sum_{0 \leqslant k \leqslant n} \binom{n}{k} H
   (\underbrace{\varphi, \ldots, \varphi}_k, \underbrace{\psi, \ldots,
   \psi}_{n - k}) . \]
Then note that by Lemma~\ref{lem:exp-D-star} we have
\begin{align*}
  (e^{\mathD^{\ast}_{\varphi}} 1) (\eta) (e^{\mathD^{\ast}_{\psi}} 1) (\eta)& = e^{ 2\eta (\varphi) - \| \varphi \|^2} e^{2 \eta (\psi) - \| \psi \|^2} = e^{2 \eta (\varphi + \psi) - \| \varphi + \psi \|^2 + 2 \langle \varphi, \psi \rangle} \\
  & = (e^{\mathD^{\ast}_{\varphi + \psi}} 1) (\eta) e^{2 \langle \varphi, \psi \rangle}.
\end{align*}
Expanding the exponentials,
\[ \sum_{m, n}  \frac{H_m (\varphi)}{m!} \frac{H_n (\psi)}{n!} = \sum_{r,
   \ell} \frac{H_r (\varphi + \psi)}{r!} \frac{\left(2\langle
   \varphi, \psi \rangle \right)^{\ell}}{\ell !} = \sum_{p, q, \ell} \frac{H
   (\overbrace{\varphi, \ldots, \varphi}^p, \overbrace{\psi, \ldots,
   \psi}^q)}{p!q!} \frac{\left( 2 \langle \varphi, \psi \rangle
   \right)^{\ell}}{\ell !}, \]
and identifying the terms of the same homogeneity in $\varphi$ and $\psi$
respectively we get
\begin{equation}
  \label{eq:wick polynomial product} H_m (\varphi) H_n (\psi) = \sum_{p + \ell
  = m} \sum_{q + \ell = n} \frac{m!n!}{p!q! \ell !} H (\overbrace{\varphi,
  \ldots, \varphi}^p, \overbrace{\psi, \ldots, \psi}^q) \left( 2
  \langle \varphi, \psi \rangle \right)^{\ell} .
\end{equation}
This gives a general formula for such products. By polarization of this
multilinear form, we can also get a general formula for the products of
general Wick polynomials. Indeed taking $\varphi = \sum_{i = 1}^m \kappa_i
\varphi_i$ and $\psi = \sum_{j = 1}^n \lambda_j \psi_j$ for arbitrary real
coefficients $\kappa_1, \ldots, \kappa_m$ and $\lambda_1, \ldots, \lambda_n$,
we have
\begin{align*}
    &H_m (\sum_{i = 1}^m \kappa_i \varphi_i) H_n (\sum_{j = 1}^n \lambda_j \psi_j) \\
    &\hspace{50pt} = \sum_{i_1, \ldots, i_m} \sum_{j_1, \ldots, j_n} \kappa_{i_1} \cdots \kappa_{i_m} \lambda_{j_1} \cdots \lambda_{j_m} H (\varphi_{i_1}, \ldots, \varphi_{i_m}) H (\psi_{j_1}, \ldots, \psi_{j_n}).
\end{align*}
Deriving this with respect to all the $\kappa, \lambda$ parameters and setting
them to zero, we single out the term
\[ \sum_{\sigma \in S_m, \omega \in S_n} H (\varphi_{\sigma (1)}, \ldots,
   \varphi_{\sigma (m)}) H (\psi_{\omega (1)}, \ldots, \psi_{\omega (n)}) =
   m!n!H (\varphi_1, \ldots, \varphi_m) H (\psi_1, \ldots, \psi_n), \]
where $S_k$ denotes the symmetric group on $\{ 1, \ldots, k \}$, and where we
used the symmetry of the Wick polynomials. Doing the same also for the right
hand side of~(\ref{eq:wick polynomial product}) we get
\begin{align*}
   & H (\varphi_1, \ldots, \varphi_m) H (\psi_1, \ldots, \psi_n) \\
   &\hspace{50pt} = \sum_{p +\ell = m} \sum_{q + \ell = n} \frac{1}{p!q! \ell !} \sum_{i, j} H(\overbrace{\varphi_{i_1}, \ldots, \varphi_{i_p}}^p, \overbrace{\psi_{j_1}, \ldots, \psi_{j_q}}^q) \prod_{r = 1}^{\ell} (2 \langle \varphi_{i_{p + r}}, \psi_{j_{q + r}} \rangle),
\end{align*}
where the sum over $i, j$ runs over $i_1, \ldots, i_m$ permutation of $1,
\ldots, m$ and similarly for $j_1, \ldots, j_n$. Since $H(\varphi_{i_1}, \ldots, \varphi_{i_p}, \psi_{j_1}, \ldots, \psi_{j_q})(\eta)$ is orthogonal to $1$ whenever $p+q>0$, we obtain in particular
\[ \mathbb{E} [H (\psi_1, \ldots, \psi_n)(\eta) H (\psi_1, \ldots, \psi_n)(\eta)] =
   \frac{1}{n!} \sum_{i, j} \prod_{r = 1}^n (2 \langle \psi_{i_r}, \psi_{j_r}
   \rangle) = \sum_{\sigma \in S_n} \prod_{r = 1}^n ( 2 \langle \psi_r,
   \psi_{\sigma (r)} \rangle ). \]
In conclusion, we have shown that the family
\[
   \Big\{ \Big(\sum_{\sigma \in S_n} \prod_{r = 1}^n ( 2 \langle \psi_r,
   \psi_{\sigma (r)} \rangle )\Big)^{-1/2} H (\psi_{i_1}, \ldots, \psi_{i_n}) (\eta) : n \geqslant 0, i_1, \ldots, i_n \in \mathbb{N} \Big\}
\]
is an orthonormal basis of $L^2(\mathrm{law}(\eta))$.
\begin{remark}
  In our problem it will be convenient to take the Fourier basis as basis in
  the above computations. Let $e_k (x) = \exp (i k x) / \sqrt{2 \pi} = a_k (x)
  + ib_k (x)$ where $( \sqrt{2} a_k )_{k \in \mathbb{N}}$ and
  $( \sqrt{2} b_k )_{k \in \mathbb{N}}$ form together a real
  valued orthonormal basis for $L^2 (\mathbb{T})$. Then $\rho (e_k)^{\ast} =
  \rho (e_{- k})$ whenever $\rho$ is real valued, and we will denote $\mathD_k
  = \mathD_{e_k} = \mathD_{a_k} + i \mathD_{b_k}$ and similarly for
  $\mathD^{\ast}_k = \mathD_{a_k}^{\ast} - i \mathD_{b_k}^{\ast} = - \mathD_{-
  k} + 2 \rho (e_{- k})$. In this way, $\mathD^{\ast}_k$ is the adjoint of
  $\mathD_k$ with respect to the Hermitian scalar product on $L^2 (\Omega ;
  \mathbb{C})$ and the Ornstein--Uhlenbeck generator takes the form
  \begin{equation}
    L_0 = \sum_{k \in \mathbb{N}} (\mathD_{\partial_x a_k}^{\ast}
    \mathD_{\partial_x a_k} + \mathD_{\partial_x b_k}^{\ast}
    \mathD_{\partial_x b_k}) = \frac{1}{2} \sum_{k \in \mathbb{Z}} k^2
    \mathD_k^{\ast} \mathD_k
  \end{equation}
  (convince yourself of the last identity by observing that $D_{k}^\ast D_k + D_{-k}^\ast D_{-k} = 2(\mathD_{ a_k}^{\ast}  \mathD_{ a_k} + \mathD_{ b_k}^{\ast} \mathD_{b_k})$!). Similarly,
  \begin{equation}
    \mathcal{E} (F, G) = \sum_{k \in \mathbb{Z}} k^2 (\mathD_k F)^{\ast}
    (\mathD_k G) .
  \end{equation}
\end{remark}

\subsection{The It{\^o} trick}

We are ready now to start our computations. Recall that we want to analyse
$J_t (\varphi) = \int_0^t \partial_x X_s^2 (\varphi) \mathd s$ using It{\^o}
calculus with respect to the Ornstein--Uhlenbeck process. We want to
understand $J_t$ as a correction term in It{\^o}'s formula: If we can find
a function $G$ such that $L_0 G (X_t) = \partial_x X_t^2$, then we get from It\^o's formula
\[
   \int_0^t \partial_x X_s^2 \mathd s = G(X_t) - G(X_0) - M_{G,t},
\]
where $M_G$ is a martingale depending on $G$. Of course, $G$ will
not be a cylindrical function but we only defined $L_0$ on cylindrical
functions. So to make the following calculations rigorous we would again have
to replace $\partial_x X_t^2$ by $\partial_x \Pi_n X^2_t$ and then pass to the
limit, see the paper~{\cite{gubinelli_regularization_2013}} for details. As before we will perform the calculations already in the limit $N = + \infty$,
in order to simplify the computations and not to obscure the ideas through
technicalities. The next problem is that the pointwise evaluation $\int_0^t \partial_x X_s^2(x) \mathd s$ does not make any sense because the integral will only be defined as a space distribution. So we will consider
\[
   G \colon \CS' \to \CS'
\]
instead of $G \colon \CS' \to \mathbb{C}$. Note however that we can reduce every such $G$ to a function from $\CS'$ to $\mathbb{C}$ by considering $\rho \mapsto G(\rho)(e_k)$ for all $k$.

Now for a fixed $k$, we have
\begin{equation}\label{eq:square explicit}
   \partial_x X_t^2 (e_k) = \frac{ik}{\sqrt{2}} \sum_{\ell + m = k} X_t (e_{\ell}) X_t (e_m) = \frac{ik}{\sqrt{2}}  \sum_{\ell + m = k} H_{\ell, m} (X_t),
\end{equation}
where $H_{\ell, m} (\rho) = \frac{1}{4} (\mathD^{\ast}_{- \ell}
\mathD^{\ast}_{- m} 1) (\rho) = \rho (e_{\ell}) \rho (e_m) - \frac{1}{2}
\delta_{\ell + m = 0}$ is a second order Wick polynomial so that $L_0
H_{\ell, m} = - (\ell^2 + m^2) H_{\ell, m}$ by~(\ref{eq:L0-eigenvalues}).
Therefore, it is enough to take
\begin{equation}
  G (X_t) (e_k) = - ik \sum_{\ell + m = k} \frac{H_{\ell, m} (X_t)}{\ell^2 +
  m^2} .
\end{equation}
This corresponds to the distribution $G (X_t) (\varphi) = - \int_0^{\infty}
\partial_x (e^{s \Delta} X_t)^2 (\varphi) \mathd s$ (check it!). Then
\[ G (X_t) (\varphi) = G (X_0) (\varphi) + M_{G, t} (\varphi) + J_t (\varphi),
\]
where $M_{G, t} (\varphi)$ is a martingale with quadratic variation
\[ \mathd \langle M_{G, \ast} (\varphi), M_{G, \ast} (\varphi) \rangle_t
   =\mathcal{E} (G (\ast) (\varphi), G (\ast) (\varphi)) (X_t) \mathd t. \]
We can estimate
\[ \mathbb{E} [| J_t (\varphi) - J_s (\varphi) |^{2 p}] \lesssim_p
   \mathbb{E} [| M_{G, t} (\varphi) - M_{G, s} (\varphi) |^{2 p}]
   +\mathbb{E} [| G (X_t) (\varphi) - G (X_s) (\varphi) |^{2 p}] . \]
To bound the martingale expectation, we will use the following Burkholder
inequality:

\begin{lemma}
  Let $m$ be a continuous local martingale with $m_0 = 0$. Then for all $T
  \geqslant 0$ and $p > 1$,
  \[ \mathbb{E} [\sup_{t \leqslant T} | m_t |^{2 p}] \leqslant C_p
     \mathbb{E} [\langle m \rangle_T^p] . \]
\end{lemma}

\begin{proof}
  Start by assuming that $m$ and $\langle m \rangle$ are bounded. It{\^o}'s
  formula yields
  \[ \mathd | m_t |^{2 p} = (2 p) | m_t |^{2 p - 1} \mathd m_t + \frac{1}{2}
     (2 p) (2 p - 1) | m_t |^{2 p - 2} \mathd \langle m \rangle_t, \]
  and therefore
  \[ \mathbb{E} [| m_T |^{2 p}] = C_p \mathbb{E} \Big[ \int_0^T | m_s |^{2
     p - 2} \mathd \langle m \rangle_s \Big] \leqslant C_p \mathbb{E}
     [\sup_{t \leqslant T} | m_t |^{2 p - 2} \langle m \rangle_T] . \]
  By Cauchy--Schwartz we get
  \[ \mathbb{E} [| m_T |^{2 p}] \leqslant C_p \mathbb{E} [\sup_{t \leqslant
     T} | m_t |^{2 p}]^{(2 p - 2) / 2 p} \mathbb{E} [\langle m
     \rangle_T^p]^{1 / p} . \]
  But now Doob's $L^p$ inequality yields $\mathbb{E} [\sup_{t \leqslant T} |
  m_t |^{2 p}] \leqslant C'_p \mathbb{E} [| m_T |^{2 p}]$, and this implies
  the claim in the bounded case. The unbounded case can be treated with a
  localization argument. 
\end{proof}

Applying Burkholder's inequality, we obtain
\begin{align*}
    \mathbb{E} [| J_t (\varphi) - J_s (\varphi) |^{2 p}] & \lesssim_p  \mathbb{E} \Big[ \Big| \int_s^t \mathcal{E} (G (\ast) (\varphi), G (\ast) (\varphi)) (X_r) \mathd r \Big|^p \Big] +\mathbb{E} [| G (X_t)(\varphi) - G (X_s) (\varphi) |^{2 p}] \\
     & \leqslant (t - s)^{p - 1} \int_s^t \mathbb{E} [| \mathcal{E} (G (\ast)(\varphi), G (\ast) (\varphi)) (X_r) |^p] \mathd r \\
     &\qquad +\mathbb{E} [| G (X_t) (\varphi) - G (X_s) (\varphi) |^{2 p}] \\
     & = (t - s)^p \mathbb{E} [| \mathcal{E} (G (\ast) (\varphi), G (\ast)(\varphi)) (\eta) |^p] +\mathbb{E} [| G (X_t) (\varphi) - G (X_s) (\varphi) |^{2 p}],
\end{align*}
using that $X_r \sim \eta$. Now
\[ \mathD_m G (\rho) (e_k) = - 2 ik \frac{\rho (e_{k - m})}{(k - m)^2 + 
   m^2}, \]
and therefore
\begin{align*}
    \mathcal{E} (G (\ast) (e_k), G (\ast) (e_k)) (\rho) & = \sum_m m^2 \mathD_{- m} G (\rho) (e_{- k}) \mathD_m G (\rho) (e_k) \\
    & = 4 k^2 \sum_{\ell + m = k} m^2 \frac{| \rho (e_{\ell}) |^2}{(\ell^2 + m^2)^2} \lesssim k^2 \sum_{\ell + m = k} \frac{| \rho (e_{\ell}) |^2}{\ell^2 + m^2},
 \end{align*}
which implies that
\[ \mathbb{E} [| \mathcal{E} (G (\ast) (e_k), G (\ast) (e_k)) (\eta) |]
   \lesssim k^2 \mathbb{E} \sum_{\ell + m = k} \frac{| \eta (e_{\ell})
   |^2}{\ell^2 + m^2} \lesssim k^2 \sum_{\ell + m = k} \frac{1}{\ell^2 + m^2}
   \lesssim | k | . \]
A similar computation gives also that
\[ \mathbb{E} [| \mathcal{E} (G (\ast) (e_k), G (\ast) (e_k)) (\eta) |^p]
   \lesssim | k |^p . \]
Further, we have
\begin{align*}
   \mathbb{E} [| G (X_t) (e_k) - G (X_s) (e_k) |^2] & \lesssim k^2 \sum_{\ell + m = k} \mathbb{E} \Big[ \frac{|H_{\ell, m} (X_t) - H_{\ell, m} (X_s))^2}{(\ell^2 + m^2|^2} \Big] \\
    & \lesssim k^2 | t - s | \sum_{\ell + m = k} \frac{m^2}{(\ell^2 + m^2)^2} \lesssim | k |  | t - s |.
\end{align*}
And finally, since $G$ is a second order polynomial of a Gaussian process we can apply once more Gaussian hypercontractivity to obtain
\[ \mathbb{E} [| J_t (e_k) - J_s (e_k) |^{2 p}] \lesssim_p (t - s)^p | k |^p
   . \]

The advantage of the It{\^o} trick with respect to the explicit Gaussian
computation is that it goes over to the non--Gaussian case. Indeed note that
while the boundary term $G (X_t) (\varphi) - G (X_s) (\varphi)$ has been
estimated using a lot of the Gaussian information about $X$, we used only the
law at a fixed time to handle the term $\int_s^t \mathcal{E} (G (\ast)
(\varphi), G (\ast) (\varphi)) (X_r) \mathd r$.

\

In order to carry over these computation to the full process $u$ solution of
the non--linear dynamics we need to replace the generator of $X$ with that of
$u$ and to have a way to handle the boundary terms. The idea is now to reverse
the Markov process $u$ in time, which will allow us to kill the antisymmetric
part of the generator and at the same time kill the boundary terms. Indeed
observe that if $u$ solves the stochastic Burgers equation, then formally we
have the It{\^o} formula

\[ \mathd_t F (u_t) = \sum_{i = 1}^n F_i (u_t) \mathd M_t (\varphi_i) + L F
   (u_t) \mathd t, \]
where $L$ is now the full generator of the non--linear dynamics, given by
\[ L F (\rho) = L_0 F (\rho) + \sum_i F_i (\rho) \langle \partial_x \rho^2,
   \varphi_i \rangle = L_0 F (\rho) + B F (\rho), \]
where
\[ B F (\rho) = \sum_k (\partial_x \rho^2) (e_k) \mathD_k F (\rho) . \]
Formally, the non--linear term is antisymmetric with respect to the
invariant measure of $L_0$. Indeed since $B$ is a first order operator
\begin{equation}
\label{eq:antisymmetry of nonlinear} 
\mathbb{E} [(B F (\eta)) G (\eta)] =\mathbb{E} [(B (F G) (\eta))]
   -\mathbb{E} [F (\eta) (B G (\eta))] =  -\mathbb{E} [F (\eta) (B G (\eta))]
\end{equation}
provided $\mathbb{E} [B F (\eta)]=0$ for any cylinder function $F$. Let us show this. 
\begin{equation*}
\begin{aligned}
\mathbb{E} [B F (\eta)] & = \sum_k \mathbb{E} [(\partial_x \eta^2) (e_k)
   \mathD_k F (\eta)]\\
   &  = - \sum_k \mathbb{E} [(\mathD_k (\partial_x \eta^2) (e_k)) F (\eta)] +
   \sum_k \mathbb{E} [\mathD_k [(\partial_x \eta^2) (e_k) F (\eta)]]
\end{aligned}
\end{equation*}
But now we get from~\eqref{eq:square explicit}
\begin{equation*}
\mathD_k (\partial_x \eta^2) (e_k) = \sqrt{2} i k \eta (e_0) = \pi^{-1/2} i k \langle \eta ,1 \rangle = 0,
\end{equation*}
where we used that $\langle \eta, 1 \rangle = 0$. 
Gaussian integration by parts then formally gives
\begin{equation*}
\begin{aligned}
\mathbb{E} [B F (\eta)] &= \sum_k \mathbb{E} [\mathD_k [(\partial_x
   \eta^2) (e_k) F (\eta)]] = \sum_k \mathbb{E} [\eta (e_k) (\partial_x
   \eta^2) (e_k) F (\eta)] \\
   & = \mathbb{E} [\langle \eta, \partial_x \eta^2 \rangle F (\eta)] = 
   \frac{1}{3} \mathbb{E} [\langle 1, \partial_x \eta^3 \rangle F (\eta)] = 0
\end{aligned}
\end{equation*}
since $ \langle 1, \partial_x \eta^3 \rangle = - \langle \partial_x 1, \eta^3
\rangle = 0$ (but of course $\langle \eta, \partial_x \eta^2 \rangle$ is not well defined).

\medskip
The dynamics of $u$ backwards in time has a Markovian description which is the
subject of the next exercise.

\begin{exercise}
  Let $(y_t)_{t \geqslant 0}$ be a stationary Markov process on a Polish
  space, with semigroup $(P_t)_{t \geqslant 0}$ and stationary distribution
  ${\mu}$. Show that if $P^{\ast}_t$ is the adjoint of $P_t$ in $L^2
  ({\mu})$, then $(P^{\ast}_t)$ is a semigroup of operators on $L^2
  ({\mu})$ (that is $P^{\ast}_0 = \tmop{id}$ and $P^{\ast}_{s + t} =
  P^{\ast}_s P^{\ast}_t$ as operators on $L^2 ({\mu})$). Show that if $y_0
  \sim {\mu}$, then for all $T > 0$ the process $\hat{y}_t = y_{T - t}$,
  $t \in [0, T]$, is also Markov, with semigroup $(P^{\ast}_t)_{t \in [0,
  T]}$, and that ${\mu}$ is also an invariant distribution for
  $(P^{\ast}_t)$. Show also that if $(P_t)$ has generator $L$ then
  $(P^{\ast}_t)$ has generator $L^{\ast}$ which is the adjoint of $L$ with
  respect to $L^2 ({\mu})$.
\end{exercise}

Now if we reverse the process in time letting $\hat{u}_t = u_{T - t}$, we have
by stationarity
\[ \mathbb{E} [F (\hat{u}_t) G (\hat{u}_0)] =\mathbb{E} [F (u_{T - t}) G
   (u_T)] =\mathbb{E} [F (u_0) G (u_t)] . \]
So if we denote by $\hat{L}$ the generator of $\hat{u}$:
\[ \mathbb{E} [\hat{L} F (\hat{u}_0) G (\hat{u}_0)] = \left.
   \frac{\mathd}{\mathd t} \right|_{t = 0} \mathbb{E} [F (\hat{u}_t) G
   (\hat{u}_0)] = \left. \frac{\mathd}{\mathd t} \right|_{t = 0} \mathbb{E}
   [F (u_0) G (u_t)] =\mathbb{E} [L G (u_0) F (u_0)], \]
which means that $\hat{L}$ is the adjoint of $L$, that is
\[ \hat{L} F (\rho) = L_0 F (\rho) - B F (\rho) = L_0 F (\rho) - \sum_k
   (\partial_x \rho^2) (e_k) \mathD_k F (\rho) . \]
   In other words, the reversed process solves
\[ \hat u_t (\varphi) = \hat u_0 (\varphi) + \int_0^t \hat u_s (\Delta \varphi) \mathd s +
   \int_0^t \langle \hat u^2_s, \partial_x \varphi \rangle \mathd s - \int_0^t
   \hat \xi_s (\partial_x \varphi) \mathd s \]
   for a different space-time white noise $\hat \xi$. Then It{\^o}'s formula for $\hat u$ gives
\[ \mathd_t F (\hat{u}_t) = \sum_{i = 1}^n F_i (\hat{u}_t) \mathd \hat{M}_t
   (\varphi_i) + \hat{L} F (\hat{u}_t) \mathd t, \]
where for all test functions $\varphi$, the process $\hat{M} (\varphi)$ is a
martingale in the filtration of $\hat{u}$ with covariance
\[ \mathd \langle \hat{M} (\varphi), \hat{M} (\psi) \rangle_t = \langle
   \partial_x \varphi, \partial_x \psi \rangle_{L^2 (\mathbb{T})} \mathd t.
\]
Combining the It{\^o} formulas for $u$ and $\hat{u}$, we get
\[ F (u_T) (\varphi) = F (u_0) (\varphi) + M_{F, T} (\varphi) + \int_0^T L F
   (u_s) (\varphi) \mathd s \]
and
\begin{align*}
   F (u_0) (\varphi) & = F (\hat{u}_T) (\varphi) = F (\hat{u}_0) (\varphi) +  \hat{M}_{F, T} (\varphi) + \int_0^T \hat{L} F (\hat{u}_s) (\varphi) \mathd s \\
   & = F (u_T) (\varphi) + \hat{M}_{F, T} (\varphi) + \int_0^T \hat{L} F (u_s) (\varphi) \mathd s,
\end{align*}
and summing up these two equalities gives
\[ 0 = M_{F, T} (\varphi) + \hat{M}_{F, T} (\varphi) + \int_0^T (\hat{L} + L)
   F (u_s) (\varphi) \mathd s, \]
that is
\[ 2 \int_0^T L_0 F (u_s) (\varphi) \mathd s = - M_{F, T} (\varphi) -
   \hat{M}_{F, T} (\varphi) . \]
An added benefit of this forward--backward representation is that the only term
which required quite a lot of informations about $X$, that is the boundary
term $F (X_t) (\varphi) - F (X_s) (\varphi)$ does not appear at all now. As
above if $2 L_0 F (\rho) = \partial_x \rho^2$, we end up with
\begin{equation}
  \label{eq:martingale trick sbe}  \int_0^T \partial_x u_s^2 (\varphi) \mathd
  s = - M_{F, T} (\varphi) - \hat{M}_{F, T} (\varphi) .
\end{equation}
\

\begin{exercise}
  Perform a similar formal calculation as in~(\ref{eq:antisymmetry of
  nonlinear}) to see that $\mathbb{E} [L F (\eta)] = 0$ for all cylindrical
  functions $F$, so that $\eta$ should also be invariant for the stochastic Burgers
  equation. Combine this with~(\ref{eq:martingale trick sbe}) to show that
  setting $\mathcal{N}_t^N (\varphi) = \int_0^t \partial_x (\Pi_N u_s)^2
  (\varphi) \mathd s$ we have
  \[ \mathbb{E} [| \mathcal{N}_t^N (e_k) -\mathcal{N}_s^N (e_k) |^{2 p}]
     \lesssim_p (t - s)^p | k |^p \]
  and letting $\mathcal{N}_t^{N, M} =\mathcal{N}_t^N -\mathcal{N}_t^M$ we get
  \[ \mathbb{E} [| \mathcal{N}_t^{N, M} (e_k) -\mathcal{N}_s^{N, M} (e_k)
     |^{2 p}] \lesssim_p (| k | / N)^{\varepsilon p} (t - s)^p | k |^p \]
  for all $1 \leqslant N \leqslant M$. Use this to derive that
  \[ (\mathbb{E} [\| \mathcal{N}_t^{N, M} -\mathcal{N}_s^{N, M} \|^{2
     p}_{H^{\alpha}}])^{1 / 2 p} \lesssim_{p, \alpha} N^{- \varepsilon / 2} (t
     - s)^{1 / 2} \]
  for all $\alpha < - 1 - \varepsilon$, and realize that this estimate allows
  you to prove compactness of the approximations $\mathcal{N}^N$ and then
  convergence to a limit $\mathcal{N}$ in $L^{2 p} (\Omega ; C^{1 / 2 -} H^{-
  1 -})$.
\end{exercise}

\subsection{Controlled distributions}

Let us cook up a definition which will allow us to rigorously perform the
formal computations above in a general setting.

\begin{definition}
  Let $u, \mathcal{A}: \mathbb{R}_+ \times \mathbb{T} \rightarrow
  \CS'(\mathbb{T})$ be a couple of generalized (i.e. distribution-valued) processes such that
  \begin{enumerateroman}
    \item For all $\varphi \in \CS (\mathbb{T})$ the process $t \mapsto u_t
    (\varphi)$ is a continuous semimartingale satisfying
    \[ u_t (\varphi) = u_0 (\varphi) + \int_0^t u_s (\Delta \varphi) \mathd s
       +\mathcal{A}_t (\varphi) + M_t (\varphi), \]
    where $t \mapsto M_t (\varphi)$ is a martingale with quadratic variation
    $\langle M (\varphi), M (\psi) \rangle_t = \langle \partial_x \varphi,
    \partial_x \psi \rangle_{L^2 (\mathbb{T})} t$ and $t \mapsto \mathcal{A}_t
    (\varphi)$ is a finite variation process with $\mathcal{A}_0(\varphi) = 0$.
    
    \item For all $t \geqslant 0$ the random distribution $\varphi \mapsto u_t
    (\varphi)$ is a zero mean space white noise with variance $\| \varphi
    \|_{L^2_0}^2 / 2$.
    
    \item For any $T > 0$ the reversed process $\hat{u}_t = u_{T - t}$ has
    again properties $i, i i$ with martingale $\hat{M}$ and finite variation
    part $\hat{\mathcal{A}}$ such that $\hat{\mathcal{A}}_t (\varphi) = -
    (\mathcal{A}_T (\varphi) -\mathcal{A}_{T - t} (\varphi))$.
  \end{enumerateroman}
  Any pair of processes $(u, \mathcal{A})$ satisfying these condition will be
  called \tmtextit{controlled by the Ornstein--Uhlenbeck process} and we will
  denote the set of all such processes with $\mathcal{Q}_{\tmop{ou}}$.
\end{definition}

\begin{theorem}[\cite{gubinelli_regularization_2013}, Lemma~1]
  Assume that $(u, \mathcal{A}) \in \mathcal{Q}_{\tmop{ou}}$ and for any $N
  \geqslant 1$, $t \geqslant 0$, $\varphi \in \CS$ let
  \[ \mathcal{N}_t^N (\varphi) = \int_0^t \partial_x (\Pi_N u_s)^2 (\varphi)
     \mathd s \]
  Then for any $p \geqslant 1$ $(\mathcal{N}^N)_{N \geqslant 1}$ converges in
  $L^p(\Omega)$ to a space--time distribution $\mathcal{N} \in C^{1 / 2 -} H^{-
  1 -}$.
\end{theorem}

We are now at a point where we can give a meaning to our original equation.

\begin{definition}
  A pair of random distribution $(u, \mathcal{A}) \in \mathcal{Q}_{\tmop{ou}}$
  is an {\tmem{energy solution}} to the stochastic Burgers equation if it
  satisfies
  \[ u_t (\varphi) = u_0 (\varphi) + \int_0^t u_s (\Delta \varphi) \mathd s
     +\mathcal{N}_t (\varphi) + M_t (\varphi) \]
  for all $t \geqslant 0$ and $\varphi \in \CS$. That is if
  $\mathcal{A}=\mathcal{N}$. 
\end{definition}

Now we are in a relatively standard setting of needing to prove existence and
uniqueness of such energy solutions. Note that in general the solutions are
{\tmem{pairs}} of processes $(u, \mathcal{A})$.

\begin{remark}
  The notion of energy solution has been introduced (in a slightly different
  way) in the work of Gon{\c c}alves and
  Jara~{\cite{goncalves_universality_2010}} on macroscopic universal
  fluctuations of weakly asymmetric interacting particle systems.
\end{remark}

\subsection{Existence of solutions}

For the existence the way to proceed is quite standard. We approximate the
equation, construct approximate solutions and then try to have enough
compactness to have limiting points which then naturally will satisfy the
requirements for energy solutions. For any $N \geqslant 1$ consider solutions
$u^N$ to
\[ \partial_t u^N = \Delta u^N + \partial_x \Pi_N (\Pi_N u^N)^2 + \partial_x
   \xi \]
These are generalized functions such that
\[ \mathd u^N_t (e_k) = - k^2 u^N_t (e_k) \mathd t + [\partial_x \Pi_N (\Pi_N
   u^N)^2] (e_k) \mathd t + ik \mathd \beta_t (k) \]
for $k \in \mathbb{Z}$ and $t \geqslant 0$. We take $u_0$ to be the white
noise with covariance $u_0 (\varphi) \sim \mathcal{N} (0, \| \varphi \|^2 /
2)$. The point of our choice of the non--linearity is that this
(infinite--dimensional) system of equations decomposes into a finite
dimensional system for $(v^N (k) = \Pi_N u^N (e_k))_{k : | k | \leqslant N}$
and an infinite number of one--dimensional equations for each $u^N (e_k)$ with
$| k | > N$. Indeed if $| k | > N$ we have $[\partial_x \Pi_N (\Pi_N u^N)^2]
(e_k) = 0$ so $u_t (e_k) = X_t (e_k)$ the Ornstein--Uhlenbeck process with
initial condition $X_0(e_k) = u_0(e_k)$ which renders it stationary in time (check it).
The equation for $(v^N (k))_{| k | \leqslant N}$ reads
\[ \mathd v^N_t (k) = - k^2 v^N_t (k) \mathd t + b_k (v^N_t) \mathd t + ik
   \mathd \beta_t (k), \hspace{2em} | k | \leqslant N, t \geqslant 0 \]
where
\[ b_k (v^N_t) = ik \sum_{\ell + m = k} \mathbb{I}_{| \ell |, | k |, | m |
   \leqslant N} v^N_t (\ell) v^N_t (m) . \]
This is a standard finite--dimensional ODE having global solutions for all
initial conditions which gives rise to a nice Markov process. The fact that
solutions do not blow up even if the interaction is quadratic can be seen
by computing the evolution of the norm
\[ A_t = \sum_{| k | \leqslant N} | v^N_t (k) |^2 \]
and by showing that
\[ \mathd A_t = 2 \sum_{| k | \leqslant N} v^N_t (- k) \mathd v^N_t (k) = -
   2 \sum_{|k| \le N} k^2 |v^N_t(k)|^2 \mathd t + 2 \sum_{| k | \leqslant N} v^N_t (- k) b_k (v^N_t)
   \mathd t + 2  i  k \sum_{| k | \leqslant N} v^N_t (- k) \mathd \beta_t
   (k). \]
Since $A$ is nonnegative, we increase its absolute value by omitting the first contribution. But now
\begin{align*}
   \sum_{| k | \leqslant N} v^N_t (- k) b_k (v^N_t) & = 2  i  \sum_{k, \ell, m : \ell + m = k} \mathbb{I}_{| \ell |, | k |, | m | \leqslant N} kv^N_t (\ell) v^N_t (m) v^N_t (- k) \\
    & = - 2  i  \sum_{k, \ell, m : \ell + m + k = 0} \mathbb{I}_{| \ell |, | k |, | m | \leqslant N}  (k) v^N_t (\ell) v^N_t (m) v^N_t (k)
\end{align*}
and by symmetry of this expression it is equal to
\[ = - \frac{2}{3}  i  \sum_{k, \ell, m : \ell + m + k = 0} \mathbb{I}_{|
   \ell |, | k |, | m | \leqslant N}  (k + \ell + m) v^N_t (\ell) v^N_t (m)
   v^N_t (k) = 0, \]
so $|A_t| \le |A_0 + M_t|$ where $\mathd M_t = 2 \sum_{| k | \leqslant N}
\mathbb{I}_{| k | \leqslant N} ( i  k) v^N_t (- k) \mathd \beta_t (k)$. Now
\[ \mathbb{E} [M_T^2] \lesssim \int_0^T \sum_{| k | \leqslant N} k^2 | v^N_t
   (k) |^2 \mathd t \lesssim N^2 \int_0^T A_t \mathd t \]
and then by martingales inequalities
\begin{align*}
   \mathbb{E} [\sup_{t \in [0, T]} (A_t)^2] & \leqslant 2\mathbb{E} [A_0^2] + 2\mathbb{E} [\sup_{t \in [0, T]} (M_t)^2] \leqslant 2\mathbb{E} [A_0^2]^{} + 8\mathbb{E} [M_T^2]^{} \\
    & \leqslant 2\mathbb{E} [A_0^2]^{} + CN^2 \int_0^T \mathbb{E} (A_t) \mathd t.
\end{align*}
Now Gronwall's inequality gives
\[ \mathbb{E} [\sup_{t \in [0, T]} (A_t)^2] \lesssim e^{CN^2 T} \mathbb{E}
   [A_0^2], \]
from where we can deduce (by a continuation argument) that almost surely there
is no blowup at finite time for the dynamics. The generator $L^N$ for the Galerkin dynamics is given by
\[ L^N F (\rho) = L_0 F (\rho) + B^N F (\rho), \]
where
\[ B^N F (\rho) = \sum_k  \mathbb{I}_{ | k | \leqslant N} (\partial_x \rho^2) (e_k) \mathD_k F (\rho) . \]
And again the non--linear  drift $B^N$ is antisymmetric with respect to the
invariant measure of $L_0$ by a computation similar to that for the full drift $B$. Next, using Echeverr\'ia's criterion~\cite{Echeverria1982} we can obtain the invariance of the white noise from its infinitesimal invariance which can be checked at the level of the generator $L^N$. 
Finally it is also possible to 
rigorously show that the reversed process is a Markov process with generator
\[ \hat L^N F (\rho) = L_0 F (\rho) - B^N F (\rho), \]
thus proving that the reversed non-linear drift is the opposite of the forward one. Taking
$$
\mathcal{A}^N_t(e_k) = \int_0^t b_k(v^N_s) ds
$$
we obtain that $(v^N, \mathcal{A}^N)\in \mathcal{Q}_{\tmop{ou}}$. Note that this result depends on the fact that we kept the full linear part $L_0$ of the generator. A more standard Galerkin truncation would have lead us to a process which is controlled by the Galerkin--truncated OU process. Estimates would have resulted in a similar way but our setup is simpler. 

Given that $(v^N, \mathcal{A}^N)$ is controlled by the OU process, the It{\^o} trick applied to $\mathcal{A}^N$ provides enough compactness in order to pass to the
limit as $N\to\infty$ and build an energy solution to the Stochastic Burgers equation. See~\cite{gubinelli_regularization_2013} for additional details on the limiting procedure and \cite{Russo2001} for details on how to implement the It\^o trick on the level of diffusions.

\begin{remark}
   There is however one small catch: For a controlled distribution $(u, \mathcal{A})$ we required $\mathcal{A}(\varphi)$ to be of finite variation for every test function $\varphi$. The solution $(v^N, \mathcal{A}^N)$ to the truncated equation will satisfy this, but in the limit $\mathcal{A}(\varphi)$ will only have vanishing quadratic variation and it will not be of finite variation (in other words $u(\varphi)$ is a Dirichlet process and not a semimartingale). Luckily in this setting it is still possible to derive an It\^o formula and everything goes through as described above, see~\cite{gubinelli_regularization_2013} for details.
\end{remark}

\section{Besov spaces}\label{sec:preliminaries}

Here we collect some classical results from harmonic analysis which we will
need in the following. We concentrate on distributions and SPDEs on the torus,
but everything in this Section applies mutatis mutandis on the full space
$\mathbb{R}^d$, see~{\cite{Gubinelli2012}}. The only problem is that then the
stochastic terms will no longer be in the Besov spaces $\CC^{\alpha}$ which we
encounter below but rather in weighted Besov spaces. Handling SPDEs in weighted function spaces is more delicate and we prefer here to concentrate on the simpler situation of the torus.

We will use Littlewood--Paley blocks to obtain a decomposition of
distributions into an infinite series of smooth functions. Of course, we have
already such a decomposition at our disposal: $f = \sum_k \hat{f} (k)
e^{\ast}_k$. But it turns out to be convenient not to consider each Fourier
coefficient separately, but to work with projections on dyadic Fourier blocks.

\begin{definition}
  A \tmtextit{dyadic partition of unity} $(\chi,\rho)$ consists of two nonnegative radial
  functions $\chi, \rho \in C^{\infty} (\mathbb{R}^d, \mathbb{R})$, where
  $\chi$ is supported in a ball $\CB = \{ | x | \leqslant c \}$ and $\rho$ is
  supported in an annulus $\CA = \{ a \leqslant | x | \leqslant b \}$ for
  suitable $a, b, c > 0$, such that
  \begin{enumerate}
    \item $\chi + \sum_{j \geqslant 0} \rho (2^{- j} \cdummy) \equiv 1$ and
    
    \item \label{property-2-dyadic}$\tmop{supp} (\chi) \cap \tmop{supp} (\rho
    (2^{- j} \cdummy)) = \emptyset$ for $j \geqslant 1$ and $\tmop{supp} (\rho
    (2^{- i} \cdummy)) \cap \tmop{supp} (\rho (2^{- j} \nosymbol \cdummy))
    = \emptyset$ for all $i, j \geqslant 0$ with $| i - j | \geqslant 1$.
  \end{enumerate}
  We will often write $\rho_{- 1} = \chi$ and $\rho_j = \rho (2^{- j}
  \cdummy)$ for $j \geqslant 0$.
\end{definition}

Dyadic partitions of unity exist, see~{\cite{Bahouri2011}}. From now on
we fix a dyadic partition of unity $(\chi, \rho)$ and define the dyadic blocks
\[ \Delta_j f = \rho_j (\mathD) f = \CF^{- 1} (\rho_j \hat{f}), \hspace{1em}
   j \geqslant - 1, \]
where here and in the following we use that every function on $\mathbb{R}^d$ can be naturally interpreted as a function on $\mathbb{Z}^d$. We also use the notation
\[ S_j f = \sum_{i \leqslant j - 1} \Delta_i f \]
as well as $K_i = (2 \pi)^{d / 2} \CF^{- 1} \rho_i$ so that
\[ K_i \ast f = \CF^{- 1} ( \rho_i \CF f ) = \Delta_i f. \]
From this representation we can also see the reason for considering smooth partitions rather than indicator functions: From Young's inequality we get only $\| \mathbb{I}_{[2^j, 2^{j + 1})} (| \mathD |) f \|_{L^{\infty}} \le \| \CF^{-1}\mathbb{I}_{[2^j, 2^{j + 1})}\|_{L^1} \| f \|_{L^{\infty}} \lesssim j \| f \|_{L^{\infty}}$ for $f \in L^\infty$, whereas $\| \rho_j (\mathD)
f \|_{L^{\infty}} \lesssim \| f \|_{L^{\infty}}$ uniformly in $j$.

Every dyadic block has a compactly supported Fourier transform and is
therefore in $\CS$. It is easy to see that $f = \sum_{j \geqslant - 1}
\Delta_j f = \lim_{j \rightarrow \infty} S_j f$ for all $f \in \CS'$.

For $\alpha \in \mathbb{R}$, the H{\"o}lder-Besov space $\CC^{\alpha}$ is
given by $\CC^{\alpha} = B^{\alpha}_{\infty, \infty} (\mathbb{T}^d,
\mathbb{R})$, where for $p, q \in [1, \infty]$ we define
\[
    B^{\alpha}_{p, q} = B^{\alpha}_{p, q} (\mathbb{T}^d, \mathbb{R}) = \Big\{ f \in \CS' : \|f\|_{B^{\alpha}_{p, q}} = \Big( \sum_{j \geqslant - 1} (2^{j \alpha} \| \Delta_j f\|_{L^p})^q \Big)^{1 / q} < \infty \Big\},
\]
with the usual interpretation as $\ell^{\infty}$ norm if $q = \infty$. Then
$B^{\alpha}_{p, q}$ is a Banach space and while the norm $\lVert \cdummy
\rVert_{B^{\alpha}_{p, q}}$ depends on $(\chi, \rho)$, the space
$B^{\alpha}_{p, q}$ does not and any other dyadic partition of unity
corresponds to an equivalent norm (for $(p,q) = (\infty, \infty)$ this follows from Lemma~\ref{lem: Besov
regularity of series} below, for the general case see~\cite{Bahouri2011}, Lemma~2.69). We write $\lVert \cdummy \rVert_{\alpha}$
instead of $\lVert \cdummy \rVert_{B^{\alpha}_{\infty, \infty}}$.

\begin{exercise}
  Let $\delta_0$ denote the Dirac delta in 0. Show that $\delta_0 \in \CC^{-d}$.
\end{exercise}

If $\alpha \in (0, \infty) \setminus \mathbb{N}$, then $\CC^{\alpha}$ is the
space of $\lfloor \alpha \rfloor$ times differentiable functions whose partial
derivatives of order $\lfloor \alpha \rfloor$ are ($\alpha - \lfloor \alpha
\rfloor$)-H{\"o}lder continuous (see page~99 of~{\cite{Bahouri2011}}). Note
however, that for $k \in \mathbb{N}$ the space $\CC^k$ is strictly larger
than $C^k$, the space of $k$ times continuously differentiable functions.
Below we will give the proof for $\alpha \in (0, 1)$, but before we still need
some tools.

Recall that Schwartz functions on $\mathbb{R}^d$ are functions $f\in C^\infty(\mathbb{R}^d)$ such that for every multiindex $\mu$ and all $n\ge 0$ we have
\[
\sup_{x\in\mathbb{R}^d}\, (1+|x|)^n |\partial^\mu f(x)| < \infty.
\]
\begin{lemma}
  \label{lem:poisson}(Poisson summation) Let $\varphi : \mathbb{R}^d
  \rightarrow \mathbb{C}$ be a Schwartz function. Then
  \[ \CF^{- 1} \varphi (x) = \sum_{k \in \mathbb{Z}^d} \CF^{-
     1}_{\mathbb{R}^d} \varphi (x + 2 \pi k), \]
  for all $x \in \mathbb{T}^d$, where $\CF^{- 1}_{\mathbb{R}^d} \varphi (x)
  = (2 \pi)^{- d / 2} \int_{\mathbb{R}^d} \varphi (y) e^{i \langle x, y
  \rangle} \mathd y$.
\end{lemma}

\begin{proof}
  Let $g (x) = \sum_{k \in \mathbb{Z}^d} \CF^{- 1}_{\mathbb{R}^d} \varphi (x
  + 2 \pi k)$. The function $\CF^{- 1}_{\mathbb{R}^d} \varphi$ is of rapid
  decay since $\varphi \in \CS$ so the sum converges absolutely and defines a
  continuous function $g : \mathbb{R}^d \rightarrow \mathbb{R}$ which is
  periodic of period $2 \pi$ in every direction. The Fourier transform over
  the torus $\mathbb{T}^d$ of this function is
  \[ \CF g (y) = \int_{\mathbb{T}^d} e^{- i \langle x, y \rangle} g (x)
     \frac{\mathd x}{(2 \pi)^{d / 2}} = \int_{\mathbb{T}^d} \sum_{k \in
     \mathbb{Z}^d} \CF^{- 1}_{\mathbb{R}^d} \varphi (x + 2 \pi k) e^{- i
     \langle x + 2 \pi k, y \rangle} \frac{\mathd x}{(2 \pi)^{d / 2}} \]
  since $e^{- i \langle 2 \pi k, y \rangle} = 1$ for all $y \in \mathbb{Z}^d$. By dominated convergence the
  sum and the integral can be combined in a overall integration over
  $\mathbb{R}^d$:
  \[ \CF g (y) = \int_{\mathbb{R}^d} \CF^{- 1}_{\mathbb{R}^d} \varphi (x)
     e^{- i \langle x, y \rangle} \frac{\mathd x}{(2 \pi)^{d / 2}} =
     \CF_{\mathbb{R}^d} \CF^{- 1}_{\mathbb{R}^d} \varphi (y) = \varphi (y),
  \]
  where $\CF_{\mathbb{R}^d} f (x) = \CF_{\mathbb{R}^d}^{-1} f (-x)$. So we deduce that $g (x) = \CF^{- 1} \varphi (x)$.
\end{proof}

\begin{exercise}\label{exercise:besov-space-inequalities}Show that $\lVert \cdummy
  \rVert_{\alpha} \leqslant \lVert \cdummy \rVert_{\beta}$ for $\alpha
  \leqslant \beta$, that $\lVert \cdummy \rVert_{L^{\infty}} \lesssim \lVert
  \cdummy \rVert_{\alpha}$ for $\alpha > 0$, that $\lVert \cdummy
  \rVert_{\alpha} \lesssim \lVert \cdummy \rVert_{L^{\infty}}$ for $\alpha
  \leqslant 0$, and that $\|S_j \cdummy \|_{L^{\infty}} \lesssim 2^{j \alpha}
  \| \cdummy \|_{\alpha}$ for $\alpha < 0$. These inequalities will be very important for us in the following and we will often use them without mentioning it specifically.
  
  \tmtextbf{Hint:} When proving $\lVert \cdummy \rVert_{\alpha} \lesssim
  \lVert \cdummy \rVert_{L^{\infty}}$ for $\alpha \leqslant 0$, you might need
  Poisson's summation formula.
\end{exercise}

The following Bernstein inequality is extremely useful when dealing with
functions with compactly supported Fourier transform.

\begin{lemma}
  \label{lemma:Bernstein}(Bernstein inequality) Let $\CB$ be a ball and $k \in
  \mathbb{N}_0$. For any $\lambda \geqslant 1$, $1 \leqslant p \leqslant q
  \leqslant \infty$, and $f \in L^p$ with $\tmop{supp} (\CF f) \subseteq
  \lambda \CB$ we have
  \[ \max_{{\mu} \in \mathbb{N}^d : |{\mu}| = k} \|
     \partial^{{\mu}} f\|_{L^q} \lesssim_{k, \CB} \lambda^{k + d (
     \frac{1}{p} - \frac{1}{q} )} \|f\|_{L^p} . \]
\end{lemma}

\begin{proof}
  Let $\psi$ be a compactly supported $C^{\infty}$ function on $\mathbb{R}^d$
  such that $\psi \equiv 1$ on $\CB$ and write $\psi_{\lambda} (x) = \psi
  (\lambda^{- 1} x)$. Then
  \begin{align*}
      \partial^{{\mu}} f (x) & = \partial^{{\mu}} \CF^{- 1} ( \psi_{\lambda} \CF f ) (x) = (2 \pi)^{d / 2} \langle f,    \partial^{{\mu}} ( \CF^{- 1} \psi_{\lambda} ) (x - \cdummy) \rangle \\
      & = (2 \pi)^{d / 2} ( f \ast     \partial^{{\mu}} ( \CF^{- 1} \psi_{\lambda} ) ) (x).
   \end{align*}
  By Young's inequality, we get
  \[ \| \partial^{{\mu}} f \|_{L^q} \lesssim \| f \|_{L^p} \|
     \partial^{{\mu}} ( \CF^{- 1} \psi_{\lambda} )
     \|_{L^r}, \]
  where $1 + 1 / q = 1 / p + 1 / r$. Now it is a short exercise to verify $\|
  \cdummy \|_{L^r} \leqslant \| \cdummy \|_{L^1}^{1 / r} \| \cdummy
  \|_{L^{\infty}}^{1 - 1 / r}$, and
  \begin{align*}
     \left\| \partial^{{\mu}} \left( \CF^{- 1} \psi_{\lambda} \right) \right\|_{L^1} & = \int_{\mathbb{T}^d} \Big| \sum_k \partial^{{\mu}} \left( \CF^{- 1}_{\mathbb{R}^d} \psi_{\lambda} \right) (x + 2 \pi k) \Big| \mathd x \leqslant \int_{\mathbb{R}^d} | \partial^{{\mu}} ( \CF^{- 1}_{\mathbb{R}^d} \psi_{\lambda} ) (x) | \mathd x \\
      & = \lambda^{| {\mu} |} \int_{\mathbb{R}^d} \lambda^d | ( \partial^{{\mu}} \CF^{- 1}_{\mathbb{R}^d} \psi ) (\lambda x)| \mathd x \simeq \lambda^{| {\mu} |} \nocomma,
  \end{align*}
  whereas
  \begin{align*}
     \sup_{x \in \mathbb{T}^d} \Big| \sum_k \partial^{{\mu}} (\CF^{- 1}_{\mathbb{R}^d} \psi_{\lambda}) (x + 2 \pi k) \Big| & =  \lambda^{d + | {\mu} |} \sup_{x \in \mathbb{T}^d} \Big| \sum_k ( \partial^{{\mu}} \CF^{- 1}_{\mathbb{R}^d} \psi ) (\lambda (x + 2 \pi k)) \Big| \\
     & \lesssim \lambda^{d + | {\mu} |} \sup_{x \in \mathbb{T}^d} \sum_k (1 + \lambda | x + 2 \pi k |)^{- 2 d} \\
     & \lesssim \lambda^{d + | {\mu} |} \sup_{x \in \mathbb{T}^d} \sum_k (1 + | x + 2 \pi k |)^{- 2 d} \lesssim \lambda^{d + | {\mu} |}.
  \end{align*}
  We end up with
  \[ \| \partial^{{\mu}} f \|_{L^q} \lesssim \| f \|_{L^p} \|
     \partial^{{\mu}} ( \CF^{- 1} \psi_{\lambda} )
     \|_{L^r} \lesssim \| f \|_{L^p} \lambda^{| {\mu} | / r}
     \lambda^{(d + | {\mu} |) (1 - 1 / r)} = \| f \|_{L^p} \lambda^{|
     {\mu} | + d (1 / p - 1 / q)} . \]
  
\end{proof}

It then follows immediately that for $\alpha \in \mathbb{R}$, $f \in
\CC^{\alpha}$, ${\mu} \in \mathbb{N}^d_0$, we have $\partial^{{\mu}}
f \in \CC^{\alpha - | {\mu} |}$. Another simple application of the
Bernstein inequalities is the Besov embedding theorem, the proof of which we
leave as an exercise.

\begin{lemma}
  \label{lem:besov embedding}(Besov embedding) Let $1 \leqslant p_1 \leqslant
  p_2 \leqslant \infty$ and $1 \leqslant q_1 \leqslant q_2 \leqslant \infty$,
  and let $\alpha \in \mathbb{R}$. Then $B^{\alpha}_{p_1, q_1}$ is
  continuously embedded into $B^{\alpha - d (1 / p_1 - 1 / p_2)}_{p_2, q_2}$.
\end{lemma}

\begin{exercise}
  \label{exo:wn besov}In the setting of Exercise~\ref{exo:wn}, use Besov
  embedding to show that $\mathbb{E} [\| \tilde{\xi}
  \|_{- d / 2 - \varepsilon}^p] < \infty$ for all $p \geqslant 1$ and
  $\varepsilon > 0$ (in particular $\tilde{\xi} \in \CC^{- d / 2 -}$ almost
  surely).
  
  \tmtextbf{Hint:} Estimate $\mathbb{E} [\| \tilde{\xi} \|_{B^{\alpha}_{2 p,
  2 p}}^{2 p}]$ using Gaussian hypercontractivity (equivalence of moments).
\end{exercise}

As another application of the Bernstein inequality, let us show that
$\CC^{\alpha} = C^{\alpha}$ for $\alpha \in (0, 1)$.

\begin{lemma}
  For $\alpha \in (0, 1)$ we have $\CC^{\alpha} = C^{\alpha}$, the space of
  $\alpha$-H{\"o}lder continuous functions, and
  \[ \| f \|_{\alpha} \simeq \| f \|_{C^{\alpha}} = \| f \|_{L^{\infty}} +
     \sup_{x \neq y} \frac{| f (x) - f (y) |}{d_{\mathbb{T}^d} (x,
     y)^{\alpha}}, \]
  where $d_{\mathbb{T}^d} (x, y)$ denotes the canonical distance on
  $\mathbb{T}^d$.
\end{lemma}

\begin{proof}
  Start by noting that for $f \in \CC^{\alpha}$ we have $\| f \|_{L^{\infty}}
  \leqslant \sum_j \| \Delta_j f \|_{L^{\infty}} \leqslant \sum_j 2^{- j
  \alpha} \| f \|_{\alpha} \lesssim \| f \|_{\alpha}$. Let now $x \neq y \in
  \mathbb{T}^d$ and choose $j_0$ with $2^{- j_0} \simeq d_{\mathbb{T}^d} (x,
  y)$. For $j \leqslant j_0$ we use Bernstein's inequality to obtain
  \[ | \Delta_j f (x) - \Delta_j f (y) | \lesssim \| \mathD \Delta_j f
     \|_{L^{\infty}} d_{\mathbb{T}^d} (x, y) \lesssim 2^j \| \Delta_j f
     \|_{L^{\infty}} d_{\mathbb{T}^d} (x, y) \leqslant 2^{j (1 - \alpha)} \| f
     \|_{\alpha} d_{\mathbb{T}^d} (x, y), \]
  whereas for $j > j_0$ we simply estimate
  \[ | \Delta_j f (x) - \Delta_j f (y) | \lesssim \| \Delta_j f
     \|_{L^{\infty}} \lesssim 2^{- j \alpha} \| f \|_{\alpha} . \]
  Summing over $j$, we get
  \begin{align*}
     | f (x) - f (y) | & \leqslant \sum_{j \leqslant j_0} 2^{j (1 - \alpha)} \| f \|_{\alpha} d_{\mathbb{T}^d} (x, y) + \sum_{j > j_0} 2^{- j \alpha} \| f \|_{\alpha} \\
     & \simeq \| f \|_{\alpha} (2^{j_0 (1 - \alpha)} d_{\mathbb{T}^d} (x, y) +  2^{- j_0 \alpha}) \simeq \| f \|_{\alpha} d_{\mathbb{T}^d} (x, y)^{\alpha}.
  \end{align*}
  Conversely, if $f \in C^{\alpha}$, then we estimate $\| \Delta_{- 1} f
  \|_{L^{\infty}} \lesssim \| f \|_{L^{\infty}}$. For $j \geqslant 0$, the
  function $\rho_j$ satisfies $\int (\CF^{-1} \rho_j)(x) \mathd x = 0$, and therefore
  \begin{align*}
     | \Delta_j f (x) | & = \Big| \int_{\mathbb{T}^d} \CF^{- 1} \rho_j (x - y) (f (y) - f (x)) \mathd y \Big| \\
     & = \Big| \int_{\mathbb{T}^d} \sum_k \CF^{- 1}_{\mathbb{R}^d} \rho_j (x - y + 2 \pi k) (f (y) - f (x)) \mathd y \Big| \\
     & = \Big| \int_{\mathbb{R}^d} \CF^{- 1}_{\mathbb{R}^d} \rho_j (x - y) (f(y) - f (x)) \mathd y \Big|.
  \end{align*}
  Now $| f (y) - f (x) | \leqslant \| f \|_{C^{\alpha}} d_{\mathbb{T}^d} (x,
  y)^{\alpha} \leqslant \| f \|_{C^{\alpha}} | x - y |^{\alpha}$, and
  thus we end up with
  \begin{align*}
     | \Delta_j f (x) | & \leqslant \| f \|_{C^{\alpha}} \Big| 2^{j d} \int_{\mathbb{R}^d} | ( \CF^{- 1}_{\mathbb{R}^d} \rho ) (2^j (x - y)) |  | x - y |^{\alpha} \mathd y \Big| \\
     & = \| f \|_{C^{\alpha}} 2^{-j\alpha} \Big| 2^{j d} \int_{\mathbb{R}^d} | ( \CF^{- 1}_{\mathbb{R}^d} \rho ) (2^j (x - y)) |  | 2^j (x - y) |^{\alpha} \mathd y \Big| \lesssim \| f \|_{C^{\alpha}} 2^{- j \alpha}.
  \end{align*}
\end{proof}

The following lemma, a characterization of Besov regularity for functions that
can be decomposed into pieces which are localized in Fourier space, will be
immensely useful in what follows.

\begin{lemma}
  \label{lem: Besov regularity of series}{\tmdummy}
  
  \begin{enumeratenumeric}
    \item Let $\CA$ be an annulus, let $\alpha \in \mathbb{\mathbb{R}}$, and
    let $(u_j)$ be a sequence of smooth functions such that $\CF u_j$ has its
    support in $2^j \CA$, and such that $\|u_j \|_{L^{\infty}} \lesssim 2^{- j
    \alpha}$ for all $j$. Then
    \[ u = \sum_{j \geqslant - 1} u_j \in \CC^{\alpha} \hspace{2em} \tmop{and}
       \hspace{2em} \|u\|_{\alpha} \lesssim \sup_{j \geqslant - 1} \{2^{j
       \alpha} \|u_j \|_{L^{\infty}} \} . \]
    \item Let $\CB$ be a ball, let $\alpha > 0$, and let $(u_j)$ be a sequence
    of smooth functions such that $\CF u_j$ has its support in $2^j \CB$, and
    such that $\|u_j \|_{L^{\infty}} \lesssim 2^{- j \alpha}$ for all $j$.
    Then
    \[ u = \sum_{j \geqslant - 1} u_j \in \CC^{\alpha} \hspace{2em} \tmop{and}
       \hspace{2em} \|u\|_{\alpha} \lesssim \sup_{j \geqslant - 1} \{2^{j
       \alpha} \|u_j \|_{L^{\infty}} \} . \]
  \end{enumeratenumeric}
\end{lemma}

\begin{proof}
  If $\CF u_j$ is supported in $2^j \CA$, then $\Delta_i u_j \neq 0$ only for
  $i \sim j$. Hence, we obtain
  
  \begin{gather*}
    \| \Delta_i u\|_{L^{\infty}} \leqslant \sum_{j : j \sim i} \| \Delta_i u_j
    \|_{L^{\infty}} \leqslant \sup_{k \geqslant - 1} \{2^{k \alpha} \|u_k
    \|_{L^{\infty}} \}  \sum_{j : j \sim i} 2^{- j \alpha} \simeq \sup_{k
    \geqslant - 1} \{2^{k \alpha} \|u_k \|_{L^{\infty}} \} 2^{- i \alpha} .
  \end{gather*}
  
  If $\CF u_j$ is supported in $2^j \CB$, then $\Delta_i u_j \neq 0$ only for
  $i \lesssim j$. Therefore,
  \[ \| \Delta_i u\|_{L^{\infty}} \leqslant \sum_{j : j \gtrsim i} \| \Delta_i
     u_j \|_{L^{\infty}} \leqslant \sup_{k \geqslant - 1} \{2^{k \alpha} \|u_k
     \|_{L^{\infty}} \}  \sum_{j : j \gtrsim i} 2^{- j \alpha} \lesssim
     \sup_{k \geqslant - 1} \{2^{k \alpha} \|u_k \|_{L^{\infty}} \} 2^{- i
     \alpha}, \]
  using $\alpha > 0$ in the last step.
\end{proof}

When solving SPDEs, we will need the smoothing properties of the heat
semigroup. 
We define $\LL^{\alpha} = C \CC^{\alpha} \cap C^{\alpha / 2} L^{\infty}$
for $\alpha \in (0, 2)$. For $T > 0$ we set $\LL^{\alpha}_T = C_T \CC^{\alpha}
\cap C^{\alpha / 2}_T L^{\infty}$ and we equip $\LL^{\alpha}_T$ with the norm
\[ \| \cdummy \|_{\LL^{\alpha}_T} = \max \{ \| \cdummy \|_{C_T \CC^{\alpha}},
   \| \cdummy \|_{C^{\alpha / 2}_T L^{\infty}} \} . \]
The notation $\LL^{\alpha}$ is chosen to be reminiscent of the operator $\LL =
\partial_t - \Delta$ and indeed the parabolic spaces $\LL^{\alpha}$ are
adapted to $\LL$ in the sense that the temporal regularity ``counts twice'',
which is due to the fact that $\LL$ contains a first order temporal but a
second order spatial derivative. If we would replace $\Delta$ by a fractional
Laplacian $- (- \Delta)^{\sigma}$, then we would have to consider the space $C
\CC^{\alpha} \cap C^{\alpha / (2 \sigma)} L^{\infty}$ instead of
$\LL^{\alpha}$.

We have the following Schauder estimate on the scale of $( \LL^{\alpha})_{\alpha}$ spaces:

\begin{lemma}
  \label{lem:schauder}Let $\alpha \in (0, 2)$ and let $(P_t)_{t \geqslant 0}$
  be the semigroup generated by the periodic Laplacian, $\CF (P_t f) (k) =
  e^{- t | k |^2} \CF f (k)$. For $f \in C \CC^{\alpha - 2}$ define $J f (t) =
  \int_0^t P_{t - s} f_s \mathd s$. Then $J f$ is the solution to $\LL J f = f$, $J f(0) = 0$, and we have
  \[ \| J f \|_{\LL^{\alpha}_T} \lesssim (1 + T) \| f \|_{C_T \CC^{\alpha -
     2}} \]
  for all $T > 0$. If $u \in \CC^{\alpha}$, then $t \mapsto P_t u$ is the solution to $\LL P_\cdot u = 0$, $P_0 u = u$, and we have
  \[ \| t \mapsto P_t u \|_{\LL^{\alpha}_T} \lesssim \| u \|_{\alpha} . \]
\end{lemma}

\tmtextbf{Bibliographic notes.} For a gentle introduction to Littlewood--Paley
theory and Besov spaces see the recent monograph~{\cite{Bahouri2011}}, where
most of our results are taken from. There the case of tempered distributions
on $\mathbb{R}^d$ is considered. The theory on the torus is developed
in~{\cite{Schmeisser1987}}. The Schauder estimates for the heat semigroup are
classical and can be found in~{\cite{Gubinelli2012,Gubinelli2014}}.

\section{Diffusion in a random environment}\label{sec:homogenization}

Let us consider the following $d$-dimensional homogenization problem. Fix
$\varepsilon > 0$ and let $u^{\varepsilon} : \mathbb{R}_+ \times
\mathbb{T}^d \rightarrow \mathbb{R}$ be the solution to the Cauchy problem
\begin{equation}\label{eq:oscillating potential}
   \partial_t u^{\varepsilon} (t, x) = \Delta u^{\varepsilon} (t, x) +
   \varepsilon^{- \alpha} V (x / \varepsilon) u^{\varepsilon} (t, x),
   \hspace{2em} u^{\varepsilon} (0) = u_0,
\end{equation}
where $V : \mathbb{T}^d_{\varepsilon} \rightarrow \mathbb{R}$ is a random
field defined on the rescaled torus $\mathbb{T}^d_{\varepsilon} =
(\mathbb{R}/ (2 \pi \varepsilon^{- 1} \mathbb{Z}))^d$. This model describes
the diffusion of particles in a random medium (replacing $\partial_t$ by $i
\partial_t$ gives the Schr{\"o}dinger equation of a quantum particle evolving
in a random \ potential). For a review of related results the reader can give
a look at the recent paper of Bal and Gu~{\cite{bal_limiting_2013}}. The limit
$\varepsilon \rightarrow 0$ corresponds to looking at the large scale behavior
of the model since~\eqref{eq:oscillating potential} can be understood as the equation for the
{\tmem{macroscopic}} density $u^{\varepsilon} (t, x) = u (t / \varepsilon^2, x
/ \varepsilon)$ which corresponds to a {\tmem{microscopic}} density $u :
\mathbb{R}_+ \times \mathbb{T}_{\varepsilon}^d \rightarrow \mathbb{R}$
evolving according to the parabolic equation
\[ \partial_t u (t, x) = \Delta u (t, x) + \varepsilon^{2 - \alpha} V (x) u
   (t, x), \hspace{2em} u (0, \cdummy) = u_0 (\varepsilon
   \cdummy). \]
Slightly abusing notation, we do not index $u$ or $V$ by $\varepsilon$ despite the fact that they of course depend on it.   
We assume that $V \colon \mathbb{T}^d_{\varepsilon} \rightarrow \mathbb{R}$ is
Gaussian and has mean zero and homogeneous correlation function
$C_{\varepsilon}$ given by
\[ C_{\varepsilon} (x - y) =\mathbb{E} [V (x) V (y)] = ( \varepsilon /
   \sqrt{2 \pi} )^d \sum_{k \in \varepsilon \mathbb{Z}^d} e^{i
   \langle x - y, k \rangle} R (k). \]
On $R \colon\mathbb{R}^d \rightarrow \mathbb{R}_+$ we make the following hypothesis: for
some $\beta \in (0, d]$ we have $R (k) = | k |^{\beta - d} \tilde{R} (k)$
where $\tilde{R} \in \CS (\mathbb{R}^d)$ is a smooth radial function of rapid
decay. For $\beta < d$ it would be equivalent to require that spatial
correlations (in the limit $\varepsilon \rightarrow 0$) decay as $| x |^{-
\beta}$. For $\beta = d$ this hypothesis means that spatial correlations are
of rapid decay. Indeed by dominated convergence
\begin{align*}
   \lim_{\varepsilon \rightarrow 0} C_{\varepsilon} (x) & = \int_{\mathbb{R}^d} \frac{\mathd k}{(2 \pi)^{d / 2}} e^{i \langle x, k \rangle} R (k) = \int_{\mathbb{R}^d} \frac{\mathd k}{(2 \pi)^{d / 2}} e^{i \langle x, k \rangle} | k |^{\beta - d} \tilde{R} (k) \\
    & = (2 \pi)^{d / 2} \left( \CF^{- 1}_{\mathbb{R}^d} (| \cdot |^{\beta - d}) \ast \CF_{\mathbb{R}^d}^{- 1} (\tilde{R}) \right) (x) .
\end{align*}
Here we applied the formula of Exercise~\ref{exercise:convolution}, which also holds for the Fourier transform on $\mathbb{R}^d$. Now $\CF^{- 1}_{\mathbb{R}^d} (\tilde{R}) \in \CS (\mathbb{R}^d)$ and
$\CF^{- 1}_{\mathbb{R}^d} (| \cdot |^{\beta - d}) (x) \simeq | x |^{- \beta}$
if $0 < \beta < d$ (see for example Proposition~1.29 of~{\cite{Bahouri2011}}),
so $\lim_{\varepsilon \rightarrow 0} | C_{\varepsilon} (x) | \lesssim | x |^{-
\beta}$ for $| x | \rightarrow + \infty$.

\

Let us write $V_{\varepsilon} (x) = \varepsilon^{- \alpha} V (x /
\varepsilon)$ so that~\eqref{eq:oscillating potential} can be rewritten as $\partial_t u^{\varepsilon} = \Delta u^{\varepsilon} +
V_{\varepsilon} u^{\varepsilon}$, and let us compute the variance of the
Littlewood--Paley blocks of $V_{\varepsilon}$.

In order to perform more easily some computations we can introduce a family of
centered complex Gaussian random variables $\{ g (k) \}_{k \in \varepsilon
\mathbb{Z}_0}$ such that $g (k)^{\ast} = g (- k)$ and $\mathbb{E} [g (k) g
(k')] = \delta_{k + k' = 0}$ and represent $V_{\varepsilon} (x)$ as
\[ V_{\varepsilon} (x) = \frac{\varepsilon^{d / 2 - \alpha}}{( \sqrt{2
   \pi} )^{d / 2}} \sum_{k \in \varepsilon \mathbb{Z}^d} e^{i \langle
   x, k / \varepsilon \rangle} \sqrt{R (k)} g (k) \]

\begin{lemma}\label{lem:Veps reg}
  Assume $\beta - 2 \alpha \geqslant 0$.We have for any $\varepsilon > 0$ and
  $i \geqslant 0$ and any $0 \leqslant \kappa \leqslant \beta - 2 \alpha$:
  \[ \mathbb{E} [| \Delta_i V_{\varepsilon} (x) |^2] \lesssim 2^{(2 \alpha +
     \kappa) i} \varepsilon^{\kappa} . \]
  This estimate implies that if $\beta > 2 \alpha$, then for all  $\delta > 0$ we have $V_{\varepsilon} \rightarrow 0$ in $L^2 (\Omega ; B^{- \alpha - \delta}_{2, 2} (\mathbb{T}^d))$ as $\varepsilon \rightarrow 0$.
\end{lemma}

\begin{proof}
  A spectral computation gives
  \[ \Delta_i V_{\varepsilon} (x) = \frac{\varepsilon^{d / 2 - \alpha}}{(
     \sqrt{2 \pi} )^{d / 2}} \sum_{k \in \varepsilon \mathbb{Z}^d}
     e^{i \langle x, k / \varepsilon \rangle} \rho_i (k / \varepsilon) \sqrt{R
     (k)} g (k) \]
  so
  \begin{equation}
    \begin{array}{lll}
      \mathbb{E} [| \Delta_i V_{\varepsilon} (x) |^2] & = & \varepsilon^d
      ( \sqrt{2 \pi} )^{-d} \varepsilon^{- 2 \alpha} \sum_{k \in
      \varepsilon \mathbb{Z}^d} \rho_i (k / \varepsilon)^2 R (k)\\
      & = & ( \sqrt{2 \pi} )^{-d} \varepsilon^{d - 2 \alpha} \sum_{k
      \in \varepsilon \mathbb{Z}^d} \rho (k / (\varepsilon 2^i))^2 R (k)\\
      & \lesssim & \varepsilon^{d - 2 \alpha} 2^{id} \sup_{k \in \varepsilon
      2^i \CA} R (k),
    \end{array} \label{eq:est-lp-v}
  \end{equation}
  where $\CA$ is the annulus in which $\rho$ is supported. Now recall that $\beta \le d$ so that $(\varepsilon
  2^i)^{\beta-d}\ge1$ whenever $\varepsilon
  2^i \leqslant 1$, which leads to $\mathbb{E} [| \Delta_i V_{\varepsilon} (x) |^2]
  \lesssim 2^{id} \varepsilon^{d - 2 \alpha} (\varepsilon 2^i)^{\beta - d} =
  \varepsilon^{\beta - 2 \alpha} 2^{i \beta}$ in that case. The assumption $\beta - 2
  \alpha \geqslant 0$ then implies $\mathbb{E} [| \Delta_i V_{\varepsilon}
  (x) |^2] \lesssim 2^{(2 \alpha + \kappa) i} \varepsilon^{\kappa}$ for any $0
  \leqslant \kappa \leqslant \beta - 2 \alpha$. In the case $\varepsilon 2^i >
  1$ we use that $\int_{\R^d} R (k) \mathd k < + \infty$ to estimate
  \[ \varepsilon^d \sum_{k \in \varepsilon \mathbb{Z}^d} \rho (k /
     (\varepsilon 2^i))^2 R (k) \leqslant \varepsilon^d \sum_{k \in
     \mathbb{Z}^d} R (\varepsilon k) \lesssim \int_{\mathbb{R}^d} R (k)
     \mathd k < + \infty, \]
  and then $\mathbb{E} [| \Delta_i V_{\varepsilon} (x) |^2] \lesssim
  \varepsilon^{- 2 \alpha} \lesssim 2^{2 \alpha i} (\varepsilon 2^i)^{\kappa}$
  for any small $\kappa > 0$. 
\end{proof}

\begin{remark}
   Using Gaussian hypercontractivity, we get from Lemma~\ref{lem:Veps reg} that
   \[
      \mathbb{E} [| \Delta_i V_{\varepsilon} (x) |^{2p}] \lesssim \mathbb{E} [| \Delta_i V_{\varepsilon} (x) |^{2}]^p \lesssim 2^{(2 \alpha +  \kappa) p i} \varepsilon^{\kappa p}
   \]
   whenever $p \ge 1$, and therefore
   \[
      \lim_{\varepsilon \to 0} \mathbb{E} [ \| V_{\varepsilon} \|^{2p}_{B_{2p,2p}^{- \alpha - \delta}} ] = \lim_{\varepsilon \to 0} \sum_{i \ge -1} 2^{i(-\alpha-\delta)2p} \int_\mathbb{T} \mathbb{E} [| \Delta_i V_{\varepsilon} (x) |^{2p}] \mathd x = 0
   \]
   whenever $\delta > 0$. By the Besov embedding theorem, this shows that for all $p,\delta > 0$
   \[
      \lim_{\varepsilon \to 0} \mathbb{E} [ \| V_{\varepsilon} \|^{p}_{\CC^{- \alpha - \delta}} ] = 0.
   \]
\end{remark}

Slightly improving the computation carried out in equation~(\ref{eq:est-lp-v}) we can also see that if $\beta - 2 \alpha < 0$, then essentially $V_{\varepsilon}$ does not converge in any reasonable sense since the variance of the Littlewood--Paley blocks explodes.

\begin{remark}
  \label{rmk:different V blocks orthogonal}The same calculation as
  in~(\ref{eq:est-lp-v}) shows that
  \[ \mathbb{E} [\Delta_i V_{\varepsilon} (x) \Delta_j V_{\varepsilon} (x)] =
     0 \]
  whenever $| i - j | > 1$, because in that case $\rho_i \rho_j \equiv 0$.
\end{remark}

The previous analysis shows that it is reasonable to take $\alpha \leqslant
\beta / 2$ in order to have some hope of obtaining a well defined limit as
$\varepsilon \rightarrow 0$. In this case $V_{\varepsilon}$ stays bounded in probability (at
least) in spaces of distributions of regularity $-\alpha -$. This brings us to
the problem of obtaining estimates for the parabolic PDE
\[ \LL u^{\varepsilon} (t, x) = (\partial_t - \Delta) u^{\varepsilon} (t, x) =
   V_{\varepsilon} (x) u^{\varepsilon} (t, x), \hspace{2em} (t, x) \in [0, T]
   \times \mathbb{T}^d, \]
depending only on negative regularity norms of $V_{\varepsilon}$. On one side
the regularity of $u^{\varepsilon}$ is then limited by the regularity of the
right hand side which cannot be better than that of $V_{\varepsilon}$. On the
other side the product of $V_{\varepsilon}$ with $u^{\varepsilon}$ can cause
problems since we try to multiply an (a priori) irregular object with one of
limited regularity.

Assume that $V_{\varepsilon}$ converges to zero in $\CC^{\gamma - 2}$ for $\gamma > 0$. It is then reasonable to assume that also $V_{\varepsilon} u^{\varepsilon} \in C_T
\CC^{\gamma - 2}$, uniformly in $\varepsilon > 0$, and that $u^{\varepsilon} \in C_T \CC^{\gamma}$ as a
consequence of the regularising effect of the heat operator
(Lemma~\ref{lem:schauder}). We will see in Section~\ref{ssec:bony} below that
the product $V_{\varepsilon} u^{\varepsilon}$ is under control only if $\gamma
+ \gamma - 2 > 0$, that is if $\gamma > 1$. If $V_{\varepsilon} \rightarrow
0$ in $\CC^{- 1 +}$, it is not difficult to show that $u^{\varepsilon}$
converges as $\varepsilon \rightarrow 0$ to the solution $u$ of the linear
equation $\LL u = 0$ (for example this will follow from our analysis below, but in fact it is much simpler to show). In this case the random potential will not have any
effect in the limit.

The interesting situation then is when $\gamma \leqslant 1$. To understand
what could happen in this case let us use a simple transformation of the
solution. Write $u^{\varepsilon} = \exp (X^{\varepsilon}) v^{\varepsilon}$
where $X^{\varepsilon}$ satisfies the equation $\LL X^{\varepsilon} =
V_{\varepsilon}$ with initial condition $X^{\varepsilon} (0, \cdot) = 0$. Then
\[ \LL u^{\varepsilon} = \exp (X^{\varepsilon}) \left( v^{\varepsilon} \LL
   X^{\varepsilon} + \LL v^{\varepsilon} - v^{\varepsilon} (\partial_x
   X^{\varepsilon})^2 - 2 \langle \partial_x X^{\varepsilon}, \partial_x
   v^{\varepsilon} \rangle_{\mathbb{R}^d} \right) = \exp (X^{\varepsilon})
   v^{\varepsilon} V_{\varepsilon} . \]
Since $\exp (X^{\varepsilon}) > 0$ on $[0, T] \times \mathbb{T}^d$, this
implies that $v^{\varepsilon}$ satisfies
\[ \LL v^{\varepsilon} - v^{\varepsilon} | \partial_x X^{\varepsilon} |^2 - 2
   \langle \partial_x X^{\varepsilon}, \partial_x v^{\varepsilon}
   \rangle_{\mathbb{R}^d} = 0, \hspace{2em} (t, x) \in [0, T] \times
   \mathbb{T}^d . \]
Our Schauder estimates imply that $X^{\varepsilon} = J V_\varepsilon \in C_T \CC^{\gamma}$ with
uniform bounds in $\varepsilon > 0$, so that the problematic term is $|
\partial_x X^{\varepsilon} |^2$ for which this estimate does not guarantee
existence.

Note that $J (e^{i\langle \cdot, k \rangle})(t,x) = e^{i\langle x, k \rangle} (1 - e^{-t |k|^2})/|k|^2$, which yields
\begin{equation} \label{eq:formula-X}
  \partial_x X^{\varepsilon} (t, x) = \frac{\varepsilon^{d / 2 -
  \alpha}}{( \sqrt{2 \pi} )^{d / 2}} \sum_{k \in \varepsilon
  \mathbb{Z}^d_0} e^{i \langle x, k / \varepsilon \rangle} G_{\varepsilon}
  (t, k) g (k)
\end{equation}
where $\mathbb{Z}^d_0 =\mathbb{Z}^d \backslash \{ 0 \}$ and where
\[ G_{\varepsilon} (t, k) = i \frac{k}{\varepsilon} \frac{[1 - e^{- t | k /
   \varepsilon |^2}]}{| k / \varepsilon |^2} \sqrt{R (k)} . \]
\begin{lemma}\label{lem:Dx squared reg}
  Assume that
  \[ \sigma^2 = ( \sqrt{2 \pi} )^d \int_{\mathbb{R}^d} \frac{R
     (k)}{k^2} \mathd k < + \infty . \]
  Then if $\alpha = 1$ and $t>0$ we have
  \[ \lim_{\varepsilon \rightarrow 0} \mathbb{E} [| \partial_x
     X^{\varepsilon} |^2 (t, x)] = \sigma^2, \]
  and if $\alpha < 1$ and $t>0$
  \[ \lim_{\varepsilon \rightarrow 0} \mathbb{E} [(| \partial_x
     X^{\varepsilon} |)^2 (t, x)] = 0. \]
  Moreover 
  \[ \tmop{Var} [\Delta_q (| \partial_x X^{\varepsilon} |^2) (t, x)] \lesssim \varepsilon^{4 - 4 \alpha} \min(\sigma^4,  (\varepsilon 2^q)^{\beta - 2}  \| \tilde R \|_{\infty} \sigma^2)
     . \]
\end{lemma}

\begin{proof}
  A computation similar to that leading to equation~(\ref{eq:est-lp-v}) gives
  \begin{align*}
      \mathbb{E} [| \partial_x X^{\varepsilon} |^2 (t, x)] & = \varepsilon^d ( \sqrt{2 \pi} )^d \varepsilon^{- 2 \alpha} \sum_{k \in \varepsilon \mathbb{Z}^d_0} | k / \varepsilon |^2 \Big[ \int_0^t e^{-(t - s)  | k / \varepsilon |^2} \mathd s \Big]^2 R (k) \\
      & = \varepsilon^d ( \sqrt{2 \pi} )^d \varepsilon^{2 - 2 \alpha} \sum_{k \in \varepsilon \mathbb{Z}^d_0} \frac{[1 - e^{- t (k / \varepsilon)^2}]^2}{k^2} R (k),
  \end{align*}
  which for $\varepsilon \rightarrow 0$, $t > 0$, and $\alpha \le 1$ tends to
  \[ \lim_{\varepsilon \rightarrow 0} \mathbb{E} [| \partial_x
     X^{\varepsilon} |^2 (t, x)] =\mathbb{I}_{\alpha =1} ( \sqrt{2 \pi})^d \int_{\mathbb{R}^d} \frac{R (k)}{k^2} \mathd k
     =\mathbb{I}_{\alpha = 1} \sigma^2 . \]
  Let us now study the variance of $| \partial_x X^{\varepsilon} |^2 (t, x)$.
  Using equation~\eqref{eq:formula-X} we have
  \[
     \Delta_q (| \partial_x X^{\varepsilon} |^2) (t, x) = \frac{\varepsilon^{d - 2\alpha}}{(2 \pi)^{d/2}}   \sum_{k_1, k_2 \in \varepsilon \mathbb{Z}^d_0} e^{i \langle k_1 + k_2, x / \varepsilon \rangle} \rho_q ((k_1 + k_2) / \varepsilon) G_\varepsilon (t,k_1) G_\varepsilon (t,k_2) g (k_1) g (k_2).
  \]
  By Wick's theorem~({\cite{Janson1997}}, Theorem~1.28)
  \begin{align*}
     \tmop{Cov} (g (k_1) g (k_2), g (k_1') g (k_2')) & = \mathbb{E} [g (k_1) g(k_1')] \mathbb{E} [g (k_2) g (k_2')] +\mathbb{E} [g (k_1) g (k_2')] \mathbb{E} [g (k_2) g (k_1')] \\
      & =\mathbb{I}_{k_1 + k_1' = k_2 + k_2' = 0} +\mathbb{I}_{k_1 + k_2' = k_2 + k_1' = 0},
  \end{align*}
  which implies
  \[ \tmop{Var} [\Delta_q (| \partial_x X^{\varepsilon} |^2) (t, x)] = 2
     \frac{\varepsilon^{2d  - 4\alpha}}{(2 \pi )^{d}} \sum_{k_1, k_2 \in \varepsilon \mathbb{Z}^d_0} (\rho_q ((k_1 +  k_2) / \varepsilon))^2 | G_\varepsilon (t,k_1) |^2 | G_\varepsilon (t,k_2) |^2 . \]
  For any $q \geqslant 0$ (the case $q = - 1$ is left to the reader), the variables $k_1$ and $k_2$ are bounded away from 0 and we have
  \[ \tmop{Var} [\Delta_q (| \partial_x X^{\varepsilon} |^2) (t, x)] \lesssim
     \varepsilon^{2 d + 4 - 4 \alpha} \sum_{k_1, k_2 \in \varepsilon
     \mathbb{Z}^d_0} (\rho_q ((k_1 + k_2) / \varepsilon))^2 \frac{| R (k_1) |
     | R (k_2) |}{| k_1 |^2  | k_2 |^2} . \]
A first estimate is obtained by just dropping the factor $ \rho_q ((k_1 + k_2) / \varepsilon)$ and results in the bound
\[
\tmop{Var} [\Delta_q (| \partial_x X^{\varepsilon} |^2) (t, x)] \lesssim
     \varepsilon^{2 d + 4 - 4 \alpha} \sum_{k_1, k_2 \in \varepsilon
     \mathbb{Z}^d_0} \frac{| R (k_1) |
     | R (k_2) |}{| k_1 |^2  | k_2 |^2} \lesssim    \varepsilon^{ 4 - 4 \alpha} \sigma^4
\]
Another estimate proceeds by taking into account the constraint given by the support of  $ \rho_q ((k_1 + k_2) / \varepsilon)$. In
 order to satisfy $k_1 + k_2 \sim \varepsilon 2^q$ we must have $k_2
  \lesssim k_1 \sim \varepsilon 2^q$ or $\varepsilon 2^q \lesssim k_1 \sim
  k_2$. In the first case
  \begin{align*}
     \varepsilon^{2 d + 4 - 4 \alpha} \sum_{k_1, k_2 \in \varepsilon \mathbb{Z}^d_0} & \mathbb{I}_{k_2 \lesssim k_1 \sim \varepsilon 2^q} \frac{| R (k_1) |  | R (k_2) |}{| k_1 |^2  | k_2 |^2}  \lesssim 2^{q(\beta-2)} \varepsilon^{ d  + \beta+ 2 - 4 \alpha} \| \tilde R \|_{\infty} \sum_{k_2 \in \varepsilon \mathbb{Z}^d_0} \mathbb{I}_{k_2 \lesssim \varepsilon 2^q} \frac{| R (k_2) |}{| k_2 |^2} \\
     & \lesssim (\varepsilon 2^q)^{\beta - 2} \| \tilde R \|_{\infty} \int \mathd k  \frac{| R (k) |}{| k |^2} \lesssim  (\varepsilon 2^q)^{\beta - 2} \varepsilon^{4 - 4 \alpha} \| \tilde R \|_{\infty} \sigma^2
  \end{align*}
  since $| R (k_1) | / | k_1 |^2 \lesssim \| \tilde R \|_{\infty} (\varepsilon 2^q)^{\beta-d-2}$. If $\varepsilon 2^q \lesssim k_1 \sim k_2$ we similarly have
  \begin{align*}
     \varepsilon^{2 d + 4 - 4 \alpha} \sum_{k_1, k_2 \in \varepsilon \mathbb{Z}^d_0} & \mathbb{I}_{\varepsilon 2^q \lesssim k_1 \sim k_2} \frac{| R (k_1) |  | R (k_2) |}{| k_1 |^2  | k_2 |^2} \lesssim 2^{q(\beta-2)} \varepsilon^{d +\beta + 2 - 4 \alpha} \|\tilde R \|_{\infty} \sum_{k_2 \in \varepsilon \mathbb{Z}^d_0} \mathbb{I}_{\varepsilon 2^q \lesssim k_2} \frac{ | R (k_2) |}{ | k_2 |^2} \\
     & \lesssim  (\varepsilon 2^q)^{\beta - 2} \varepsilon^{4 - 4 \alpha} \| \tilde R \|_{\infty} \int \mathd k \frac{| R (k) |}{| k |^2}
       \lesssim  (\varepsilon 2^q)^{\beta - 2} \varepsilon^{4 - 4 \alpha} \| \tilde R \|_{\infty} \sigma^2.
  \end{align*}
\end{proof}

This lemma shows that the interesting situation is $\alpha = 1$. Then provided
$\sigma^2 < + \infty$ and $\beta > 2$ we have $| \partial_x X^{\varepsilon} |^2
(t) \rightarrow \sigma^2$ in $L^2(\Omega; \CC^{0 -})$ for all $t > 0$, and in fact the
convergence is uniform for $t \in [c, C]$ whenever $0 < c < C$. Since all the operations that appear in the equation for $v^\varepsilon$ are continuous, it is then easy to see that $v^{\varepsilon}$ converges to the solution of the PDE
\begin{equation}
  \LL v = \sigma^2 v \label{eq:homog}
\end{equation}
and since $X^\varepsilon$ is a continuous linear functional of $V_\varepsilon$, we have $X^{\varepsilon} \rightarrow 0$ in $C_T \CC^{\gamma}$ and thus we finally obtain the convergence of $(u^{\varepsilon})_{\varepsilon > 0}$ to the same
$v$.

Thus, we have (modulo technical details) shown the following theorem:
\begin{theorem}\label{thm:homogenization}
   Let $\beta \in (0,d]$ and let $R = | \cdot |^{\beta - d} \tilde{R}$, where $\tilde{R} \in \CS (\mathbb{R}^d)$ is a smooth radial function of rapid decay, and assume that $\sigma^2 = ( \sqrt{2 \pi} )^d \int_{\mathbb{R}^d} R(k)/k^2 \mathd k < \infty$. Let $V \colon \mathbb{T}^d_{\varepsilon} \rightarrow \mathbb{R}$ be a continuous Gaussian function with mean zero and correlation
\[
   \mathbb{E} [V (x) V (y)] = C_{\varepsilon} (x - y) = ( \varepsilon / \sqrt{2 \pi} )^d \sum_{k \in \varepsilon \mathbb{Z}^d} e^{i \langle x - y, k \rangle} R (k).
\]
Consider the solution $u^{\varepsilon} : \mathbb{R}_+ \times
\mathbb{T}^d \rightarrow \mathbb{R}$ to the Cauchy problem
\[
   \partial_t u^{\varepsilon} (t, x) = \Delta u^{\varepsilon} (t, x) +
   \varepsilon^{- \alpha} V (x / \varepsilon) u^{\varepsilon} (t, x),
   \hspace{2em} u^{\varepsilon} (0) = u_0,
\]
where $u_0 \in C^{\infty}(\mathbb{T}^d)$. If $\alpha \in (0, 1 \wedge \beta/2)$, then $u^\varepsilon$ converges to the solution $u$ of
\[
   \partial_t u (t, x) = \Delta u (t, x), \hspace{2em} u (0) = u_0.
\]
However, if $1 = \alpha < \beta/2$, then $u^\varepsilon$ converges to the solution $v$ of
\[
   \partial_t v (t, x) = \Delta v (t, x) + \sigma^2 v(t,x), \hspace{2em} v(0) = u_0.
\]
\end{theorem}

\subsection{The 2d generalized parabolic Anderson model}

The case $\alpha = 1$ and $\beta = 2$ remains open in the previous analysis.  When
$\beta = 2$ we cannot expect $\sigma^2$ to be finite and moreover from the above
computations we see that the variance of $| \partial_x X^{\varepsilon} |^2$
remains finite and does not go to zero so the limiting object should satisfy a
stochastic PDE rather than a deterministic one. If we let
$\sigma^2_{\varepsilon} (t) =\mathbb{E} [| \partial_x X^{\varepsilon} |^2 (t,
x)]$ (which depends on time but which is easily shown to be independent of $x
\in \mathbb{T}^2$), then we expect that solving the {\tmem{renormalized}}
equation
\[ \LL \tilde{u}^{\varepsilon} = V_{\varepsilon} \tilde{u}^{\varepsilon} -
   \sigma^2_{\varepsilon} \tilde{u}^{\varepsilon} \]
should give rise in the limit to a well defined random field $\tilde{u}$
satisfying $\tilde{u} = e^X \tilde{v}$, where
\[ \LL \tilde{v} = \tilde{v} \zeta + 2 \langle \partial_x X, \partial_x
   \tilde{v} \rangle_{\mathbb{R}^d} \]
and where $X$ is the limit of $X^{\varepsilon}$ as $\varepsilon \rightarrow 0$
while $\zeta$ is the limit of $(\partial_x X^{\varepsilon})^2 -
\sigma^2_{\varepsilon} $. The relation of $u^{\varepsilon}$ with
$\tilde{u}^{\varepsilon}$ is $\tilde{u}^{\varepsilon} (t, x)
= e^{- \int_0^t \sigma^2_{\varepsilon} (s) \mathd s} u^{\varepsilon} (t, x)$.
The renormalization procedure is therefore equivalent to a time--dependent
rescaling of the solution to the initial problem. Without renormalization, the solution will simply drift of to $+\infty$, so in order to see a nontrivial behavior, we have to put ourselves in a different reference frame by multiplying with $ e^{- \int_0^t \sigma^2_{\varepsilon} (s) \mathd s}$. One familiar situation where such a need for renormalization arises is in the central limit theorem: If $(Y_n)$ is a sequence of i.i.d. random variables with unit variance and mean $\mu > 0$, then $(n^{-1/2} \sum_{k=1}^n Y_k)$ diverges to $+\infty$, but once we subtract the diverging constants $n^{1/2} \mu$ we get that $(n^{-1/2} \sum_{k=1}^n Y_k - n^{1/2}\mu)$ converges weakly to a standard Gaussian distribution.

We will study the renormalization and convergence problem for a more general
equation of the form
\begin{equation}
  \LL u^{\varepsilon} = F (u^{\varepsilon}) V_{\varepsilon}, \label{eq:gpam}
\end{equation}
where $F : \mathbb{R} \rightarrow \mathbb{R}$ is a sufficiently smooth
function, in general non--linear. One possible motivation is that if $z^{\varepsilon}$ solves
the linear PDE $\LL z^\varepsilon = z^\varepsilon V_\varepsilon$ and we set $u^{\varepsilon} = \varphi (z^{\varepsilon})$ for
some invertible $\varphi : \mathbb{R} \rightarrow \mathbb{R}$ such that
$\varphi' > 0$, then
\[ \LL u^{\varepsilon} = \varphi' (z^{\varepsilon}) \LL z^{\varepsilon} -
   \varphi'' (z^{\varepsilon}) | \partial_x z^{\varepsilon} |^2 = \varphi'
   (z^{\varepsilon}) z^{\varepsilon} V_{\varepsilon} - \varphi''
   (z^{\varepsilon}) (\varphi' (z^{\varepsilon}))^{- 2} | \partial_x
   u^{\varepsilon} |^2 \]
and thus $u^{\varepsilon}$ satisfies the PDE
\[ \LL u^{\varepsilon} = F_1 (u^{\varepsilon}) V_{\varepsilon} + F_2
   (u^{\varepsilon}) (\partial_x u^{\varepsilon})^2 \]
where $F_1 (x) = \varphi' (\varphi^{- 1} (x)) \varphi^{- 1} (x)$ and $F_2 (x)
= - \varphi'' (\varphi^{- 1} (x)) (\varphi' (\varphi^{- 1} (x)))^{- 2}$. In
the situation we are interested in, the second term in the right hand side is
simpler to treat than the first term so, for the time being, we will drop it
and we will concentrate on the equation~(\ref{eq:gpam}) in $d = 2$ with
$\alpha = 1$ and short ranged ($\beta = d$) potential $V$ which we refer to as
{\tmem{generalized parabolic Anderson model}} ({\tmname{gpam}}).

Under these conditions $V_{\varepsilon}$ converges to the white noise in
space which we usually denote with $\xi$ and our aim will be to set up a
theory in which the operations involved in the definition of the
dynamics of the {\tmname{gpam}} are well defined, including the possibility of
the renormalization which already appears in the linear case as hinted above.

While the reader should always have in mind a limiting procedure from a well
defined model like the ones we were considering so far, in the following we
will mostly discuss the limiting equation. The specific phenomena appearing
when trying to track the oscillations of the term $F (u^{\varepsilon})
V_{\varepsilon}$ as $\varepsilon \rightarrow 0$ will be described by a
{\tmem{renormalized product}} $F (u) \renorm \xi$ and so we write the
{\tmname{gpam}} as
\begin{equation}
  \label{eq:pam-intro} \LL u (t, x) = F (u (t, x)) \renorm \xi (x),
  \hspace{2em} u (0) = u_0 .
\end{equation}
In the linear case $F (u) = u$, the problem of the renormalization can be
solved along the lines suggested above. Another possible line of attack comes
from the theory of Gaussian spaces and in particular from Wick products, see
for example~{\cite{Hu2002}}. However, the definition of the Wick product
relies on the concrete chaos expansion of its factors, and since nonlinear
functions change the chaos expansion in a complicated way, there is little
hope of directly extending the Wick product approach to the nonlinear case and
moreover using these non--local (in the probability space) objects can deliver
solutions which are not physically acceptable~{\cite{Chan2000}}.

Equation~(\ref{eq:pam-intro}) is structurally very similar to the stochastic
differential equation
\begin{equation}
  \label{eq:sde-intro} \partial_t v (t) = F (v (t)) \partial_t B^H (t),
  \hspace{2em} v (0) = v_0,
\end{equation}
where $B^H$ denotes a fractional Brownian motion with Hurst index $H \in (0,
1)$. There are many ways to solve~(\ref{eq:sde-intro}) in the Brownian case.
Since we are interested in a way that might extend to~(\ref{eq:pam-intro})
where the irregularity appears along the two--dimensional spatial variable
$x$, we should exclude all approaches based on information, filtrations, and a
direction of time; in particular, any approach that works for $H \neq 1 / 2$
might seem promising. But Lyons' theory of rough paths~{\cite{Lyons1998}}
equips us exactly with the techniques we need to solve~(\ref{eq:sde-intro})
for general $H$. More precisely, if for $H > 1 / 3$ we are given
$\int_0^{\cdot} B^H_s \mathd B_s^H$, then we can use the controlled rough path
integral~{\cite{Gubinelli2004}} to make sense of $\int_0^{\cdot} f_s \mathd
B^H_s$ for any $f$ which ``looks like'' $B^H$, and this allows us to
solve~(\ref{eq:sde-intro}). So the main ingredients required for controlled
rough paths are the integral $\int_0^{\cdot} B^H_s \mathd B_s^H$ for the
reference path $B^H$, and the fact that we can describe paths which look like
$B^H$. It is worthwhile to note that while we need probability theory to
construct $\int_0^{\cdot} B^H_s \mathd B_s^H$, the construction of
$\int_0^{\cdot} f_s \mathd B^H_s$ is achieved using pathwise arguments and it
is given as a continuous map of $f$ and $( B^H, \int_0^{\cdot} B^H_s
\mathd B^H_s )$. As a consequence, the solution to the
SDE~(\ref{eq:sde-intro}) depends pathwise continuously on $( B^H,
\int_0^{\cdot} B^H_s \mathd B^H )$.

By the structural similarity of~(\ref{eq:pam-intro}) and~(\ref{eq:sde-intro}),
we might hope to extend the rough path approach to~(\ref{eq:pam-intro}). The
equivalent of $B^H$ is given by the solution $\vartheta$ to $\LL \vartheta =
\xi$, $\vartheta (0) = 0$, and the equivalent of $\int_0^{\cdot} B^H_s
\mathd B^H_s$ turns out to be the renormalized product
$\vartheta \renorm \xi$. Then we might hope that given $\vartheta \renorm \xi$ we are able to define $f \renorm
\xi$ for all $f$ that ``look like $\vartheta$'', however this is to be
interpreted. Of course, rough paths can only be applied to functions of a
one--dimensional index variable, while for~(\ref{eq:pam-intro}) the problem
lies in the irregularity of $\xi$ in the spatial variable $x \in
\mathbb{T}^2$.

In the following we combine the ideas from controlled rough paths with Bony's
paraproduct, a tool from functional analysis that allows us to extend rough paths to
functions of a multidimensional parameter. Using the paraproduct, we are able
to make precise in a simple way what we mean by ``distributions looking like a
reference distribution''. We can then define products of suitable
distributions and solve~(\ref{eq:pam-intro}) as well as many other interesting
singular SPDEs.

\subsection{More singular problems}\label{sec:rbe,kpz,she}

Keeping the homogenization problem as leitmotiv for these lectures, we could
consider also space--time varying environments $V_{\varepsilon} (t, x) =
\varepsilon^{- \alpha} V (t / \varepsilon^2, x / \varepsilon)$. The scaling of
the temporal variable is chosen so that it is compatible with the diffusive
scaling from a microscopic description, where $V (t, x)$ has typical variation
in space and time in scales of order $1$. Assume that $d = 1$, then when the
random field $V$ is Gaussian, zero mean, and with short--range space--time
correlations, the natural choice for the magnitude of the macroscopic
fluctuations is $\alpha = 3 / 2$. \ In this case $V_{\varepsilon}$ converges
as $\varepsilon \rightarrow 0$ to a space--time white noise $\xi$.
Understanding the limit dynamics as $\varepsilon \rightarrow 0$ of the
solution $u^{\varepsilon}$ to the linear equation $\LL u^{\varepsilon} =
V_{\varepsilon} u^{\varepsilon}$ represents now a more difficult problem than
in the time independent situation. A Gaussian computation shows that the
random field $X^{\varepsilon}$, solution to $\LL X^{\varepsilon} =
V_{\varepsilon}$ (e.g. with zero initial condition), stays bounded in $C_T
\CC^{1 / 2 -}$ as $\varepsilon \rightarrow 0$. Since $\LL$ is a second order
operator (if we use an appropriate parabolic weighting of the time and space
regularities), $\xi$ is expected to live in a space of distributions of
regularity $- 3 / 2 -$. This is to be compared with the $- 1 -$ of the space
white noise which had to be dealt with in the {\tmname{gpam}}. Renormalization
effects are then expected to be stronger in this setting and the limiting
object, which we denote with $w$, should satisfy a (suitably renormalized)
linear stochastic heat equation with multiplicative noise ({\tmname{she}})
\begin{equation}
  \label{eq:heat-intro} \LL w (t, x) = w (t, x) \renorm \xi (t, x),
  \hspace{2em} w (0) = w_0 .
\end{equation}
As indicated by the computations in the more regular case, it is useful to
consider the change of variables $w = e^h$ which is called Cole--Hopf
transformation. Here $h : [0, \infty) \times \mathbb{T} \rightarrow
\mathbb{R}$ is a new unknown which satisfies now the Kardar--Parisi--Zhang
({\tmname{kpz}}) equation:
\begin{equation}
  \label{eq:kpz-intro} \LL h (t, x) = (\partial_x h (t, x))^{\renorm 2} + \xi
  (t, x), \hspace{2em} h (0) = h_0
\end{equation}
where the difficulty comes now from the squaring of the derivative but which
has the nice feature to be additively perturbed by the space--time white
noise, a feature which simplifies many considerations. Another relevant model
in applications is obtained by taking the space derivative of {\tmname{kpz}}
and letting $u (t, x) = \partial_x h (t, x)$ in order to obtain the stochastic
conservation law
\begin{equation}
  \label{eq:burgers-intro} \LL u (t, x) = \partial_x (u (t, x))^{\renorm 2} +
  \partial_x \xi (t, x), \hspace{2em} u (0) = u_0,
\end{equation}
which we will refer to as the stochastic Burgers equation ({\tmname{sbe}}). In
all these cases, $\renorm$ denotes a suitably renormalized product.

The {\tmname{kpz}} equation was derived by Kardar--Parisi--Zhang in 1986 as a
universal model for the random growth of an interface~{\cite{Kardar1986}}. For
a long time it could not be solved due to the fact that there was no way to
make sense of the nonlinearity $(\partial_x h)^{\renorm 2}$
in~(\ref{eq:kpz-intro}). The only way to make sense of {\tmname{kpz}} was to
apply the Cole-Hopf transform~{\cite{Bertini1997}}: solve
{\tmname{she}}~(\ref{eq:heat-intro}) (which is accessible to It{\^o}
integration) and set $h = \log w$. But there was no intrinsic interpretation
of what it means to solve~(\ref{eq:kpz-intro}). Finally, in 2011
Hairer~{\cite{Hairer2013b}} used rough paths to give a meaning to the equation
and to obtain solutions directly at the {\tmname{kpz}} level. In
Section~\ref{sec:rbe} we will sketch how to recover his solution in the
paracontrolled setting. Applications of the techniques used by Hairer to solve
the {\tmname{kpz}} problem to a more general homogenization problem with
ergodic potentials (not necessarily Gaussian) have been studied
in~{\cite{hairer_random_2013}}.

\subsection{Hairer's regularity structures}

In~{\cite{Hairer2014}}, Hairer introduces a theory of regularity structures
which can also be considered a generalization of the theory of controlled
rough paths to functions of a multidimensional index variable. Hairer
fundamentally rethinks the notion of regularity. Usually a function is called
smooth if it can be approximated around every point by a polynomial of a given
degree (the Taylor polynomial). Naturally, the solution to an SPDE driven by
--say-- Gaussian space-time white noise is not smooth in that sense. So in
Hairer's theory, a function is called smooth if locally it can be approximated
by the noise (and higher order terms constructed from the noise). This induces
a natural topology in which the solutions to semilinear SPDEs depend
continuously on the driving signal.

At this date it seems that the theory of regularity structures has a wider
range of applicability than the paracontrolled approach described
in~{\cite{Gubinelli2012}}, but also at the expense of a very deep conceptual
sophistication. There are problems (like the one--dimensional heat equation
with multiplicative noise and general nonlinearity) that cannot be solved
using paracontrolled distributions, but these problems seem also quite
difficult (even if doable and there is work in progress) to tackle with regularity
structures. Moreover, equations of a more general kind, say dispersive
equations or wave equations, are still poorly (or not at all) understood in
both approaches.

\section{The paracontrolled PAM}

As we have tried to motivate in the previous sections we are looking for a
theory for {\tmname{pam}} which describes the possible limits of the equation
\begin{equation}
  \LL u = F (u) \eta \label{eq:pam-eta}
\end{equation}
driven by sufficiently regular $\eta$ but as $\eta$ is converging to the space
white noise $\xi$. From this point of view we are looking for a priori
estimates on the solution $u$ to~(\ref{eq:pam-eta}) which depend only on
distributional norms of $\eta$. So in the following we will assume that we
have at hand only a uniform control of $\eta$ in $C_T \CC^{\gamma - 2}$ for
some $\gamma > 0$. For the application to the 2d space white noise we could
take $\gamma = 1 -$, but we will not use this specific information in order to
probe the range of applicability of our approach and we will only assume that
the exponent $\gamma$ is such that $3 \gamma - 2 > 0$.

\

Assume for a moment that we are in the simpler situation $\gamma > 1$ and
$u_0 \in \CC^{\gamma}$ and let us try to solve equation~(\ref{eq:pam-eta}) via
Picard iterations $(u^n)_{n \geqslant 0}$ starting from $u^0 \equiv u_0$.
Since $F$ preserves the $C \CC^{\gamma}$-regularity (which can be seen by
identifying $C \CC^{\gamma}$ with the classical space of bounded
H{\"o}lder--continuous functions of space), the product $F (u^0 (t)) \eta$ is
well defined as an element of $\CC^{\gamma - 2}$ for all $t \geqslant 0$ since
$2 \gamma - 2 > 0$ and we are in condition to apply
Corollary~\ref{cor:product} below on the product of elements in
H{\"o}lder--Besov spaces. Now by Lemma~\ref{lem:schauder}, the heat semigroup generated by the Laplacian gains
two degrees of regularity so that the solution $u^1$ to $\LL u^1 = F (u^0)
\eta$, $u^1 (0) = u_0$, is in $C \CC^{\gamma}$. From here we obtain a
contraction on $C_T \CC^{\gamma}$ for some small $T > 0$ whose value does not
depend on $u_0$, which gives us global in time existence and uniqueness of
solutions. Note that in one dimension the space white noise has regularity
$\CC^{- 1 / 2 -}$ (see Exercise~\ref{exo:wn besov}) so taking $\gamma = 3 / 2
-$ we have determined that the one--dimensional {\tmname{pam}} can be solved
globally in time with standard techniques.

\

When the condition $2 \gamma - 2 > 0$ is not satisfied we still have that if
$\eta \in C_T \CC^{\gamma - 2}$ then $u \in \LL^{\gamma} = C_T \CC^{\gamma -
2} \cap C_T^{\gamma / 2} L^{\infty}$ by the standard parabolic estimates of
Lemma~\ref{lem:schauder}. However with the regularities at hand we cannot use
Corollary~\ref{cor:product} anymore to guarantee the continuity of the
operator $(u, \eta) \mapsto F (u) \eta$. Moreover, as already seen in the
simpler homogenization problems of Theorem~\ref{thm:homogenization} above this is not a technical difficulty but a
real issue of the regime $\gamma \leqslant 1$. We expect that controlling the
model in this regime can be quite tricky since limits exists when $\eta
\rightarrow 0$ but the limiting solution still feels residual order one
effects from the vanishing driving signal $\eta$. This situation cannot be
improved from the point of view of standard analytic considerations. What is
needed is a finer control of the solution $u$ which allows to analyse in more
detail the possible resonances between the fluctuations of $u$ and those of
$\eta$.

\

Before going on we will revise the problem of multiplication of distributions
in the scale of H{\"o}lder--Besov spaces, introducing the basic tool of our
general analysis: Bony's paraproduct.

\subsection{The paraproduct and the resonant term}\label{ssec:bony}

Paraproducts are bilinear operations introduced by Bony~{\cite{Bony1981}} to
linearize a class of nonlinear hyperbolic PDEs in order to analyse the
regularity of their solutions. In terms of Littlewood--Paley blocks, a general
product $fg$ of two distributions can be (at least formally) decomposed as
\[ fg = \sum_{j \geqslant - 1} \sum_{i \geqslant - 1} \Delta_i f \Delta_j g =
   f \lpara g + f \rpara g + f \reso g. \]
Here $f \lpara g$ is the part of the double sum with $i < j - 1$, $f \rpara g$
is the part with $i > j + 1$, and $f \reso g$ is the ``diagonal'' part, where
$|i - j| \leqslant 1$. More precisely, we define
\[ f \lpara g = g \rpara f = \sum_{j \geqslant - 1} \sum_{i = - 1}^{j - 2}
   \Delta_i f \Delta_j g \hspace{2em} \text{and} \hspace{2em} f \reso g =
   \sum_{|i - j| \leqslant 1} \Delta_i f \Delta_j g. \]
Of course, the decomposition depends on the dyadic partition of unity used to
define the blocks $\Delta_j$, and also on the particular choice of the pairs
$(i, j)$ in the diagonal part. The choice of taking all $(i, j)$ with $|i - j|
\leqslant 1$ into the diagonal part corresponds to the fact that the partition
of unity can be chosen such that $\tmop{supp} \CF (\Delta_i f \Delta_j g)
\subseteq 2^j \CA$ if $i < j - 1$, where $\CA$ is a suitable annulus. If $| i
- j | \leqslant 1$, the only apriori information on the spectral support of
the various term in the double sum is $\tmop{supp} \CF (\Delta_i f \Delta_j g)
\subseteq 2^j \CB$, that is they are supported in balls and in particular they
can have non--zero contributions to very low wave vectors. \ We call $f \lpara
g$ and $f \rpara g$ \tmtextit{paraproducts}, and $f \reso g$ the
\tmtextit{resonant} term.

Bony's crucial observation is that $f \lpara g$ (and thus $f \rpara g$) is
always a well-defined distribution. Heuristically, $f \lpara g$ behaves at
large frequencies like $g$ (and thus retains the same regularity), and $f$
provides only a frequency modulation of $g$. The only difficulty in
constructing $fg$ for arbitrary distributions lies in handling the diagonal
term $f \reso g$. The basic result about these bilinear operations is given by
the following estimates.

\begin{theorem}
  (Paraproduct estimates)\label{thm:paraproduct} For any $\beta \in
  \mathbb{R}$ and $f, g \in \CS'$ we have
  \begin{equation}
    \label{eq:para-1} \|f \lpara g\|_{\beta} \lesssim_{\beta}
    \|f\|_{L^{\infty}} \|g\|_{\beta},
  \end{equation}
  and for $\alpha < 0$ furthermore
  \begin{equation}
    \label{eq:para-2} \|f \lpara g\|_{\alpha + \beta} \lesssim_{\alpha, \beta}
    \|f\|_{\alpha} \|g\|_{\beta} .
  \end{equation}
  For $\alpha + \beta > 0$ we have
  \begin{equation}
    \label{eq:para-3} \|f \reso g\|_{\alpha + \beta} \lesssim_{\alpha, \beta}
    \|f\|_{\alpha} \|g\|_{\beta} .
  \end{equation}
\end{theorem}

\begin{proof}
  There exists an annulus $\CA$ such that $S_{j - 1} f \Delta_j g$ has Fourier
  transform supported in $2^j \CA$, and for $f \in L^{\infty}$ we have
  \[ \|S_{j - 1} f \Delta_j g\|_{L^{\infty}} \leqslant \|S_{j - 1}
     f\|_{L^{\infty}}  \| \Delta_j g\|_{L^{\infty}} \lesssim
     \|f\|_{L^{\infty}} 2^{- j \beta} \|g\|_{\beta} . \]
  By Lemma~\ref{lem: Besov regularity of series}, we thus
  obtain~(\ref{eq:para-1}). The proof of~(\ref{eq:para-2})
  and~(\ref{eq:para-3}) works in the same way, where for estimating $f \reso
  g$ we need $\alpha + \beta > 0$ because the terms of the series are
  supported in a ball and not in an annulus.
\end{proof}

In combination with Exercise~\ref{exercise:besov-space-inequalities} above, we deduce the following simple corollary:

\begin{corollary}
  \label{cor:product}Let $f \in \CC^{\alpha}$ and $g \in \CC^{\beta}$ with
  $\alpha + \beta > 0$, then the product $(f, g) \mapsto fg$ is a bounded
  bilinear map from $\CC^{\alpha} \times \CC^{\beta}$ to $\CC^{\alpha \wedge
  \beta}$. While $f \lpara g$, $f \rpara g$, and $f \reso g$ depend on the
  specific dyadic partition of unity, the product $f g$ does not.
\end{corollary}

The independence of the product from the dyadic partition of unity easily
follows by taking smooth approximations.

\

The ill--posedness of $f \reso g$ for $\alpha + \beta \leqslant 0$ can be
interpreted as a resonance effect since $f \reso g$ contains exactly those
part of the double series where $f$ and $g$ are in the same frequency range.
The paraproduct $f \lpara g$ can be interpreted as frequency modulation of $g$,
which should become more clear in the following example.

\begin{example}
  In Figure~\ref{fig:slow} we see a slowly oscillating positive function $u$,
  while Figure~\ref{fig:fast} depicts a fast sine curve $v$. The product $u
  v$, which here equals the paraproduct $u \lpara v$ since $u$ has no rapidly
  oscillating components, is shown in Figure~\ref{fig:modulation}. We see that
  the local fluctuations of $u v$ are due to $v$, and that $u v$ is
  essentially oscillating with the same speed as $v$.

\begin{figure}[h]
\centering
\begin{minipage}{0.45\textwidth}
\centering
 {\resizebox{5.5cm}{1.5cm}{\includegraphics{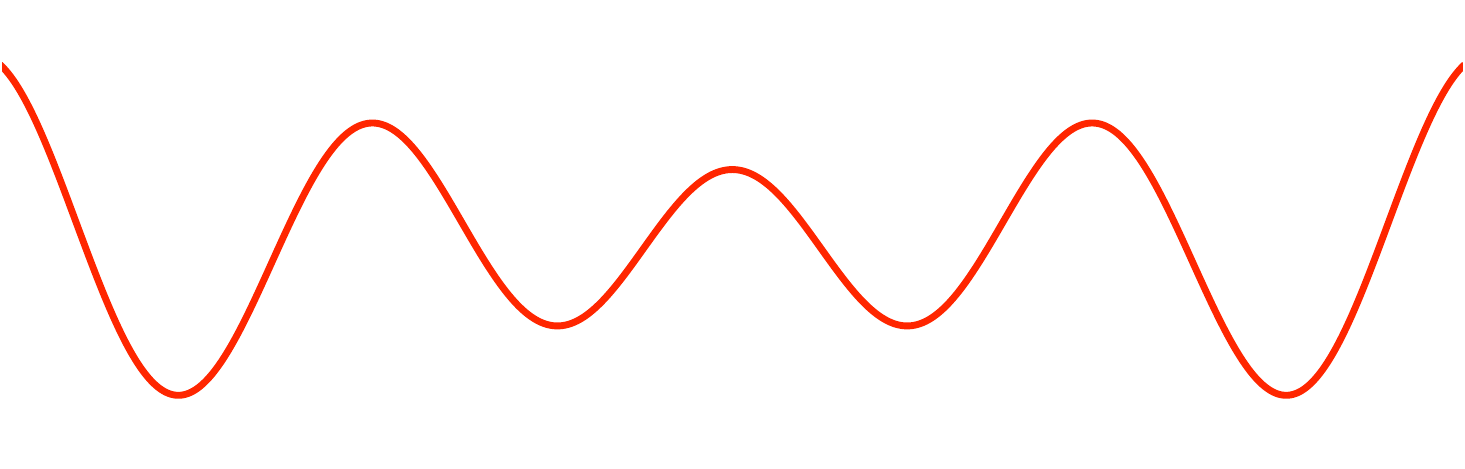}}}
\caption{The function $u$}
 \label{fig:slow}
\end{minipage}\hfill
\begin{minipage}{0.45\textwidth}
\centering
 {\resizebox{5.5cm}{1.5cm}{\includegraphics{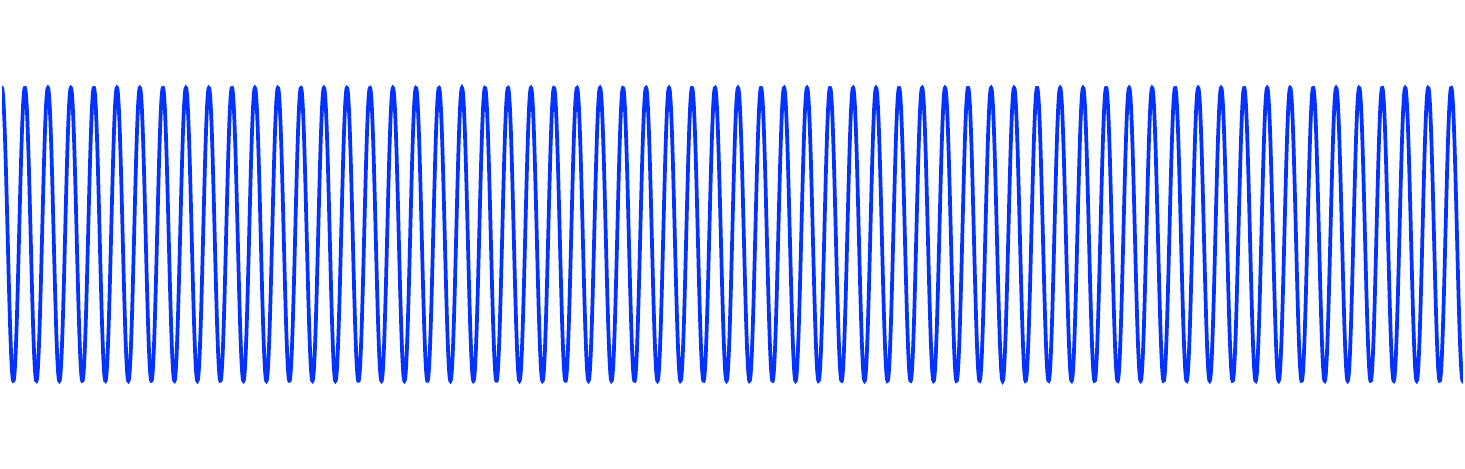}}}
\caption{The function $v$}
\label{fig:fast}
\end{minipage}\\
\begin{minipage}{0.45\textwidth}
\centering
 {\resizebox{5.5cm}{1.5cm}{\includegraphics{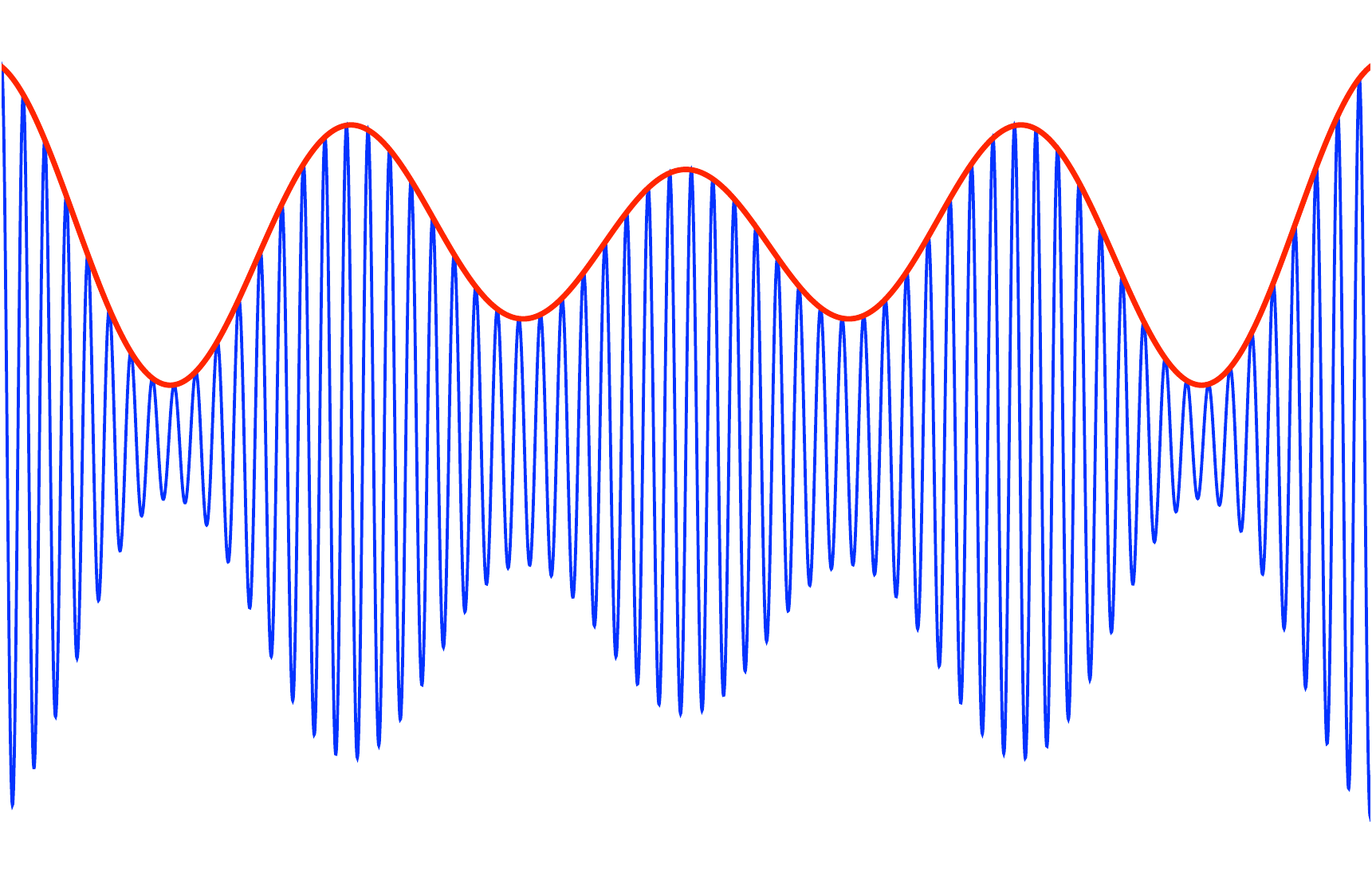}}}
\caption{The function $u \lpara v$}
\label{fig:modulation}
\end{minipage}
\end{figure}

\end{example}

\begin{example}
  If $f \in \CC^{\gamma} (\mathbb{T})$ and $g \in \CC^{\delta} (\mathbb{T})$
  with $\gamma + \delta > 1$, then we can define $\int f \mathd g \assign \int
  (f \partial_t g)$, which is well defined since $\partial_t g \in \CC^{\delta
  - 1}$ and $\gamma + \delta - 1 > 0$, and since integration is a linear map.
  In this way we recover the Young integral~{\cite{Young1936}}.
\end{example}

\begin{example}
  Let $B^H$ be a fractional Brownian bridge on $\mathbb{T}$ (or simply a
  fractional Brownian motion on $[0, \pi]$, reflected on $[\pi, 2 \pi]$) and
  assume that $H > 1 / 2$. We have $\varphi (B^H) \in \CC^{H -}$ for all
  Lipschitz continuous $\varphi$, and $\partial_t B^H \in \CC^{(H - 1) -}$,
  and in particular $\varphi (B^H) \partial_t B^H$ is well-defined. This can
  be used to solve SDEs driven by $B^H$ in a pathwise sense.
\end{example}

The condition $\alpha + \beta > 0$ is essentially sharp, at least at this
level of generality, see~{\cite{Young1936}} for counterexamples. It excludes
of course the Brownian case: if $B$ is a Brownian motion, then almost surely
$B \in \CC^{\alpha}_{\tmop{loc}}$ for all $\alpha < 1 / 2$ (meaning that $\varphi B \in \CC^\alpha$ whenever $\varphi $ is a smooth compactly supported function), so that
$\partial_t B \in \CC^{\alpha - 1}_{\tmop{loc}}$ and thus $B \reso \partial_t
B$ fails to be well defined. See also~{\cite{Lyons2007}}, Proposition~1.29 for
an instructive example which shows that this is not a shortcoming of our
description of regularity, but that it is indeed impossible to define the
product $B \partial_t B$ as a continuous bilinear operation on distribution
spaces.

\

Other counterexamples are given by our discussion of the homogenization
problem in Theorem~\ref{thm:homogenization} above. More simply, one can consider the following situation.

\begin{example}
  Consider the sequence of functions $f_n : \mathbb{T} \rightarrow
  \mathbb{C}$ given by$f_n (x) = e^{in^2 x} / n$. Then it is easy to show
  that $\| f_n \|_{\gamma} \rightarrow 0$ for all $\gamma < 1 / 2$. However
  let
  \[
     g_n (x) = \tmop{Re} f_n (x) \tmop{Im} \partial_x f_n (x) = (\cos (n^2 x))^2 = \frac{\cos(2 n^2 x) + 1}{2}
  \]
  Then $g_n \rightarrow 1 / 2$ in $\CC^{0 -}$ which shows that the map
  $f \mapsto (\tmop{Re} f)  (\partial_x \tmop{Im} f)$ cannot be continuous in
  $\CC^{\gamma}$ if $\gamma < 1 / 2$. Pictorially the situation is summarized
  in Figure~\ref{fig:resonances}, where we sketched the three dimensional
  curve given by $x \mapsto ( \tmop{Re} f_n (x) \nocomma, \tmop{Im} f_n
  (x), \int_0^x g_n (y) \mathd y )$ for various values of $n$ and in the
  limit.
  \begin{figure}[h]
  \centering
    \resizebox{5.5cm}{1.5cm}{\includegraphics{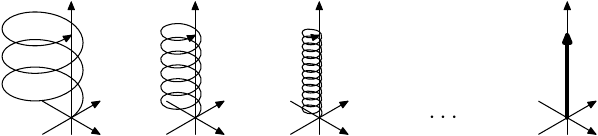}}
    \caption{\label{fig:resonances}Resonances give macroscopic effects}
  \end{figure}
\end{example}

\subsection{Commutator estimates and paralinearization}

The product $F (u) \eta$ appearing in the right hand side of {\tmname{pam}}
can be decomposed via the paraproduct $\lpara$ as a sum of three terms
\[ F (u) \eta = F (u) \lpara \eta + F (u) \reso \eta + F (u) \rpara \eta . \]
The first and the last of these terms are continuous in any topology we will choose
for $F (u)$ and $\eta$. The {\tmem{resonant term}} $F (u) \reso \eta$ however
is problematic. It gathers the products of the oscillations of $F (u)$ and
$\eta$ on comparable dyadic scales and these products can contribute to all
larger scales in such a way that microscopic oscillations might build up to a
macroscopic effect which does not disappear in the limit (as we  have already
seen in Theorem~\ref{thm:homogenization}). If the function $F$ is smooth enough, then we expect the resonances
between $F (u)$ and $\eta$ to correspond to the resonances between $u$ and
$\eta$, and as we will see this is justified.

\

The expected regularity of the different terms is
\begin{equation}
  \label{eq:para-rhs-intro} \underbrace{F (u) \lpara \eta}_{\gamma - 2} +
  \underbrace{F (u) \reso \eta}_{2 \gamma - 2} + \underbrace{F (u) \rpara
  \eta}_{2 \gamma - 2},
\end{equation}
but unless $2 \gamma - 2 > 0$ the resonant term $F (u) \reso \eta$ cannot be
controlled using only the $C \CC^{\gamma}$--norm of $u$ and the $C \CC^{\gamma
- 2}$--norm of $\eta$. If $F$ is at least $C^2$, we can use a
{\tmem{paralinearization}} result (stated precisely in
Lemma~\ref{lemma:para-taylor} below) to rewrite this term as
\begin{equation}
  \label{eq:rde expand Fu} F (u) \reso \eta = F' (u)  (u \reso \eta) + \Pi_F
  (u, \eta),
\end{equation}
with a remainder $\Pi_F (u, \eta) \in \CC^{2 \gamma - 2}$ provided $3 \gamma -
2 > 0$. The difficulty is now localized in the linearized resonant product $u
\reso \eta$. In order to control this term, we would like to exploit the fact
that the function $u$ is not a generic element of $C \CC^{\gamma}$ but that it
has a specific structure, since $\LL u$ has to match the paraproduct
decomposition given in~(\ref{eq:para-rhs-intro}) where the least regular term
is expected to be $F (u) \lpara \eta \in C \CC^{\gamma - 2}$.

In order to do so, we postulate that the solution $u$ is given by the
following \tmtextit{paracontrolled ansatz:}
\[ u = u^X \lpara X + u^{\sharp}, \]
for functions $u^X, X, u^{\sharp}$ such that $u^X, X \in C \CC^{\gamma}$ and the
remainder $u^{\sharp} \in C \CC^{2 \gamma}$. This decomposition allows for a
finer analysis of the resonant term $u \reso \eta$: indeed, we have
\begin{equation}
  \label{eq:rde expand resonant term} u \reso \eta = (u^X \lpara X) \reso \eta
  + u^{\sharp} \reso \eta = u^X  (X \reso \eta) + C (u^X, X, \eta) +
  u^{\sharp} \reso \eta,
\end{equation}
where the {\tmem{commutator}} is defined by $C (u^X, X, \eta) = (u^X \lpara X)
\reso \eta - u^X  (X \reso \eta)$. Observe now that the term $u^{\sharp} \reso
\eta$ does not pose any further problem, as it can be controlled in $C \CC^{3 \gamma
- 2}$. The key point is now that the commutator is a bounded multilinear
function of its arguments as long as the sum of their regularities is strictly
positive, see Lemma~\ref{lemma:commutator} below. By assumption, we have $3
\gamma - 2 > 0$, and therefore $C (u^X, X, \eta) \in C \CC^{2 \gamma - 2}$.

The only problematic term which remains to be handled is thus the bilinear
functional of the noise given by $X \reso \eta$. Here we need to make the
assumption that $X \reso \eta \in C \CC^{2 \gamma - 2}$ in order for the
product $u^X  (X \reso \eta)$ to be well defined. This assumption is not
guaranteed by the analytical estimates at hand, and it has to be added as a
further requirement to our construction.

Granting this last step, we have obtained that the right hand side of
equation~(\ref{eq:pam-eta}) is well defined and a continuous function of $(u,
u^X, u^{\sharp}, X, \eta, X \reso \eta) \in C\CC^\gamma \times C \CC^\gamma \times C \CC^{2\gamma} \times C \CC^\gamma \times C \CC^{\gamma-2} \times C \CC^{2\gamma-2}$.

\

It remains to check that the paracontrolled ansatz is coherent with the
equation satisfied by solutions to {\tmname{pam}}. Let us first consider the
linear example $F (u) = u$. Here we saw that the solution is of the form $u =
e^X v$ with
\[ \LL v = v | \partial_x X |^2 + 2 \langle \partial_x v, \partial_x X
   \rangle_{\mathbb{R}^2}, \]
where $| \partial_x X |^2 \in C \CC^{2 \gamma - 2}$ by Lemma~\ref{lem:Dx squared reg} and $\partial_x X \in C
\CC^{\gamma - 1}$ and therefore $v \in C \CC^{2 \gamma}$ by the Bony and Schauder estimates. Note that here we have a clash of notation, because a priori the $X$ that we defined in Section~\ref{sec:homogenization} does not have to be equal to the paracontrolling distribution $X$. But of course, as the notation suggests we will see momentarily that we can choose them to be the same. In the setting of Section~\ref{sec:homogenization}, we have in particular
\[ u = e^X v = v \lpara e^X + C \CC^{2 \gamma} = v \lpara (e^X \lpara X) + C
   \CC^{2 \gamma}, \]
   where the notation $u = v \lpara e^X + C \CC^{2\gamma}$ means that $u - v \lpara e^X \in C \CC^{2\gamma}$, and where we used a paralinearization result in last step (see
Lemma~\ref{lemma:paralinearization} below). Now the double paraproduct $f
\lpara (g \lpara h)$ satisfies
\[ \| f \lpara (g \lpara h) - (f g) \lpara h \|_{\alpha + \beta} \lesssim \| f
   \|_{\alpha} \| g \|_{\alpha} \| h \|_{\beta}, \]
see~{\cite{Bony1981}}, and therefore $u = (v e^X) \lpara X + C \CC^{2 \gamma} =
u \lpara X + C \CC^{2 \gamma}$ which shows that the paracontrolled ansatz is at
least justified in the linear case and indeed we can choose the paracontrolling distribution to be $X$.

\

In the nonlinear case, the paracontrolled ansatz and the Leibniz rule for the
paraproduct imply that~(\ref{eq:pam-eta}) can be rewritten as
\[
   \LL u = \LL (u^X \lpara X + u^{\sharp}) = u^X \lpara \LL X + [ \LL, u^X \lpara ] X + \LL u^{\sharp} = F (u) \lpara \eta + F (u) \reso \eta + F(u) \rpara \eta,
\]
where we recall that $[ \LL, u^X \lpara ] X = \LL (u^X \lpara X) -
u^X \lpara \LL X$ denotes the commutator. If we choose $X$ such that $\LL X =
\eta$ and we set $u^X = F (u)$, then we can use~(\ref{eq:rde expand Fu})
and~(\ref{eq:rde expand resonant term}) to obtain the following equation for
the remainder $u^{\sharp}$:
\begin{equation}
  \begin{array}{ll}
    \LL u^{\sharp} & = F' (u) F (u)  (X \reso \eta) + F (u) \rpara \eta -
    [ \LL, F (u) \lpara ] X\\
    & \hspace{2em} + F' (u) C (F (u), X, \eta) + F' (u)  (u^{\sharp} \reso
    \eta) + \Pi_F (u, \eta) .
  \end{array} \label{eq:u-sharp}
\end{equation}
Lemma~\ref{lemma:comm-L-para} below ensures that $J [ \LL, F (u)
\lpara ] X \in C \CC^{2 \gamma}$ whenever $F(u) \in \LL^{\gamma}$ (which easily follows from $u \in \LL^\gamma$ by using the increment characterization of $\CC^\gamma$ regularity), and combining the paraproduct estimates with the estimates for $C$ and $\Pi_F$ that we discussed above, we see that all the other terms on the right hand side are in $C
\CC^{2 \gamma - 2}$. So the Schauder estimate Lemma~\ref{lem:schauder} allows us to control $u^{\sharp}$ in $C
\CC^{2 \gamma}$. Together with $u = F (u) \lpara X + u^{\sharp}$,
equation~(\ref{eq:u-sharp}) gives an equivalent description of the solution, because we only rewrote the original problem. This allows us to obtain a priori estimates on $u$ and $u^{\sharp}$ in terms of
$(u_0, \| \eta \|_{\gamma - 2}, \|X \reso \eta \|_{2 \gamma - 2})$, see Chapter~5 of~\cite{Gubinelli2012} for details. It is now
straightforward to show that if $F \in C^3_b$, then $u$ depends continuously on the data
$(u_0, \eta, X \reso \eta)$, so that we have a robust strategy to pass to the
limit in~(\ref{eq:gpam}) and to make sense of the solution
to~(\ref{eq:pam-eta}) also for irregular $\eta \in C \CC^{\gamma - 2}$ as long
as $\gamma > 2 / 3$.

\

In the remainder of this section we will prove the results (paralinearization
and various key commutators) which we used in the discussion above, before
going on to gather the consequences of our analysis in the next section. When
the time dependence does not play any role we state the results for
distributions depending only on the space variable as the extension to time
varying functions will not add further difficulty.

\begin{lemma}
  \label{lemma:commutator}Assume that $\alpha, \beta, \gamma \in \mathbb{R}$
  are such that $\alpha + \beta + \gamma > 0$ and $\beta + \gamma \neq 0$.
  Then for $f, g, h \in C^{\infty}$ the trilinear operator
  \[ C (f, g, h) = ((f \lpara g) \reso h) - f (g \reso h) \]
  satisfies
  \begin{equation}
    \label{eq:commutator bound} \|C (f, g, h)\|_{\beta + \gamma} \lesssim
    \|f\|_{\alpha} \|g\|_{\beta} \|h\|_{\gamma},
  \end{equation}
  and can thus be uniquely extended to a bounded trilinear operator from
  $\CC^{\alpha} \hspace{-0.17em} \times \hspace{-0.17em} \CC^{\beta}
  \hspace{-0.17em} \times \hspace{-0.17em} \CC^{\alpha}$ to $\CC^{\beta +
  \gamma}$.
\end{lemma}

\begin{proof}
  For $\beta + \gamma > 0$ this follows from the paraproduct estimates, so let
  $\beta + \gamma < 0$. By definition
  \begin{align*}
     C (f, g, h) & = \sum_{i, j, k, \ell} \Delta_i (\Delta_j f \Delta_k g) \Delta_{\ell} h (\mathbb{I}_{j < k - 1} \mathbb{I}_{| i - \ell | \leqslant 1} - \mathbb{I}_{| k - \ell | \leqslant 1}) \\
      & = \sum_{i, j, k, \ell} \Delta_i (\Delta_j f \Delta_k g) \Delta_{\ell} h (\mathbb{I}_{j < k - 1} \mathbb{I}_{| i - \ell | \leqslant 1}\mathbb{I}_{| k - \ell | \leqslant N} - \mathbb{I}_{| k - \ell |\leqslant 1}),
  \end{align*}
  where we used that $\CF(S_{k - 1} f \Delta_k g)$ has support in an annulus $2^k
  \CA$, so that $\Delta_i (S_{k - 1} f \Delta_k g) \neq 0$ only if $| i - k |
  \leqslant N - 1$ for some fixed $N \in \mathbb{N}$, which in combination
  with $| i - \ell | \leqslant 1$ yields $| k - \ell | \leqslant N$. Now the assumptions on our partition of unity guarantee that for
  fixed $k$, the term $\sum_{\ell} \mathbb{I}_{2 \leqslant | k - \ell |
  \leqslant N} \Delta_k g \Delta_{\ell} h$ is spectrally supported in an
  annulus $2^k \CA$, so that $\sum_{k, \ell} \mathbb{I}_{2 \leqslant | k -
  \ell | \leqslant N} \Delta_k g \Delta_{\ell} h \in \CC^{\beta + \gamma}$ and
  we may add and subtract $f \sum_{k, \ell} \mathbb{I}_{2 \leqslant | k -
  \ell | \leqslant N} \Delta_k g \Delta_{\ell} h$ to $C (f, g, h)$ while
  maintaining the bound~(\ref{eq:commutator bound}). It remains to treat
  \begin{align}\label{eq:commutator pr1} \nonumber
     &\sum_{i, j, k, \ell} \Delta_i  (\Delta_j f \Delta_k g) \Delta_{\ell} h \mathbb{I}_{| k - \ell | \leqslant N} (\mathbb{I}_{j < k - 1} \mathbb{I}_{| i - \ell | \leqslant 1} - 1) \\
     &\hspace{100pt} = - \sum_{i, j, k, \ell} \Delta_i  (\Delta_j f \Delta_k g) \Delta_{\ell} h \mathbb{I}_{| k - \ell | \leqslant N}(\mathbb{I}_{j \geqslant k - 1} + \mathbb{I}_{j < k - 1}\mathbb{I}_{| i - \ell | > 1}) .
  \end{align}
  We estimate both terms on the right hand side separately. For $m \geqslant - 1$ we have (recall that for indices of Littlewood--Paley blocks, $i \lesssim j$ is to be read as $2^i \lesssim 2^j$, that is $i \le j + c$ for some fixed $c$):
  \begin{align*}
     &\Big\| \Delta_m \Big( \sum_{i, j, k, \ell} \Delta_i  (\Delta_j f \Delta_k g) \Delta_{\ell} h \mathbb{I}_{| k - \ell | \leqslant N} \mathbb{I}_{j \geqslant k - 1} \Big) \Big\|_{L^{\infty}} \\
     &\hspace{35pt} \leqslant \sum_{j, k, \ell} \mathbb{I}_{| k - \ell | \leqslant N} \mathbb{I}_{j \geqslant k - 1} \| \Delta_m (\Delta_j f \Delta_k g \Delta_{\ell} h) \|_{L^{\infty}} \lesssim \sum_{j \gtrsim m} \sum_{k \lesssim j} 2^{- j \alpha} \| f \|_{\alpha} 2^{- k \beta} \| g \|_{\beta} 2^{- k \gamma} \| h \|_{\gamma} \\
     &\hspace{35pt} \lesssim \sum_{j \gtrsim m} 2^{- j (\alpha + \beta + \gamma)} \| f \|_{\alpha} \| g \|_{\beta} \| h \|_{\gamma} \lesssim 2^{- m (\alpha +  \beta + \gamma)} \| f \|_{\alpha} \| g \|_{\beta} \| h \|_{\gamma},
  \end{align*}
  using $\beta + \gamma < 0$ to get $\sum_{k\lesssim j}2^{k(\beta+\gamma)} \lesssim 2^{j(\alpha+\beta)}$. It remains to estimate the second term
  in~(\ref{eq:commutator pr1}). For $| i - \ell | > 1$ and $i \sim k \sim
  \ell$, any term of the form $\Delta_i (\cdot) \Delta_{\ell} (\cdot)$ is spectrally
  supported in an annulus $2^{\ell} \CA$, and therefore
  \begin{align*}
      & \Big\| \Delta_m \Big( \sum_{i, j, k, \ell} \Delta_i  (\Delta_j f \Delta_k g) \Delta_{\ell} h \mathbb{I}_{| k - \ell | \leqslant N} \mathbb{I}_{j < k - 1} \mathbb{I}_{| i - \ell | > 1} \Big) \Big\|_{L^{\infty}} \\
      &\hspace{80pt} \lesssim \sum_{i, j, k, \ell} \mathbb{I}_{j < k - 1} \mathbb{I}_{i \sim k \sim \ell \sim m} \|\Delta_i  (\Delta_j f \Delta_k g) \Delta_{\ell} h \|_{L^{\infty}} \\
      &\hspace{80pt} \lesssim \sum_{j \lesssim m} 2^{- j \alpha} \| f \|_{\alpha} 2^{- m \beta} \| g \|_{\beta} 2^{- m \gamma} \| h \|_{\gamma} \lesssim 2^{- m(\beta + \gamma)} \| f \|_{\alpha} \| g \|_{\beta} \| h \|_{\gamma}.
  \end{align*}
  
\end{proof}

\begin{remark}
  For $\beta + \gamma = 0$ we can apply the commutator estimate with $\gamma'
  < \gamma$, as long as $\alpha + \beta + \gamma' > 0$.
\end{remark}

Our next result is a simple paralinearization lemma for non--linear operators.

\begin{lemma}[see also~{\cite{Bahouri2011}}, Theorem 2.92]
  \label{lemma:paralinearization}Let $\alpha \in (0, 1)$, $\beta \in (0,
  \alpha]$, and let $F \in C^{1 + \beta / \alpha}_b$. There exists a locally
  bounded map $R_F : \CC^{\alpha} \rightarrow \CC^{\alpha + \beta}$ such that
  \begin{equation}
    \label{eq:para-linearization} F (f) = F' (f) \lpara f + R_F (f)
  \end{equation}
  for all $f \in \CC^{\alpha}$. More precisely, we have
  \[ \|R_F (f)\|_{\alpha + \beta} \lesssim \|F\|_{C^{1 + \beta / \alpha}_b} 
     (1 +\|f\|_{\alpha}^{1 + \beta / \alpha}) . \]
  If $F \in C_b^{2 + \beta / \alpha}$, then $R_F$ is locally Lipschitz
  continuous:
  \[ \|R_F (f) - R_F (g)\|_{\alpha + \beta} \lesssim \|F\|_{C^{2 + \beta /
     \alpha}_b}  (1 +\|f\|_{\alpha} +\|g\|_{\alpha})^{1 + \beta / \alpha}  \|f
     - g\|_{\alpha} . \]
\end{lemma}

\begin{remark}
   Since every element of $\CC^\alpha$ is bounded, the result immediately extends to unbounded $F \in C^{1+\beta/\alpha}$: Simply replace $F$ by an element of $C^{1+\beta/\alpha}_b$ which agrees with $F$ on the image of $f$.
\end{remark}

\begin{proof}[Proof of Lemma~\ref{lemma:paralinearization}]
  The difference $F (f) - F' (f) \lpara f$ is given by
  \[ R_F (f) = F (f) - F' (f) \lpara f = \sum_{i \geqslant - 1} [\Delta_i F (f)
     - S_{i - 1} F' (f) \Delta_i f] = \sum_{i \geqslant - 1} u_i, \]
  and every $u_i$ is spectrally supported in a ball $2^i \CB$. For $i < 1$, we
  simply estimate $\|u_i \|_{L^{\infty}} \lesssim \|F\|_{C^1_b}  (1
  +\|f\|_{\alpha})$. For $i \geqslant 1$ we use the fact that $f$ is a bounded
  function to write the Littlewood--Paley projections as convolutions and
  obtain
  \begin{align*}
    u_i (x) & = \int K_i  (x - y) K_{< i - 1}  (x - z)  [F (f (y)) - F' (f
    (z)) f (y)] \mathd y \mathd z\\
    & = \int K_i  (x - y) K_{< i - 1}  (x - z)  [F (f (y)) - F (f (z)) - F'
    (f (z)) (f (y) - f (z))] \mathd y \mathd z,
  \end{align*}
  where $K_i = \CF^{- 1} \rho_i$, $K_{< i - 1} = \sum_{j < i - 1} K_j$, and
  where we used that $\int K_i (y) \mathd y = \rho_i (0) = 0$ for $i \geqslant
  0$ and $\int K_{< i - 1} (z) \mathd z = 1$ for $i \geqslant 1$. Now we can
  apply a first order Taylor expansion to $F$ and use the $\beta /
  \alpha$--H{\"o}lder continuity of $F'$ in combination with the
  $\alpha$--H{\"o}lder continuity of $f$, to deduce
  \begin{align*}
    |u_i (x) | & \lesssim \|F\|_{C^{1 + \beta / \alpha}_b} \|f\|_{\alpha}^{1 +
    \beta / \alpha}  \int |K_i (x - y) K_{< 0} (x - z) | \times |z -
    y|^{\alpha + \beta} \mathd y \mathd z\\
    & = \|F\|_{C^{1 + \beta / \alpha}_b} \|f\|_{\alpha}^{1 + \beta /
    \alpha} 2^{- (i-1) (\alpha + \beta)} \\
    &\qquad \times \int |2^{(i-1)d}K_1 (2^{i-1}(x - y)) 2^{(i-1)d}K_{<0} (2^{i-1}(x - z)) | \times |2^{i-1}(z -  y)|^{\alpha + \beta} \mathd y \mathd z \\
    & \lesssim \|F\|_{C^{1 + \beta / \alpha}_b} \|f\|_{\alpha}^{1 + \beta /
    \alpha} 2^{- i (\alpha + \beta)} .
  \end{align*}
  Therefore, the estimate for $R_F (f)$ follows from Lemma~\ref{lem: Besov
  regularity of series}. The estimate for $R_F (f) - R_F (g)$ is shown in the
  same way.
\end{proof}

Let $g$ be a distribution belonging to $\CC^{\beta}$ for some $\beta < 0$.
Then the map $f \mapsto f \reso g$ behaves, modulo smoother correction terms,
like a derivative operator:

\begin{lemma}
  \label{lemma:para-taylor}Let $\alpha \in (0, 1)$, $\beta \in (0, \alpha]$,
  $\gamma \in \mathbb{R}$ be such that $\alpha + \beta + \gamma > 0$ and
  $\alpha + \gamma \neq 0$. Let $F \in C^{1 + \beta / \alpha}_b$. Then there
  exists a locally bounded map $\Pi_F : \CC^{\alpha} \times \CC^{\gamma}
  \rightarrow \CC^{\alpha + \gamma}$ such that
  \begin{equation}
    \label{eq:para-taylor} F (f) \reso g = F' (f)  (f \reso g) + \Pi_F (f, g)
  \end{equation}
  for all $f \in \CC^{\alpha}$ and all smooth $g$. More precisely, we have
  \[ \| \Pi_F (f, g)\|_{\alpha + \gamma} \lesssim \|F\|_{C^{1 + \beta /
     \alpha}_b}  (1 +\|f\|_{\alpha}^{1 + \beta / \alpha}) \|g\|_{\gamma} . \]
  If $F \in C^{2 + \beta / \alpha}_b$, then $\Pi_F$ is locally Lipschitz
  continuous:
  \begin{align*}
     &\| \Pi_F (f, g) - \Pi_F (u, v)\|_{\alpha + \gamma} \\
     &\hspace{50pt} \lesssim \|F\|_{C^{2 + \beta / \alpha}_b}  (1 +\|f\|_{\alpha} +\|u\|_{\alpha})^{1 + \beta / \alpha} (1 +\|v\|_{\gamma})  (\|f - u\|_{\alpha} +\|g - v\|_{\gamma}).
  \end{align*}
\end{lemma}

\begin{proof}
  Use the paralinearization and commutator lemmas above to deduce that
  \begin{align*}
    \Pi_F (f, g) & = F (f) \reso g - F' (f)  (f \reso g) = R_F (f) \reso g + (F'
    (f) \lpara f) \reso g - F' (f)  (f \reso g)\\
    & = R_F (f) \reso g + C (F' (f), f, g),
  \end{align*}
  so that the claimed bounds easily follow from Lemma~\ref{lemma:commutator}
  and Lemma~\ref{lemma:paralinearization}.
\end{proof}

Besides this sort of chain rule, we also have a Leibniz rule for $f \mapsto f
\reso g$:

\begin{lemma}
  \label{lem:para-taylor product}Let $\alpha \in (0, 1)$ and $\gamma < 0$ be
  such that $2 \alpha + \gamma > 0$ and $\alpha + \gamma \neq 0$. Then there
  exists a bounded trilinear operator $\Pi_{\times} : \CC^{\alpha} \times
  \CC^{\alpha} \times \CC^{\gamma} \rightarrow \CC^{\alpha + \gamma}$, such
  that
  \[ (fu) \reso g = f (u \reso g) + u (f \reso g) + \Pi_{\times} (f, u, g) \]
  for all $f, u \in \CC^{\alpha} (\mathbb{R})$ and all smooth $g$.
\end{lemma}

\begin{proof}
  It suffices to note that $fu = f \lpara u + f \rpara u + f \reso u$, which
  leads to
  \[ \Pi_{\times} (f, u, g) = (fu) \reso g - f (u \reso g) - u (f \reso g) = C
     (f, u, g) + C (u, f, g) + (f \reso u) \reso g. \]
\end{proof}

\begin{lemma}
  \label{lemma:comm-L-para}Let $\beta < 1$, $\alpha \in \mathbb{R}$, and let
  $f \in \LL^{\beta}$ and $G \in C \CC^{\alpha}$ with $\LL G \in C \CC^{\alpha
  - 2}$. There exists $H = H (f, G)$ such that $\LL H = [ \LL, f \lpara
  ] G$ and $H (0) = 0$. Moreover $H \in C \CC^{\alpha + \beta} \cap
  C^{(\alpha \wedge \beta) / 2} L^{\infty}$ and for all $T > 0$
  \[ \| H \|_{C^{(\alpha \wedge \beta) / 2}_T L^{\infty}} + \| H \|_{C_T
     \CC^{\alpha + \beta}} \lesssim \| f \|_{\LL^{\beta}_T} ( \| G
     \|_{C_T \CC^{\alpha}} + \left\| \LL G \right\|_{C_T \CC^{\alpha - 2}}) . \]
\end{lemma}

\begin{proof}
  Let $T > 0$ and let $f_{\varepsilon}$ be a time mollification of $f$ such
  that $\| \partial_t f_{\varepsilon} \|_{C_T L^{\infty}} \lesssim
  \varepsilon^{\beta / 2 - 1} \| f \|_{\LL^{\beta}_T}$ and $\| f_{\varepsilon}
  - f \|_{C_T L^{\infty}} \lesssim \varepsilon^{\beta / 2} \| f
  \|_{\LL^{\beta}}$ for all $\varepsilon > 0$. For example we can take
  $f_{\varepsilon} = \rho_{\varepsilon} \ast f$ with $\rho_{\varepsilon} (t) =
  \rho (t / \varepsilon) / \varepsilon$ and $\rho : \mathbb{R} \rightarrow
  \mathbb{R}$ compactly supported, smooth, and of unit integral. For $i
  \geqslant - 1$ we have
  \[ \LL \Delta_i H = \Delta_i  \LL H = \Delta_i \left[ \LL ((f - f_{\varepsilon}) \lpara G) - (f
     - f_{\varepsilon}) \lpara \LL G \right] + \Delta_i \left[ \LL
     (f_{\varepsilon} \lpara G) - f_{\varepsilon} \lpara \LL G \right], \]
  so that
  \begin{align*}
      \LL \Delta_i (H - (f - f_{\varepsilon}) \lpara G) & = - \Delta_i \left[ (f - f_{\varepsilon}) \lpara \LL G \right] + \Delta_i \left[ \LL(f_{\varepsilon} \lpara G) - f_{\varepsilon} \lpara \LL G \right] \\
      & = \Delta_i \left[ (f_{\varepsilon} - f) \lpara \LL G \right] + \Delta_i \left[ \LL f_{\varepsilon} \lpara G - 2 \partial_x f_{\varepsilon} \lpara \partial_x G \right],
  \end{align*}
  with initial condition $\Delta_i (H - (f - f_{\varepsilon}) \lpara G) (0) = -
  (\Delta_i (f - f_{\varepsilon}) \lpara G) (0)$. The Schauder estimates for
  $\LL$ (Lemma~\ref{lem:schauder}) give
  \begin{align*}
      &\| \Delta_i (H + (f - f_{\varepsilon}) \lpara G) \|_{\LL^{\alpha \upl \beta}_T} \\
      &\hspace{60pt} \lesssim \left\| \Delta_i \left[ (f - f_{\varepsilon}) \lpara \LL G \right] + \Delta_i \left[ \left( \LL f_{\varepsilon} \right) \lpara G - 2 \partial_x f_{\varepsilon} \lpara \partial_x G \right] \right\|_{C_T \CC^{\alpha + \beta - 2}} \\
      &\hspace{60pt} \qquad + \| (\Delta_i (f - f_{\varepsilon}) \lpara G) (0) \|_{\alpha + \beta} .
  \end{align*}
  Choosing $\varepsilon = 2^{- 2 i}$, we have
  \begin{align*}
     \| \Delta_i ((f - f_{\varepsilon}) \lpara G) \|_{C_T \CC^{\alpha + \beta}} & \lesssim 2^{\beta i} \| \Delta_i ((f - f_{\varepsilon}) \lpara G) \|_{C_T \CC^{\alpha}} \lesssim 2^{\beta i} \| f - f_{\varepsilon} \|_{C_T L^{\infty}} \| G \|_{C_T \CC^{\alpha}} \\
     & \lesssim \| f \|_{\LL_T^{\beta}} \| G \|_{C_T \CC^{\alpha}}
  \end{align*}
  and exactly the same argument also gives
  \[ \left\| \Delta_i \left[ (f - f_{\varepsilon}) \lpara \LL G \right]
     \right\|_{C_T \CC^{\alpha + \beta - 2}} \lesssim \| f \|_{\LL_T^{\beta}}
     \left\| \LL G \right\|_{C_T \CC^{\alpha - 2}} . \]
  Since $\beta < 1$, we further get
  \begin{align*}
     \left\| \Delta_i \left[ \LL f_{\varepsilon} \lpara G + \partial_x f_{\varepsilon} \lpara \partial_x G \right] \right\|_{C_T \CC^{\alpha + \beta - 2}} & \lesssim 2^{i (\beta - 2)} \| \partial_t f_{\varepsilon} \|_{C_T L^{\infty}} \| G \|_{C_T \CC^{\alpha}} + \| f_{\varepsilon}  \|_{C_T \CC^{\beta}} \| G \|_{C_T \CC^{\alpha}} \\
     & \lesssim \| f \|_{\LL_T^{\beta}} \| G \|_{C_T \CC^{\alpha}} + \| f \|_{C_T \CC^{\beta}} \| G \|_{C_T \CC^{\alpha}}.
  \end{align*}
  Combining everything, we end up with
  \[ \| \Delta_i H \|_{C_T \CC^{\alpha + \beta}} \lesssim \| f
     \|_{\LL^{\beta}_T} ( \| G \|_{C_T \CC^{\alpha}} + \left\| \LL G
     \right\|_{C_T \CC^{\alpha - 2}} ), \]
  which gives the estimate for the space regularity of $H$ since $\| \Delta_i
  H \|_{C_T L^{\infty}} \lesssim 2^{- (\alpha + \beta) i} \| \Delta_i H
  \|_{C_T \CC^{\alpha + \beta}}$. The time regularity of $H$ can be controlled
  similarly by noting that $(f - f_{\varepsilon}) \lpara G \in C^{(\alpha
  \wedge \beta) / 2}_T L^{\infty}$, uniformly in $\varepsilon$. 
\end{proof}

\subsection{Paracontrolled distributions}

Here we build a calculus of distributions satisfying a paracontrolled ansatz.
We start by defining a suitable space of such objects.

\begin{definition}
  \label{def:paracontrolled parabolic}Let $\alpha > 0$ and $\beta \in (0,
  \alpha]$ be such that $\alpha + \beta \in (0, 2)$, and let $u \in
  \LL^{\alpha}$. A pair of distributions $(f, f^u) \in \LL^{\alpha} \times
  \LL^{\beta}$ is called \tmtextit{paracontrolled} by $u$ if $f^{\sharp} = f -
  f^u \lpara u \in C \CC^{\alpha + \beta} \cap \LL^{\beta}$. In that case we
  write $f \in \CD^{\beta} = \CD^{\beta} (u)$, and for all $T > 0$ we define
  the norm
  \[ \| f \|_{\DD_T^{\beta}} = \| f \|_{C^{\alpha / 2}_T} + \| f^u
     \|_{\LL^{\beta}_T} + \| f^{\sharp} \|_{C_T \CC^{\alpha + \beta}} + \|
     f^{\sharp} \|_{C_T^{\beta / 2} L^{\infty}} . \]
  If $\tilde{u} \in \LL^{\alpha}$ and $(\tilde{f}, \tilde{f}^{\tilde{u}}) \in
  \DD^{\beta} (\tilde{u}) \nocomma$, then we also write
  \[ d_{\DD^{\beta}_T} (f, \tilde{f}) = \| f^u - \tilde{f}^{\tilde{u}}
     \|_{\LL^{\beta}_T} + \|f^{\sharp} - \tilde{f}^{\sharp} \|_{C_T
     \CC^{\alpha + \beta}} + \|f^{\sharp} - \tilde{f}^{\sharp} \|_{C_T^{\beta
     / 2} L^{\infty}} . \]
     Note that in general $f$ and $\tilde f$ do not live on the same space, so $d_{\DD^{\beta}_T}$ is not a distance.
\end{definition}

Of course we should really write $(f, f^u) \in \CD^{\beta}$ since given $f$
and $u$, the derivative $f^u$ is usually not uniquely determined. But in the
applications there will always be an obvious candidate for the derivative, and
no confusion will arise.

\begin{remark}
  The space $\CD^{\beta}$ does not depend on the specific dyadic partition of
  unity. Indeed, Bony~{\cite{Bony1981}} has shown that if $\tilde{\lpara}$ is
  the paraproduct constructed from another partition of unity, then $\| f^u
  \lpara u - f^u \tilde{\lpara} u \|_{C_T \CC^{\alpha + \beta}} \lesssim \| f^u
  \|_{C_T \CC^{\beta}} \| u \|_{C_T \CC^{\alpha}}$.
\end{remark}

\paragraph{Nonlinear operations}As an immediate consequence of
Lemma~\ref{lemma:commutator} we can multiply any distribution that is paracontrolled by $u$ with a given $v$,
provided that we know how to multiply $u$ with $v$ (of course always under suitable regularity assumptions):

\begin{theorem}[also see Theorem~3.7 of~\cite{Gubinelli2012}]\label{thm:paracontrolled product}
   Let $\alpha, \beta \in \mathbb{R}$,
  $\gamma < 0$, with $\alpha + \beta + \gamma > 0$ and $\alpha + \gamma \neq
  0$. Let $u \in C \CC^{\alpha}$, $v \in C \CC^{\gamma}$, and let $\zeta \in C
  \CC^{\alpha + \gamma}$. Then
  \[ \DD^{\beta} (u) \ni f \mapsto f \cdummy v \assign f \lpara v + f \rpara v +
     f^{\sharp} \reso v + C (f^u, u, v) + f^u \zeta \in C \CC^{\gamma} \]
  defines a bounded linear operator and for all $T > 0$ we have the bound
  \[ \| (f v)^{\sharp} \|_{C_T \CC^{\alpha + \gamma}} \assign \| f \cdummy v -
     f \lpara v \|_{C_T \CC^{\alpha + \gamma}} \lesssim \| f \|_{\CD^{\beta}_T}
     \left( \| v \|_{C_T \CC^{\gamma}} + \| u \|_{C_T \CC^{\alpha}} \| v
     \|_{C_T \CC^{\gamma}} + \| \zeta \|_{C_T \CC^{\alpha + \gamma}} \right) .
  \]
  If there exist sequences of smooth functions $(u_n)$ and $(v_n)$ converging
  to $u$ and $v$ in $C \CC^{\alpha}$ and  $C \CC^{\gamma}$ respectively for which $(u_n \reso v_n)$ converges to $\zeta$ in $C \CC^{\alpha+\gamma}$,
  then $f \cdummy v$ does not depend on the dyadic partition of unity used to
  construct it.
  
  Furthermore, there exists a quadratic polynomial $P$ so that if $\tilde{u},
  \tilde{v}, \tilde{\zeta}$ satisfy the same assumptions as $u$, $v$, $\zeta$
  respectively, if $\tilde{f} \in \CD^{\beta} (\tilde{u})$, and if
  \[ M = \max \left\{ \| u \|_{C_T \CC^{\alpha}}, \| v \|_{C_T \CC^{\gamma}},
     \| \zeta \|_{C_T \CC^{\alpha + \gamma}}, \| \tilde{u} \|_{C_T
     \CC^{\alpha}}, \| \tilde{v} \|_{C_T \CC^{\gamma}}, \| \tilde{\zeta}
     \|_{C_T \CC^{\alpha + \gamma}}, \|f\|_{\DD^{\beta}_T (u)}, \| \tilde{f}
     \|_{\DD^{\beta}_T (\tilde{u})} \right\}, \]
  then
  \[ \| (f v)^{\sharp} - (\tilde{f} \tilde{v})^{\sharp} \|_{C_T\CC^{\alpha + \gamma}}
     \leqslant P (M) \left( d_{\DD^{\beta}} (f, \tilde{f}) +\|u - \tilde{u}
     \|_{C_T\CC^\alpha} +\|v - \tilde{v} \|_{C_T\CC^\gamma} +\| \zeta - \tilde{\zeta}
     \|_{C_T\CC^{\alpha + \gamma}} \right) . \]
\end{theorem}

\begin{proof}
   Given Lemma~\ref{lemma:commutator} (and the paraproduct estimates Theorem~\ref{thm:paraproduct}), the proof is straightforward and we leave most of it as an exercise. Let us only comment on the independence of the partition of unity: Let $(u^n, v^n)$ be as announced and define $f_n \assign f^u \lpara u_n + f^\sharp$. Then
   \begin{align*}
      \lim_{n\to \infty} f_n v_n & = \lim_{n \to \infty} \big( f_n \lpara v_n + f_n \rpara v_n +  f^{\sharp} \reso v_n + C (f^u, u_n, v_n) + f^u (u_n \circ v_n) \big) \\
      & =  f \lpara v + f \rpara v +  f^{\sharp} \reso v + C (f^u, u, v) + f^u \zeta = f\cdot v.
   \end{align*}
   Since the pointwise product $f_n v_n$ does not depend on the partition of unity, also the limit must be independent.
   
   The bound on the difference is obtained by using the boundedness and multilinearity of all operators involved.
\end{proof}

From now on we will assume that there exist smooth functions $(u_n)$ and $(v_n)$ converging to $u$ and $v$ respectively for which $(u_n \reso v_n)$ converges to $\zeta$, so that the product does not depend on the partition of unity, and we will usually write $f v$ rather than $f \cdummy v$. Later we will see that the resonant term $(u_n \circ v_n)$ must often be renormalized by subtracting a large constant, but this will not affect the independence of the product from the partition of unity.

To solve equations involving general nonlinear functions, we need to examine
the stability of paracontrolled distributions under smooth functions.

\begin{theorem}\label{thm:paralinearization}
   Let $\alpha \in (0, 1)$ and $\beta \in (0,
  \alpha]$. Let $u \in \LL^{\alpha}$, $f \in \DD^{\alpha} (u)$, and $F \in
  C^{1 + \beta / \alpha}_b$. Then $F (f) \in \DD^{\beta}$ with derivative $(F
  (f))^u = F' (f) f^u$, and for all $T > 0$
  \[
     \| F (f) \|_{\DD^{\beta}_T} \lesssim \| F \|_{C^{1 + \beta / \alpha}_b} ( 1 + \| f \|^2_{\DD^{\alpha}_T} ) ( 1 + \| u  \|_{\LL_T^{\alpha}}^2 ) .
  \]
  Moreover, there exists a polynomial $P$ which satisfies for all $F \in C^{2
  + \beta / \alpha}_b$, $\tilde{u} \in \LL^{\alpha} \nocomma$, $\tilde{f} \in
  \DD^{\alpha} (\tilde{u})$, and
  \[ M \assign \max \left\{ \| u \|_{\LL^{\alpha}_T}, \| \tilde{u}
     \|_{\LL^{\alpha}_T}, \|f\|_{\DD^{\alpha}_T (u)}, \| \tilde{f}
     \|_{\DD^{\alpha}_T (\tilde{u})} \right\} \]
  the bound
  \[ d_{\DD^{\beta}_T} (F (f), F (\tilde{f})) \leqslant P (M) \| F \|_{C^{2 +
     \beta / \alpha}_T} ( d_{\DD^{\alpha}_T} (f, \tilde{f}) + \| u -
     \tilde{u} \|_{\LL^{\alpha}_T} ) . \]
\end{theorem}

The proof is not very complicated but rather lengthy, and we do not present it
here. The reader can find it in~\cite{Gubinelli2012}.

\paragraph{Schauder estimate for paracontrolled distributions}The Schauder
estimate Lemma~\ref{lem:schauder} is not quite sufficient: we also need to
understand how the heat kernel acts on the paracontrolled structure.

\begin{theorem}\label{thm:schauder paracontrolled}Let $\alpha \in (0, 1)$ and $\beta \in
  (0, \alpha]$. Let $u \in C \CC^{\alpha - 2}$ and $\LL U = u$ with $U (0) =
  0$. Let $f^u \in \LL^{\beta}$, $f^{\sharp} \in C \CC^{\alpha + \beta - 2}$,
  and $g_0 \in \CC^{\alpha + \beta}$. Then $(g, f^u) \in \CD^{\beta} (U)$,
  where $g$ solves
  \[ \LL g = f^u \lpara u + f^{\sharp}, \hspace{2em} g (0) = g_0, \]
  and we have the bound
  \[ \| g \|_{\CD^{\beta}_T (U)} \lesssim \| g_0 \|_{\alpha + \beta} + (1 + T)
     (\| f^u \|_{\LL^{\beta}_T} (1 + \| u \|_{C_T \CC^{\alpha - 2}}) + \|
     f^{\sharp} \|_{C_T \CC^{\alpha + \beta - 2}}) \]
  for all $T > 0$. If furthermore $\tilde{u}, \tilde{U}, \tilde{f}^{\tilde u},
  \tilde{f}^{\sharp}, \tilde{g}_0, \tilde{g}$ satisfy the same assumptions as
  $u, U, f^u, f^{\sharp}, g_0, g$ respectively, and if $M = \max \{ \| f^u
  \|_{\LL^{\beta}_T}, \| \tilde{u} \|_{C_T \CC^{\alpha - 2}}, 1\}$, then
  \begin{equation*}
    d_{\DD^{\beta}_T} (g, \tilde{g}) \lesssim \| g_0 - \tilde{g}_0
    \|_{\alpha + \beta} 
     + (1 + T) M (\|f^u - \tilde{f}^{\tilde u} \|_{\LL^{\beta}_T} +\|u - \tilde{u}
    \|_{C_T \CC^{\alpha - 2}} +\|f^{\sharp} - \tilde{f}^{\sharp} \|_{C_T
    \CC^{\alpha + \beta - 2}}) . 
  \end{equation*}
\end{theorem}

\begin{proof}
  Let us derive an equation for the remainder $g^{\sharp}$. We have
  \begin{align*}
    \LL g^{\sharp} & = \LL g - \LL (f' \lpara U) =  [f^u \lpara u + f^{\sharp}] - f^u \lpara \LL U - [ \LL (f^u \lpara U) - f^u \lpara \LL U ] \\
    & = f^{\sharp} -  [ \LL, f^u \lpara ] U .
  \end{align*}
  Since $\alpha \wedge \beta = \beta$ we can now apply Lemma~\ref{lemma:comm-L-para} to see that there exists $H \in C
  \CC^{\alpha + \beta} \cap C^{\beta / 2} L^{\infty}$ such that $\LL H =
  \left[ \LL, f^u \lpara \right] U$, so we can apply the standard Schauder
  estimates of Lemma~\ref{lem:schauder} to $\LL (g^{\sharp} + H) = 
  f^{\sharp}$ to get
  \[ \| g^{\sharp} \|_{C_T \CC^{\alpha + \beta}} + \| g^{\sharp}
     \|_{C_T^{(\alpha \wedge \beta) / 2} L^{\infty}} \lesssim \| f^u
     \|_{\LL^{\beta}_T} ( \| U \|_{C_T \CC^{\alpha}} + \| \LL U
     \|_{C_T \CC^{\alpha - 2}} ) + \| f^{\sharp} \|_{C_T
     \CC^{\alpha + \beta - 2}} . \]
  The estimate for $g^{\sharp} - \tilde{g}^{\sharp}$ can be derived in the same way.
\end{proof}

\tmtextbf{Bibliographic notes.} Paraproducts were introduced
in~{\cite{Bony1981}}, for a nice introduction see~{\cite{Bahouri2011}}. The
commutator estimate Lemma~\ref{lemma:commutator} is
from~{\cite{Gubinelli2012}}, but the proof here is new and the statement is
slightly different. In~{\cite{Gubinelli2012}}, we require the additional
assumption $\alpha \in (0, 1)$ under which $C$ maps $\CC^{\alpha} \times
\CC^{\beta} \times \CC^{\gamma}$ to $\CC^{\alpha + \beta + \gamma}$ and not
only to $\CC^{\beta + \gamma}$. Theorem~\ref{thm:paralinearization} is
from~{\cite{Gubinelli2012}}.

Theorem~\ref{thm:schauder paracontrolled} is new, but it is implicitly used
in~{\cite{Gubinelli2012}}. The estimates presented here will only allow us to
consider regular initial conditions. More general situations can be covered by
working on ``explosive spaces'' of the type
\[ \Big\{f \in C \left( (0, \infty), \CC^{\alpha} \right) : \sup_{t \in (0, T]} \|
   t^{- \gamma} f (t) \|_{\CC^{\alpha}} < \infty \tmop{for} \tmop{all} T > 0 \Big\}
\]
and similar for the temporal regularity. This is also done
in~{\cite{Gubinelli2012}}.

Of course it is easily possible to replace the Laplacian by more general
pseudo-differential operators. We only used two properties of $\Delta$: the
fact that $\Delta (f' \lpara U) - f' \lpara (\Delta U)$ is relatively regular, and that the
semigroup generated by $\Delta$ has a sufficiently strong regularization
effect. This is also true for fractional Laplacians and more generally for a
wide range of pseudo-differential operators.

\subsection{Fixpoint}\label{sec:fixpoint}

Let us now give the details for the solution to {\tmname{pam}} in the space of
paracontrolled distributions. Assume that $F : \mathbb{R} \rightarrow
\mathbb{R}$ is in $C^{1 + \varepsilon}_b$ for some $\varepsilon > 0$ such
that $(2 + \varepsilon) \gamma > 2$.

Let $Y \in C\CC^{\gamma}$ and let $u \in \CD^{\gamma} (Y)$. We will see below how to choose $Y$, for the moment it is an arbitrary $C \CC^\gamma$ function. From Theorem~\ref{thm:paralinearization} we know that $F (u) \in \CD^{\varepsilon\gamma} (Y)$:
\begin{equation}
  \label{eq:fixpoint map1} \CD^{\gamma} (Y) \xrightarrow{u \mapsto F (u)}
  \CD^{\varepsilon \gamma} (Y) .
\end{equation}
Assume now that $Y \reso \eta \in C \CC^{2 \gamma - 2}$ is given -- note that for the regularity assumptions we made, $Y \reso \eta$ is \emph{not a continuous functional} of $Y$ and $\eta$ but must be controlled using other means, say stochastic computations! Under this assumption, 
Theorem~\ref{thm:paracontrolled product} applied with $u = Y$, $v = \eta$, and $\zeta = Y \circ \eta$ shows that for all $f \in \CD^{\varepsilon \gamma}(Y)$ we have $f \eta = (f \eta)^{\sharp} + f \lpara
\eta$ with $(f \eta)^{\sharp} \in C \CC^{2 \gamma - 2}$ -- it is here that
we use $(2 + \varepsilon) \gamma > 2$. Integrating against the heat kernel and
assuming that $u_0 \in \CC^{2 \gamma}$, we obtain from
Theorem~\ref{thm:schauder paracontrolled} (with $u = \eta$, $f^u = f$, $f^\sharp = (f \eta)^\sharp$) that the solution $(J (f \eta) (t) + P_t u_0)_{t \geqslant 0}$ to $\LL (J (f \eta) + P_{\cdummy} u_0) = f \eta$, $J
(f \eta) (0) + P_0 u_0 = u_0$, is in $\CD^{\gamma} (X)$, where $X$ solves $\LL
X = \eta$ and $X (0) = 0$. In other words, we have a map
\begin{equation}
  \label{eq:fixpoint map2} \CD^{\varepsilon \gamma} (Y) \xrightarrow{f \mapsto
  P_\cdot u_0 + J (f \eta)} \CD^{\gamma} (X),
\end{equation}
and combining~(\ref{eq:fixpoint map1}) and (\ref{eq:fixpoint map2}) we get
\[ \CD^{\gamma} (Y) \xrightarrow{u \mapsto F (u)} \CD^{\varepsilon \gamma} (Y)
   \xrightarrow{F (u) \mapsto P_\cdot u_0 + J (F (u) \eta)} \CD^{\gamma} (X), \]
so that for all $T > 0$ we can define
\[ \Gamma_T : \CD_T^{\gamma} (Y) \rightarrow \CD_T^{\gamma} (X),
   \hspace{2em} \Gamma_T (u) = (P_\cdot u_0 + J (F (u) \eta)) |_{[0, T]} . \]
To set up a Picard iteration domain and image space should coincide which
means we should take $Y = X$. Refining the analysis, we obtain a scaling
factor $T^{\delta}$ \label{page:scaling factor} when estimating the $\CD_T^{\gamma} (X)$--norm of
$\Gamma_T (u)$. This allows us to show that for small $T > 0$, the map
$\Gamma_T$ leaves suitable balls in $\CD^{\gamma}_T (X)$ invariant, and
therefore we obtain the (local in time) \tmtextit{existence} of solutions to
the equation under the assumption $X \reso \eta \in C \CC^{2 \gamma - 2}$.

To obtain \tmtextit{uniqueness} we need to suppose that $F \in C^{2 +
\varepsilon}_b$. In that case Theorem~\ref{thm:paralinearization} gives the
local Lipschitz continuity of the map $u \mapsto F (u)$ from $\CD_T^{\gamma}
(X)$ to $\CD^{\varepsilon \gamma}_T (X)$, while
Theorem~\ref{thm:paracontrolled product} and Theorem~\ref{thm:schauder
paracontrolled} show that $f \mapsto u_0 + J (f \eta)$ defines a Lipschitz
continuous map from $\CD^{\varepsilon \gamma}_T (X)$ to $\CD^{\gamma}_T (X)$.
Again we can obtain a scaling factor $T^{\delta}$, so that $\Gamma_T$ defines
a contraction on a suitable ball of $\CD^{\gamma}_T (X)$ for some small $T >
0$.

Even better, $\Gamma_T$ not only depends locally Lipschitz continuously on
$u$, but also on the extended data $(u_0, \eta \comma X \reso \eta)$, and
therefore the solution to~(\ref{eq:pam-eta}) depends locally Lipschitz
continuously on $ (u_0, \eta, X \reso \eta)$.

\subsection{Renormalization}

So far we argued under the assumption that $X \reso \eta$ exists and has a
sufficient regularity. This should be understood via approximations as the
existence of a sequence of smooth functions $(\eta_n)$ that converges to
$\eta$, such that $(X_n \reso \eta_n)$ converges to $X \reso \eta$. However, as we will see below 
this hypothesis is questionable and actually not satisfied at all in the problem we are interested in. More
concretely, recall that we would like to take $\eta = \xi$ to be the two--dimensional space white noise. If then
$\varphi$ is a Schwartz function on $\mathbb{R}^2$ and if $\varphi_n = n
\varphi (n \cdummy)$ and
\[ \eta_n (x) = \varphi_n \ast \xi (x) = \int_{\mathbb{R}^2} \varphi_n (x -
   y) \xi (y) \mathd y = \sum_{k \in \mathbb{Z}^2} \langle \xi, \varphi_n (x
   + 2 \pi k - \cdummy) \rangle, \]
then we will see below that there exist constants $(c_n)$ with $\lim_n c_n =
\infty$, such that $(X_n \reso \eta_n - c_n)$ converges in $C_T \CC^{2 \gamma
- 2}$ for all $T > 0$.

\

This is not a problem with our specific approximation. The homogenization
setting shows that even for $\eta \rightarrow 0$ \ there are cases where the
limiting equation is nontrivial. In the paracontrolled setting we have a
continuous dependence of the solution on the data $(\eta, X \reso \eta)$, so this
non--triviality of the limit can only mean that it is $X \reso \eta$ which
does not converge to zero.

\

Another way to see that there is a problem is to consider the following
representation of the resonant term: use $\LL X = \eta$ to write
\[ X \reso \eta = X \reso \LL X = \frac{1}{2} \LL (X \reso X) + \partial_x X
   \reso \partial_x X = | \partial_x X |^2 + \frac{1}{2} \LL (X \reso X) - 2
   \partial_x X \lpara \partial_x X. \]
Integrating this equation over the torus and over $t \in [0, T]$, we get
\[
   \int_0^T \int_{\mathbb{T}^2} X \reso \eta \mathd x \mathd t  = \int_0^T \int_{\mathbb{T}^2} | \partial_x X |^2 \mathd x \mathd t + \frac{1}{2} \int_0^T \int_{\mathbb{T}^2} \LL (X \reso X ) \mathd x - 2 \int_0^T \int_{\mathbb{T}^2} (\partial_x X \lpara \partial_x X) \mathd x \mathd t.
\]
Writing $\LL = \partial_t - \Delta$ and using that $X(0) = 0$ and $\int_{\mathbb{T}^2} \Delta \psi \mathd x = 0$ for all $\psi$ (which can be seen using integration by parts and pulling the operator $\Delta$ on the constant function $1$), we thus get
\[
   \int_0^T \int_{\mathbb{T}^2} X \reso \eta \mathd x \mathd t  = \int_0^T \int_{\mathbb{T}^2} | \partial_x X |^2 \mathd x \mathd t + \frac{1}{2} \int_{\mathbb{T}^2} (X (T) \reso X (T)) \mathd x - 2 \int_0^T \int_{\mathbb{T}^2} (\partial_x X \lpara \partial_x X) \mathd x \mathd t
\]
So if $X \reso \eta \in C_T \CC^{2 \gamma - 2}$ and $X \in C_T
\CC^{\gamma}$, then all the terms should be well defined and finite (the integral over $\mathbb{T}^2$ corresponds to testing a distribution against to constant test function 1). This would mean that $\int_0^T \int_{\mathbb{T}^2} | \partial_x X |^2 \mathd x \mathd t < \noplus + \infty$, but on the other side a direct computation shows that
\[
   \int_{\mathbb{T}^2} | \partial_x X (t, \cdot) |^2 \mathd x = + \infty
\]
for any $t > 0$ almost surely if $\eta$ is the space white noise. Note also
that the problematic term $| \partial_x X |^2$ is exactly the correction term
appearing in the analysis of the linear homogenization problem in Section~\ref{sec:homogenization}.

\

In order to prove the convergence of the smooth solutions in general, we should
introduce corrections to the equation to remove the divergent constant $c_n$.
Let us see where the resonant product $X \reso \eta$ appears. We have
\begin{equation}\label{eq:paracontrolled product F(u)eta}
   (F (u) \eta)^{\sharp} = F (u) \rpara \eta + (F (u))^{\sharp} \reso \eta + C ((F (u))^X, X, \eta) + (F (u))^X (X \reso \eta).
\end{equation}
Now $(F (u))^X = F' (u) u^X$ by Theorem~\ref{thm:paralinearization}, and if
$u$ solves the equation $\LL u = F(u) \eta = F(u) \lpara \eta + (F(u)\eta)^\sharp$, then Theorem~\ref{thm:schauder paracontrolled} with $u = \eta$, $X = U$ shows that $u^X = F (u)$. So we should really consider the renormalized equation
\[
   \LL u_n = F (u_n) \renorm \eta_n := F (u_n) \xi_n - F' (u_n) F (u_n) c_n,
\]
where we recall that $(c_n)$ are the diverging constants for which $(X_n \reso \eta_n - c_n)$ converges. In that case we have
\[
   \LL u_n = F (u_n) \lpara \eta_n + F (u_n) \rpara \eta_n + (F (u_n))^{\sharp} \reso \eta_n + C (F'(u_n) F(u_n), X_n, \eta_n) + F'(u_n) F(u_n) (X_n \reso \eta_n - c_n),
\]
and now all the terms on the right hand side are under control and we can safely pass to the limit, for which we obtain the equation
\begin{equation}\label{eq:pam-renormalized}
   \LL u = F (u) \renorm \eta \assign (F (u)\renorm \eta)^{\sharp} + F (u) \lpara \eta,
\end{equation}
where $(F (u) \renorm \eta)^{\sharp}$ is
calculated using $X \renorm \eta = \lim_n (X_n \reso \eta_n - c_n)$ in the
place of $X \reso \eta$ in~\eqref{eq:paracontrolled product F(u)eta}. Formally, we also denote this product by
\[
   F (u) \renorm \eta = F (u) \eta - F' (u) F(u) \cdummy \infty,
\]
so that the solution $u$ will satisfy
\[ \LL u = F (u) - F' (u) F (u) \cdot \infty . \]
Note that the correction term has exactly the same form as the It{\^o}/Stratonovich corrector for SDEs. For the reader familiar with rough paths this will not come as a surprise: Changing the iterated integrals of a rough path $B$ from some given $\int_0^\cdot B_s \mathd B_s$ to $\int_0^\cdot B_s \mathd B_s + \varphi$ introduces a correction term $+F'(y) F(y) \partial_t \varphi$ in the ODE $\partial_t y = F(y) \partial_t B$. In our setting the resonant term takes the role of the iterated integrals, and since the structure of the ODE and {\tmname{gpam}} is very similar changing the resonant term has a similar effect as changing the iterated integrals in the ODE example.
%
%

\

\begin{remark}
  The convergence properties of $(X_n \reso \eta_n)$ are in stark contrast to the ODE setting: if we consider the equation $\partial_t u = F (u) \zeta$  rather than {\tmname{pam}}, then we should replace $X$ by $Z$ with $\partial_t Z = \zeta$. But then we have in one dimension $Z \reso \zeta = 1/ 2 \partial_t (Z \reso Z)$, so that the convergence of $(Z_n \reso \zeta_n)$ to $Z \reso \zeta$ comes for free with the convergence of $(Z_n)$ to $Z$. Indeed, $\partial_t$ is a bounded linear operator from $\CC^\gamma$ to $\CC^{\gamma-1}$ whenever $\gamma \in \R$, and $Z \mapsto Z \circ Z$ is continuous from $\CC^\gamma$ to $\CC^{2\gamma}$ whenever $\gamma > 0$. So if $(Z_n)$ converges to $Z$ in a H\"older space of positive regularity, then $(\partial_t (Z_n \circ Z_n))$ converges to $\partial_t (Z \circ Z)$. This specific representation of $Z \reso \zeta$ comes from the Leibniz rule for $\partial_t$ and it is the reason why rough path theory is trivial in one dimension, at least as long as one considers those rough paths which are limit of smooth paths. Of course, the argument breaks down as soon as $Z$ has at least two components. As we have discussed, for the second order differential operator $\LL$ we have different rules and obtain
  \[ (X \reso \eta) = (X \reso \LL X) = \frac{1}{2}  \LL (X \reso X) +
     (\partial_x X \reso \partial_x X), \]
  so that in our setting the nontrivial term is $\partial_x X \reso \partial_x X$.
\end{remark}

These considerations lead naturally to the following definition.

\begin{definition}
  \label{def:pam rough distribution}({\tmname{pam}}--enhancement) Let $\gamma
  \in (2 / 3, 1)$ and let
  \[ \mathcal{X}_{\tmop{pam}}^\gamma \subseteq \CC^{\gamma - 2} \times C \CC^{2
     \gamma - 2} \]
  be the closure of the image of the map
  \[ \Theta_{\tmop{pam}} : C^{\infty} \times C ([0, \infty), \mathbb{R})
     \rightarrow \mathcal{X}_{\tmop{pam}}^\gamma, \]
  given by
  \begin{equation}
    \Theta_{\tmop{pam}} (\theta, f) = (\theta, \Phi \renorm \theta) : =
    (\theta, \Phi \reso \theta - f), \label{eq:enhanced-pam}
  \end{equation}
  where $\Phi = J \theta$, that is $\LL \Phi = \theta$ and $\Phi (0) = 0$. We
  will call $\Theta_{\tmop{pam}} (\theta, f)$ the renormalized
  {\tmname{pam}}--enhancement of the driving distribution $\theta$. For $T >
  0$ we define $\mathcal{X}_{\tmop{pam}}^\gamma (T) = \mathcal{X}_{\tmop{pam}}^\gamma |_{[0,
  T]}$ and we write $\| \mathbb{X} \|_{\mathcal{X}_{\tmop{pam}}^\gamma (T)}$ for the
  norm of $\mathbb{X} \in \mathcal{X}_{\tmop{pam}}^\gamma (T)$ in the Banach space
  $\CC^{\gamma - 2} \times C_T \CC^{2 \gamma - 2}$. Moreover, we define the
  distance $d_{\mathcal{X}_{\tmop{pam}}^\gamma (T)} (\mathbb{X},
  \tilde{\mathbb{X}}) = \|\mathbb{X}- \tilde{\mathbb{X}}
  \|_{\mathcal{X}_{\tmop{pam}}^\gamma (T)}$.
\end{definition}

\begin{remark}
   In the homogenization example of Section~\ref{sec:homogenization} we would take $\theta = V_\varepsilon$ and $\Phi = X_\varepsilon$.
\end{remark}

\begin{remark}
  It would be more elegant to renormalize $\Phi \reso \theta$ with a constant
  and not with a time-dependent function, as we discussed above. Indeed this is possible, see Chapter~5 of~\cite{Gubinelli2012}. But since here we chose $\Phi (0) = 0$,
  we have $\Phi (0) \reso \theta = 0$ and therefore $(\Phi_n (0) \reso
  \theta_n - c_n)$ diverges for any diverging sequence of constants $(c_n)$. A
  simple way of avoiding this problem is to consider the stationary version
  $\tilde{\Phi}$ given by
  \[ \tilde{\Phi} (x) = \int_0^{\infty} P_t \Pi_{\neq 0} \theta (x) \mathd t,
  \]
  where $\Pi_{\neq 0}$ denotes the projection on the non-zero Fourier modes,
  $\Pi_{\neq 0} u = u - (2\pi)^{-d/2}
   \hat{u} (0)$. But then $\tilde{\Phi}$ does not depend
  on time and in particular $\tilde{\Phi} (0) \neq 0$, so that we have to
  consider irregular initial conditions in the paracontrolled approach which
  complicates the presentation. Alternatively, we could observe that in the
  white noise case there exist constants $(c_n)$ so that $(X_n (t) \reso \xi_n
  - c_n)$ converges for all $t > 0$, and while the limit $(X (t) \renorm
  \xi)$ diverges as $t \rightarrow 0$, it can be integrated against the heat
  kernel. Again, this would complicate the presentation and here we choose the
  simple (and cheap) solution of taking a time-dependent renormalization.
\end{remark}

\begin{theorem}
  \label{thm:pam}Let $\gamma \in (2 / 3, 1)$ and $\varepsilon > 0$ be such
  that $(2 + \varepsilon) \gamma > 2$. Let $\mathbb{X}= (\eta, X \diamond \eta) \in
  \mathcal{X}_{\tmop{pam}}^\gamma$, $F \in C^{2 + \varepsilon}_b$, and $u_0 \in
  \CC^{2 \gamma}$. Then there exists a unique solution $u \in \CD^{\gamma}
  (X)$ to the equation
  \[ \LL u = F (u) \renorm \eta, \hspace{2em} u (0) = u_0, \]
  up to the (possibly finite) explosion time $\tau = \tau (u) = \inf \{ t
  \geqslant 0 : \| u \|_{\CD^{\gamma}_t} = \infty \} > 0$.
  
  Moreover, $u$ depends on $(u_0, \mathbb{X}) \in \CC^{2 \gamma} \times
  \mathcal{X}_{\tmop{pam}}^\gamma$ in a locally Lipschitz continuous way: if $M, T >
  0$ are such that for all $(u_0, \mathbb{X})$ with $\| u_0 \|_{2 \gamma}
  \vee \| \mathbb{X} \|_{\mathcal{X}_{\tmop{pam}}^\gamma (T)} \leqslant M$, the
  solution $u$ to the equation driven by $(u_0, \mathbb{X})$ satisfies $\tau
  (u) > T$, and if $(\tilde{u}_0, \tilde{\mathbb{X}})$ is another set of data
  bounded in the above sense by $M$, then there exists $C (F, M) > 0$ for
  which
  \[ d_{\CD^{\gamma}_T} (u, \tilde{u}) \leqslant C (F, M) (\| u_0 -
     \tilde{u}_0 \|_{2 \gamma} + d_{\mathcal{X}_{\tmop{pam}}^\gamma (T)}
     (\mathbb{X}, \tilde{\mathbb{X}})) . \]
\end{theorem}

\begin{proof}
  We only have to turn the formal discussion of Section~\ref{sec:fixpoint}
  into rigorous mathematics. The small factor $T^{\delta}$ on page~\pageref{page:scaling factor} is obtained from a
  scaling argument and while this does not require any new insights it is
  somewhat lengthy and we refer to~{\cite{Gubinelli2012,Gubinelli2014}} for
  details.
  
  Let us just indicate how to iterate the construction to obtain the existence
  of solutions up to the explosion time $\tau$. Let us assume that we
  constructed the paracontrolled solution $(u, u^X)$ (with $u^X = F(u)$) on $[0, T_0]$ for some $T_0 > 0$. Now we no longer have $X
  (T_0) = 0$, and also the initial condition $u (T_0)$ is no longer in $\CC^{2
  \gamma}$. But we only used $X (0) = 0$ to see that the initial condition for $u^{\sharp}$ is $u^\sharp(0) = u_0$, and we only used $u_0 \in \CC^{2\gamma}$ to obtain a $\CC^{2 \gamma}$ initial condition for $u^{\sharp}$. On the next interval, the initial condition for $u^\sharp$ is $u^\sharp(T_0) = u(T_0) - F(u(T_0)) \lpara X(T_0)$ which is in $\CC^{2\gamma}$ by construction, since we already know that $u^\sharp \in C([0,T_0], \CC^{2\gamma})$.
  
  As for the continuity in $(u_0, \mathbb{X})$, let $(\tilde{u}_0,
  \tilde{\mathbb{X}})$ be another set of data also bounded by $M$. Then the
  solutions $u$ and $\tilde{u}$ both are bounded in $\CD^{\gamma}_T$ by some
  constant $c = c (F, M) > 0$. So by the continuity properties of the
  paracontrolled product (and the other operations involved), we can estimate
  \[ d_{\CD^{\gamma}_T} (u, \tilde{u}) \leqslant P (c) \left( \| u_0 -
     \tilde{u}_0 \|_{2 \gamma} + d_{\mathcal{X}_{\tmop{pam}}^\gamma (T)}
     (\mathbb{X}, \tilde{\mathbb{X}}) + T^{\delta} d_{\CD^{\gamma}_T} (u,
     \tilde{u}) \right) \]
  for a polynomial $P$. The local Lipschitz continuity on $[0, T]$ immediately
  follows if we choose $T > 0$ small enough. This can be iterated to obtain
  the local Lipschitz continuity on ``macroscopic'' intervals. 
\end{proof}

\begin{remark}
  For the local in time existence it is not necessary to assume $F \in C^{2 +
  \varepsilon}_b$, it suffices if $F \in C^{2 + \varepsilon}$. This can be
  seen by considering a ball containing $u_0 (x)$ for all $x \in
  \mathbb{T}^d$, a function $\tilde{F} \in C^{2 + \varepsilon}_b$ which
  coincides with $F$ on this ball, and by stopping $u$ upon exiting the ball.
  
  In the linear case $F (u) = u$ we have global in time solutions: in general
  we only get local in time solutions because we pick up a superlinear (polynomial)
  estimate when applying the paralinearization result
  Theorem~\ref{thm:paralinearization}. This step is not necessary if $F$ is
  linear, and all the other estimates are linear in $u$. 
\end{remark}

\subsection{Construction of the extended data}

In order to apply Theorem~\ref{thm:pam} to equation~(\ref{eq:pam-eta}) with
white noise perturbation, it remains to show that if $\xi$ is a spatial white
noise on $\mathbb{T}^2$, then $(\xi, X\diamond \xi)$ defines an element of
$\mathcal{X}_{\tmop{pam}}^\gamma$ whenever $\gamma \in (2/3,1)$. In other words, we need to construct $X \renorm
\xi$ and control its regularity.

Since $P_t \xi$ is a smooth function for every $t > 0$, the resonant term $P_t
\xi \reso \xi$ is a smooth function, and therefore we could formally set $X(t)
\reso \xi = \int_0^{t} (P_s \xi \reso \xi) \mathd s$. But we will see
that this expression does not make sense.

Recall that $(\hat{\xi} (k))_{k \in \mathbb{Z}^2}$ is a complex valued,
centered Gaussian process with covariance
\begin{equation}
  \label{eq:wn covariance} \mathbb{E} [\hat{\xi} (k) \hat{\xi} (k')] =
  \delta_{k + k' = 0},
\end{equation}
and such that $\hat{\xi} (k)^{\ast} = \hat{\xi} (- k)$.

\begin{lemma}
  \label{l:anderson area expectation} For any $x \in \mathbb{T}^2$ and $t > 0$
  we have
  \[ g_t =\mathbb{E} [(P_t \xi)(x) \xi (x)]  = \mathbb{E} [(P_t \xi \reso \xi) (x)] =\mathbb{E} [\Delta_{- 1}
     (P_t \xi \reso \xi) (x)] = (2 \pi)^{- 2} \sum_{k \in \mathbb{Z}^2} e^{- t
     |k|^2} . \]
  In particular, $g_t$ does not depend on the partition of unity used to
  define the $\reso$ operator, and $\int_0^t g_s \mathd s = \infty$ for all $t
  > 0$.
\end{lemma}

\begin{proof}
  Let $x \in \mathbb{T}^2$, $t > 0$, and $\ell \ge - 1$. Then
  \[ \mathbb{E} [\Delta_{\ell} (P_t \xi \reso \xi) (x)] = \sum_{|i - j|
     \leqslant 1} \mathbb{E} [\Delta_{\ell} (\Delta_i (P_t \xi) \Delta_j \xi)
     (x)], \]
  where exchanging summation and expectation is justified because it can be
  easily verified that the partial sums of $\Delta_{\ell}  (P_t \xi \reso \xi)
  (x)$ are uniformly $L^p$--bounded for any $p \ge 1$. Now $P_t = e^{- t
  \lvert \cdummy \rvert^2} (\mathD)$, and therefore we get from~\eqref{eq:wn
  covariance}
  \begin{align*}
    \mathbb{E} [\Delta_{\ell} (\Delta_i (P_t \xi) \Delta_j \xi) (x)] &= (2 \pi)^{- 1}  \sum_{k, k' \in \mathbb{Z}^2} e_{k +    k'}^{\ast} (x) \rho_{\ell}  (k + k') \rho_i (k) e^{- t |k|^2} \rho_j (k') \mathbb{E} [\hat{\xi} (k) \hat{\xi} (k')] \\
    & = (2 \pi)^{- 2} \sum_{k \in \mathbb{Z}^2} \rho_{\ell} (0)
    \rho_i (k) e^{- t |k|^2} \rho_j (k) = \delta_{\ell = - 1} (2 \pi)^{-2}
    \sum_{k \in \mathbb{Z}^2} \rho_i (k) \rho_j (k) e^{- t |k|^2} . & 
  \end{align*}
  For $|i - j| > 1$ we have $\rho_i (k) \rho_j (k) = 0$ and therefore
  \[ g_t =\mathbb{E} [(P_t \xi \reso \xi) (x)] = \mathbb{E} [(P_t \xi)(x) \xi (x)] = (2 \pi)^{- 2} \sum_{k \in
     \mathbb{Z}^2} \sum_{i, j} \rho_i (k) \rho_j (k) e^{- t |k|^2} = (2
     \pi)^{- 2} \sum_{k \in \mathbb{Z}^2} e^{- t |k|^2}, \]
  while $\mathbb{E} [(P_t \xi \reso \xi) (x) - \Delta_{- 1} (P_t \xi \reso
  \xi)) (x)] = 0$.
\end{proof}

\begin{exercise}
  Let $\varphi$ be a Schwartz function on $\mathbb{R}^2$ and set
  \[ \xi_n (x) = ((n^2 \varphi (n \cdot)) \ast \xi) (x) = \int_{\mathbb{R}^2}
     n^2 \varphi (n (x - y)) \xi (y) \mathd y = \sum_{k \in \mathbb{Z}^2}
     \langle \xi, n^2 \varphi (n (x + 2 \pi k - \cdummy)) \rangle \]
  for $x \in \mathbb{T}^2$. Write $\CF_{\mathbb{R}^2} \varphi (z) =
  \int_{\mathbb{R}^2} e^{- i \langle z, x \rangle} \varphi (x) \mathd x$.
  Show that
  \[ \mathbb{E} [(P_t \xi_n \reso \xi_n) (x)] =\mathbb{E} [\Delta_{- 1} (P_t
     \xi_n \reso \xi_n) (x)] = (2 \pi)^{- 2} \sum_{k \in \mathbb{Z}^2} e^{- t
     |k|^2}  | \CF_{\mathbb{R}^2} \varphi (k / n) |^2 . \]
  \tmtextbf{Hint:} Use Poisson summation.
\end{exercise}

The diverging time integral motivates us to study the renormalized product $X
\reso \xi - \int_0^{\cdot} g_s \mathd s$, where $\int_0^{\cdot} g_s \mathd s$
is an ``infinite function'':

\begin{lemma}
  \label{lem:pam area}Set
  \[ (X \renorm \xi) (t) = \int_0^t (P_s \xi \reso \xi - g_s) \mathd s. \]
  Then $\mathbb{E} [\|X \renorm \xi \|_{C_T \CC^{2 \gamma - 2}
  (\mathbb{T}^2)}^p] < \infty$ for all $\gamma < 1$, $p \ge 1$, $T > 0$.
  Moreover, if $\varphi$ is a Schwartz function on $\mathbb{R}^2$ with $\int
  \varphi (x) \mathd x = 1$, if $\xi_n = \varphi_n \ast \xi$ with $\varphi_n =
  n^2 \varphi (n \cdot)$ for $n \in \mathbb{N}$, and $X_n (t) =
  \int_0^{t} P_s \xi_n \mathd s$, then
  \[ \lim_{n \rightarrow \infty} \mathbb{E} [\|X \renorm \xi - (X_n \reso
     \xi_n - f_n)\|_{C_T \CC^{2 \gamma - 2} (\mathbb{T}^2)}^p] = 0 \]
  for all $p \ge 1$, where for all $x \in \mathbb{T}^2$
  \begin{align*}
    f_n (t) & =  \mathbb{E} [X_n (t, x) \xi_n (x)] =\mathbb{E} [(X_n (t)
    \reso \xi_n) (x)]  \\
    & =  (2 \pi)^{- 2} \sum_{k \in \mathbb{Z}^2 \setminus \{ 0 \}} \frac{|
    \CF_{\mathbb{R}^2} \varphi (k / n) |^2}{|k|^2} (1 - e^{- t | k |^2}) + (2
    \pi)^{- 2} t.  
  \end{align*}
\end{lemma}

\begin{proof}
  To lighten the notation, we will only show that $\mathbb{E} [\|X \renorm
  \xi \|_{C_T \CC^{2 \gamma - 2}}^p] < \infty$. The convergence of $(X_n
  \reso \xi_n - f_n)$ to $X \renorm \xi$ is shown by applying dominated
  convergence, and we leave it as an exercise. Let $t > 0$ and define $\Xi_t =
  P_t \xi \reso \xi - g_t$. Let us start by estimating $\mathbb{E} [| \Delta_{\ell} \Xi_t (x)
  |^2]$ for $\ell \geqslant - 1$ and $x \in \mathbb{T}^2$.
  Lemma~\ref{l:anderson area expectation} yields $\Delta_{\ell} g_t = 0
  =\mathbb{E} [\Delta_{\ell} (P_t \xi \reso \xi) (x)]$ for $\ell \ge 0$ and
  $x \in \mathbb{T}^2$, and $\Delta_{- 1} g_t = g_t =\mathbb{E} [\Delta_{- 1}
  (P_t \xi \reso \xi) (x)]$, so that $\mathbb{E} [| \Delta_{\ell} \Xi_t (x)
  |^2] = \tmop{Var} (\Delta_{\ell} (P_t \xi \reso \xi) (x))$. But
  \begin{align*}
    \Delta_{\ell} (P_t \xi \reso \xi) (x) & = \sum_{k \in \mathbb{Z}^2}
    e^{\ast}_k (x) \rho_{\ell} (k) \CF (P_t \xi \reso \xi)) (k)\\
    & = (2 \pi)^{- 1}  \sum_{k_1, k_2 \in \mathbb{Z}^2} \sum_{|i - j| \leqslant
    1} e^{\ast}_{k_1 + k_2} (x) \rho_{\ell} (k_1 + k_2) \rho_i (k_1) e^{- t
    |k_1 |^2} \hat{\xi} (k_1) \rho_j (k_2) \hat{\xi} (k_2),
  \end{align*}
  and therefore
  \begin{align*}
    & \tmop{Var} (\Delta_{\ell} (P_t \xi \reso \xi) (x))\\
    & \hspace{20pt} = (2 \pi)^{- 2} \sum_{k_1, k_2} \sum_{k'_1, k'_2}
    \sum_{|i - j| \leqslant 1} \sum_{|i' - j' | \leqslant 1} e^{\ast}_{k_1 +
    k_2} (x) \rho_{\ell} (k_1 + k_2) \rho_i (k_1) e^{- t |k_1 |^2} \rho_j
    (k_2)\\
    & \hspace{65pt} \times e^{\ast}_{k'_1 + k'_2} (x) \rho_{\ell} (k'_1 +
    k'_2) \rho_{i'} (k'_1) e^{- t |k'_1 |^2} \rho_{j'} (k'_2) \tmop{Cov}
    (\hat{\xi} (k_1) \hat{\xi} (k_2), \hat{\xi} (k'_1) \hat{\xi} (k'_2)),
  \end{align*}  
  where exchanging summation and expectation can be justified a posteriori by
  the uniform $L^p$--boundedness of the partial sums. Now Wick's theorem
  ({\cite{Janson1997}}, Theorem~1.28) gives
  \begin{align*}
     \tmop{Cov} (\hat{\xi} (k_1) \hat{\xi} (k_2), \hat{\xi} (k'_1) \hat{\xi}(k'_2)) & = \mathbb{E} [\hat{\xi} (k_1) \hat{\xi} (k_2) \hat{\xi} (k'_1) \hat{\xi} (k'_2)] -\mathbb{E} [\hat{\xi} (k_1) \hat{\xi} (k_2)] \mathbb{E} [\hat{\xi} (k'_1) \hat{\xi} (k'_2)] \\
     & =\mathbb{E} [\hat{\xi} (k_1) \hat{\xi} (k_2)] \mathbb{E} [\hat{\xi}(k'_1) \hat{\xi} (k'_2)] +\mathbb{E} [\hat{\xi} (k_1) \hat{\xi} (k_1')] \mathbb{E} [\hat{\xi} (k_2) \hat{\xi} (k_2')] \\
     &\hspace{10pt} + \mathbb{E} [\hat{\xi} (k_1) \hat{\xi} (k_2')] \mathbb{E} [\hat{\xi}(k_2) \hat{\xi} (k'_1)] -\mathbb{E} [\hat{\xi} (k_1) \hat{\xi} (k_2)] \mathbb{E} [\hat{\xi} (k'_1) \hat{\xi} (k'_2)] \\
     & = (\delta_{k_1 + k_1' = 0} \delta_{k_2 + k'_2 = 0} + \delta_{k_1 + k'_2 = 0} \delta_{k_2 + k_1' = 0}),
  \end{align*}
  which leads to
  \begin{align*}
    \tmop{Var} (\Delta_{\ell} (P_t \xi \reso \xi) (x)) & = (2 \pi)^{- 4}
    \sum_{k_1, k_2} \sum_{|i - j| \leqslant 1} \sum_{|i' - j' | \leqslant 1}
    \mathbb{I}_{\ell \lesssim i} \mathbb{I}_{\ell \lesssim i'} \rho^2_{\ell}
    (k_1 + k_2) \rho_i (k_1) \rho_j (k_2) \\
    &\hspace{70pt} \times [\rho_{i'} (k_1) \rho_{j'} (k_2) e^{- 2 t|k_1 |^2} + \rho_{i'}
    (k_2) \rho_{j'} (k_1) e^{- t|k_1 |^2 - t|k_2 |^2}] .
  \end{align*}
  Observe that there exists $c > 0$ such that $e^{- 2 t |k|^2} \lesssim e^{- tc 2^{2 i}}$ for all $k \in \tmop{supp} (\rho_i) \cup \tmop{supp} (\rho_j)$ with $i,j \ge - 1$ and $|i-j| \leqslant 1$. Thus
  \begin{align*}
    & \tmop{Var} (\Delta_{\ell} (P_t \xi \reso \xi) (x)) \\
    & \hspace{40pt} \lesssim \sum_{i, j, i', j'} \mathbb{I}_{\ell \lesssim i} \mathbb{I}_{i \sim j \sim i' \sim j'} \sum_{k_1, k_2}  \mathbb{I}_{\tmop{supp} (\rho_{\ell})} (k_1 + k_2) \mathbb{I}_{\tmop{supp} (\rho_i)} (k_1) \mathbb{I}_{\tmop{supp} (\rho_j)} (k_2) e^{- 2 tc 2^{2 i}} \\
    & \hspace{40pt} \lesssim \sum_{i : i \gtrsim \ell} 2^{2 i} 2^{2 \ell} e^{- tc 2^{2 i}} \lesssim \frac{2^{2 \ell}}{t} \sum_{i : i \gtrsim \ell} e^{- tc' 2^{2 i}} \lesssim \frac{2^{2 \ell}}{t} e^{- tc' 2^{2 \ell}}, & 
  \end{align*}
  where in the third step we used that $t 2^{2 i} \lesssim e^{t (c - c') 2^{2 i}}$ for all $0< c' <
  c$.
  
  Consider now $X \renorm \xi (t) = \int_0^t \Xi_s \mathd s$. We have for all
  $0 \leqslant s < t$
  \[ \mathbb{E} [\| X \renorm \xi (t) - X \renorm \xi (s) \|_{B^{2 \gamma -
     2}_{2 p, 2 p}}^{2 p}] = \sum_{\ell} 2^{2 p \ell (2 \gamma - 2)}
     \int_{\mathbb{T}^2} \mathbb{E} [| \Delta_{\ell} (X \renorm \xi (t) - X
     \renorm \xi (s)) (x) |^{2 p}] \mathd x. \]
  Since the random variable $\Delta_{\ell} (X \renorm \xi (t) - X \renorm \xi (s))
  (x)$ lives in the second non-homogeneous chaos generated by the Gaussian white noise $\xi$, we may use Gaussian hypercontractivity ({\cite{Janson1997}}, Theorem 3.50) to bound
  \begin{align*}
     \mathbb{E} [| \Delta_{\ell} (X \renorm \xi (t) - X \renorm \xi (s)) (x) |^{2 p}] & \lesssim \mathbb{E} [| \Delta_{\ell} (X \renorm \xi (t) - X \renorm \xi (s)) (x) |]^{2 p} \\
     & \leqslant \Big( \int_s^t \mathbb{E} [| \Delta_{\ell} \Xi_r (x) |] \mathd r \Big)^{2 p}.
  \end{align*}
  But we just showed that
  \begin{align*}
     \mathbb{E} [| \Delta_{\ell} \Xi_r (x) |] & \leqslant \mathbb{E} [| \Delta_{\ell} \Xi_r (x) |^2]^{1 / 2} = (\tmop{Var} (\Delta_{\ell} (P_r \xi \reso \xi) (x)))^{1 / 2} \\
     & \lesssim r^{- 1 / 2} 2^{\ell} e^{- \frac{1}{2} r c' 2^{2 \ell}} = r^{- 1 / 2} 2^{\ell} e^{- r c'' 2^{2 \ell}}
  \end{align*}
  for $c'' = c'/2 > 0$, and therefore
  \begin{align*}
    \Big(\mathbb{E} \Big[\| X \renorm \xi (t) - X \renorm \xi (s) \|_{B^{2 \gamma - 2}_{2 p, 2 p}}^{2 p}\Big]\Big)^{1 / 2 p} & \lesssim \Big( \sum_{\ell} \Big( 2^{\ell (2 \gamma - 2)} \int_s^t r^{- 1 / 2} 2^{\ell} e^{- r c'' 2^{2 \ell}} \mathd r \Big)^{2 p} \Big)^{1 / 2 p} \\
    & \leqslant \sum_{\ell} 2^{\ell (2 \gamma - 1)} \int_s^t r^{- 1 / 2} e^{-
    r c'' 2^{2 \ell}} \mathd r \\
    & \lesssim \int_s^t r^{- 1 / 2} \int_{- 1}^{\infty} (2^x)^{2  \gamma - 1} e^{- r c'' 2^{2 x}} \mathd x \mathd r.
  \end{align*}
  The change of variable $y = \sqrt{r} 2^x$ leads to
  \[
     \Big(\mathbb{E} \Big[\| X \renorm \xi (t) - X \renorm \xi (s) \|_{B^{2 \gamma - 2}_{2 p, 2 p}}^{2 p}\Big]\Big)^{1 / 2 p} \lesssim \int_s^t r^{- 1 / 2} r^{- (2 \gamma - 1) / 2}  \int_0^{\infty} y^{2 \gamma - 2} e^{- c'' y^2} \mathd y \mathd r.
  \]
  For $\gamma > 1 / 2$, the integral in $y$ is finite and we end up with
  \[
     \Big(\mathbb{E} \Big[\| X \renorm \xi (t) - X \renorm \xi (s) \|_{B^{2 \gamma - 2}_{2 p, 2 p}}^{2 p}\Big]\Big)^{1 / 2 p} \lesssim \int_s^t r^{- \gamma} \mathd r \lesssim | t - s |^{1 - \gamma}
  \]
  provided that $\gamma \in (1 / 2, 1)$. So for large enough $p$ we can use
  Kolmogorov's continuity criterion to deduce that (modulo taking a modification of $X\diamond \xi)$ we have $\mathbb{E} [\| X \renorm
  \xi \|_{C_T B^{2 \gamma - 2}_{2 p, 2 p}}^{2 p}] < \infty$ for all $T > 0$. Since this holds for all $\gamma < 1$, the claim now follows from the Besov embedding theorem, Lemma~\ref{lem:besov embedding}.
\end{proof}

Combining Theorem~\ref{thm:pam} and Lemma~\ref{lem:pam area}, we are finally
able to solve~(\ref{eq:pam-eta}) driven by a space white noise.

\begin{corollary}
  Let $\varepsilon > 0$ and let $F \in C^{2 + \varepsilon}_b$ and assume that
  $u_0$ is a random variable that almost surely takes its values in $\CC^{2
  \gamma}$ for some $\gamma \in (2 / 3, 1)$ with $(2 + \varepsilon) \gamma >
  2$. Let $\xi$ be a spatial white noise on $\mathbb{T}^2$. Then there exists
  a unique solution $u$ to
  \[ \LL u = F (u) \renorm \xi, \hspace{2em} u (0) = u_0, \]
  up to the (possibly finite) explosion time $\tau = \tau (u) = \inf \{ t
  \geqslant 0 : \| u \|_{\CD^{\gamma}_t} = \infty \}$ which is almost
  surely strictly positive.
  
  If $(\varphi_n)$ and $(\xi_n)$ are as described in Lemma~\ref{lem:pam area},
  and if $(u_{0, n})$ converges in probability in $\CC^{2 \gamma}$ to $u_0$,
  then $u$ is the limit in probability of the solutions $u_n$ to
  \[ \LL u_n = F (u_n) \renorm \xi_n, \hspace{2em} u_n (0) = u_{0,n} . \]
\end{corollary}

\begin{remark}
  We even have a stronger result: We can fix a null set outside of which $X
  \renorm \xi$ is regular enough, and once we dispose of that null set we can
  solve all equations for any regular enough $u_0$ and $F$ simultaneously,
  without ever having to worry about null sets again. This is for example
  interesting when studying stochastic flows or when studying equations with
  random $u_0$ and $F$.
  
  The pathwise continuous dependence on the signal is also powerful in several
  other applications, for example support theorems and large deviations. For
  examples in the theory of rough paths see~{\cite{Friz2010}}.
\end{remark}

\section{The stochastic Burgers equation}\label{sec:rbe}

Let us now return to the stochastic Burgers equation {\tmname{sbe}}
\begin{equation}
  \LL u = \partial_x u^2 + \partial_x \xi, \hspace{2em} u (0) = u_0,
  \label{eq:sbe-theta}
\end{equation}
where $u : [0, \infty) \times \mathbb{T} \rightarrow \mathbb{R}$, $\xi$ is a
space-time white noise, and $\partial_x$ denotes the spatial derivative. As we
argued before, the solution $u$ cannot be expected to behave better than the
Ornstein--Uhlenbeck process $X$, the solution of the linear equation $\LL X =
\partial_x \xi$, and as we saw in Section~\ref{sec:energy} $X (t)$ is for all
$t > 0$ a smooth function of the space variable plus a space white noise. By
Exercise~\ref{exo:wn besov}, the white noise in dimension $1$ has regularity
$\CC^{- 1 / 2 -}$. Thus $X \in C \CC^{- 1 / 2 -}$, and in particular $u^2$ is
the square of a distribution and a priori not well defined.

What raises some hope is that in Lemma~\ref{lem:OU square} we were able to
show that $\partial_x X^2$ exists as a space--time distribution. So as in the
previous examples there are stochastic cancellations going into $\partial_x
X^2$. The energy solution approach was designed to take those cancellations
into account in the full solution $u$, but while it allowed us to work under
rather weak assumptions which easily gave us existence of solutions, it did
not give us sufficient control to have uniqueness of solutions. On the other
side, a suitable paracontrolled ansatz for the solution $u$ will allow us to
transfer the cancellation properties of $X$ to $u$ and it will allow us to
construct $\partial_x u^2$ as a \tmtextit{continuous} bilinear map, from where
existence and uniqueness of solutions easily follows.

\subsection{Structure of the solution}

In this discussion we consider the case of zero initial condition and smooth
noise $\xi$, and we analyze the structure of the solution. Let us expand $u$
around the Ornstein--Uhlenbeck process $X$ with $\LL X = \partial_x \xi$, $X
(0) = 0$. Setting $u = X + u^{\geqslant 1}$, we have
\[ \LL u^{\geqslant 1} = \partial_x (u^2) = \partial_x (X^2) + 2 \partial_x
   (Xu^{\geqslant 1}) + \partial_x ((u^{\geqslant 1})^2) . \]
Let us define the bilinear map
\[ B (f, g) = J \partial_x (fg) = \int_0^{\cdot} P_{\cdot - s} \partial_x (f
   (s) g (s)) \mathd s. \]
Then we can proceed by performing a further change of variables in order to
remove the term $\partial_x (X^2)$ from the equation by setting
\begin{equation}
  u = X + B (X, X) + u^{\geqslant 2} \label{eq:naive-exp} .
\end{equation}
Now $u^{\geqslant 2}$ satisfies
\begin{equation}
  \begin{array}{lll}
    \LL u^{\geqslant 2} & = & 2 \partial_x (XB (X, X)) + \partial_x (B (X, X)
    B (X, X))\\
    &  & + 2 \partial_x (Xu^{\geqslant 2}) + 2 \partial_x (B (X, X)
    u^{\geqslant 2}) + \partial_x ((u^{\geqslant 2})^2) .
  \end{array} \label{eq:naive-exp-2}
\end{equation}
We can imagine to make a similar change of variables to get rid of the term
\[ 2 \partial_x (XB (X, X)) = \LL 2 B (X, B (X, X)) . \]
As we proceed in this inductive expansion, we generate a number of explicit
terms, obtained by various combinations of $X$ and $B$. Since we will have to
deal explicitly with at least some of these terms, it is convenient to
represent them with a compact notation involving binary trees. A binary tree
$\tau \in \mathcal{T}$ is either the root $\bullet$ or the combination of two
smaller binary trees $\tau = (\tau_1 \tau_2)$, where the two edges of the root
of $\tau$ are attached to $\tau_1$ and $\tau_2$ respectively. For example
\[
   (\bullet \bullet) = \text{\resizebox{.8em}{!}{\includegraphics{trees-1.eps}}} \comma  \hspace{1em}
   (\text{\resizebox{.8em}{!}{\includegraphics{trees-1.eps}}} \bullet) = \text{\resizebox{.8em}{!}{\includegraphics{trees-2.eps}}} \comma \hspace{1em} 
   (\text{\resizebox{.8em}{!}{\includegraphics{trees-2.eps}}} \bullet) = \text{\resizebox{1em}{!}{\includegraphics{trees-3.eps}}} \comma \hspace{1em} 
   (\text{\resizebox{.8em}{!}{\includegraphics{trees-1.eps}}} \resizebox{.8em}{!}{\includegraphics{trees-1.eps}}) = \text{\resizebox{1.25em}{!}{\includegraphics{trees-4.eps}}} \comma \hspace{1em} \ldots
\]
Then we define recursively
\[ X^{\bullet} = X, \hspace{2em} X^{(\tau_1 \tau_2)} = B (X^{\tau_1},
   X^{\tau_2}), \]
giving
\[
   X^{\zzone} = B (X, X), \hspace{1em} X^{\zztwo} = B (X^{\zzone}, X), \hspace{1em} X^{\zzthree} = B (X^{\zztwo}, X), \hspace{1em} X^{\zzfour} = B(X^{\zzone}, X^{\zzone}),
\]
and so on. In this notation the
expansion~(\ref{eq:naive-exp})--(\ref{eq:naive-exp-2}) reads
\begin{equation}
  u = X + X^{\zzone} + u^{\geqslant 2}, \label{eq:u-expansion}
\end{equation}
\begin{equation}
  u^{\geqslant 2} = 2 X^{\zztwo} + X^{\zzfour} + 2 B (X, u^{\geqslant 2}) + 2
  B (X^{\zzone}, u^{\geqslant 2}) + B (u^{\geqslant 2}, u^{\geqslant 2}) .
  \label{eq:u-ge-2}
\end{equation}
\begin{remark}
  We observe that formally the solution $u$ of {\tmname{sbe}} can be expanded
  as an infinite sum of terms labelled by binary trees:
  \[ u = \sum_{\tau \in \mathcal{T}} c (\tau) X^{\tau}, \]
  where $c (\tau)$ is a combinatorial factor counting the number of planar
  trees which are isomorphic (as graphs) to $\tau$. For example $c (\bullet) =
  1$, $c ( \text{\resizebox{.8em}{!}{\includegraphics{trees-1.eps}}}
  ) = 1$, $c (
  \text{\resizebox{.8em}{!}{\includegraphics{trees-2.eps}}} ) = 2$, $c
  (\text{\resizebox{1em}{!}{\includegraphics{trees-3.eps}}}) = 4$, $c (
  \text{\resizebox{1.2em}{!}{\includegraphics{trees-4.eps}}} ) = 1$ and
  in general $c (\tau) = \sum_{\tau_1, \tau_2 \in \mathcal{T}}
  \mathbb{I}_{(\tau_1 \tau_2) = \tau} c (\tau_1) c (\tau_2)$. Alternatively,
  we may truncate the summation at trees of degree at most $n$ and set
  \[ u = \sum_{\tau \in \mathcal{T}, d (\tau) < n} c (\tau) X^{\tau} +
     u^{\geqslant n}, \]
  where we denote by $d (\tau) \in \mathbb{N}_0$ the degree of the tree
  $\tau$, given by $d (\bullet) = 0$ and then inductively $d ((\tau_1 \tau_2))
  = 1 + d (\tau_1) + d (\tau_2)$. For example $d (
  \text{\resizebox{.8em}{!}{\includegraphics{trees-1.eps}}} ) = 1$, $d
  ( \text{\resizebox{.8em}{!}{\includegraphics{trees-2.eps}}} ) =
  2$, $d (\text{\resizebox{1em}{!}{\includegraphics{trees-3.eps}}}) = 3$, $d
  ( \text{\resizebox{1.2em}{!}{\includegraphics{trees-4.eps}}} ) =
  3$. We then obtain for the remainder
  \begin{align}\label{eq:truncated expansion} \nonumber
     u^{\geqslant n} & = \sum_{\tmscript{\begin{array}{c} \tau_1, \tau_2 : d (\tau_1) < n, d (\tau_2) < n\\ d ((\tau_1 \tau_2)) \geqslant n \end{array}}} c (\tau_1) c (\tau_2) X^{(\tau_1 \tau_2)} \\
     &\qquad + \sum_{\tau : d(\tau) < n} c (\tau) B (X^{\tau}, u^{\geqslant n}) + B (u^{\geqslant n}, u^{\geqslant n}) .
  \end{align}
\end{remark}

Our aim is to control the truncated expansion under the natural regularity
assumptions in the white noise case, $X \in C \CC^{- 1 / 2 -}$.
Since~(\ref{eq:truncated expansion}) contains the term $B (X, u^{\geqslant
n})$ which in turn contains the paraproduct $J \partial_x (u^{\geqslant n}
\lpara X)$, the remainder $u^{\geqslant n}$ will be at best in $C \CC^{1 / 2
-}$. But then the sum of the regularities of $X$ and $u^{\geqslant n}$ is
negative, and the term $B (X, u^{\geqslant n})$ is not well defined. We
therefore continue the expansion up to the point (turning out to be
$u^{\geqslant 3}$) where we can set up a paracontrolled ansatz for the
remainder, which will allow us to make sense of $\partial_x (X \reso
u^{\geqslant n})$ and thus of $B (X, u^{\geqslant n})$.

\subsection{Paracontrolled solution}

Inspired by the partial tree series expansion of $u$ we set up a
paracontrolled ansatz of the form
\begin{equation}
  u = X + X^{\zzone} + 2 X^{\zztwo} + u^Q, \hspace{2em} u^Q = u' \lpara Q +
  u^{\sharp}, \label{eq:paracontrolled-structure}
\end{equation}
where the functions $u', Q$ and $u^{\sharp}$ are for the moment arbitrary, but
we assume $u', Q \in \LL^{\gamma}$ and $u^{\sharp} \in \LL^{2 \gamma}$, where
from now on we fix $\gamma \in (1 / 3, 1 / 2)$. For such $u$, the nonlinear
term takes the form
\begin{align}\label{eq:u-square} \nonumber
   \partial_x u^2 & = \partial_x (X^2 + 2 X^{\zzone} X + (X^{\zzone})^2 + 4 X^{\zztwo} X) + 2 \partial_x (u^Q X) \\
  &\qquad + 2 \partial_x (X^{\zzone} (u^Q + 2 X^{\zztwo})) + \partial_x ((u^Q + 2 X^{\zztwo})^2),
\end{align}
which gives us an equation for $u^Q$:
\begin{align}\label{eq:uQ-equation} \nonumber
   \LL u^Q & = \partial_x ((X^{\zzone})^2 + 4 X^{\zztwo}X) + 2 \partial_x (u^Q X) + 2 \partial_x (X^{\zzone} (u^Q + 2 X^{\zztwo})) + \partial_x ((u^Q + 2 X^{\zztwo})^2)\\
   & = \LL X^{\zzfour} + 4 \LL X^{\zzthree} + 2 \partial_x (u^Q X) + 2 \partial_x (X^{\zzone} (u^Q + 2 X^{\zztwo})) + \partial_x ((u^Q + 2 X^{\zztwo})^2) .
\end{align}
In Lemma~\ref{lemma:ou-regularity} we showed that $X \in C H^{-1/2-}$. But now we understand Besov spaces and Gaussian hypercontractivity well enough so that we can return to the proof and modify the argumentation in order to show that $X \in C \CC^{-1/2-}$. If we then formally apply the paraproduct estimate Theorem~\ref{thm:paraproduct} (which is of course not possible since the regularity requirements for the
resonant term are not satisfied), we obtain $X^2 \in C \CC^{-1-}$ and then $\partial_x X^2 \in C \CC^{-2-}$. Therefore, $X^\zzone = J (\partial_x X^2)$ should be in $C \CC^{0-}$. Note that Lemma~\ref{lem:schauder} does not apply here, because $-2-$ is not in $(-2,0)$. But we only needed this requirement to control the temporal regularity in $L^\infty$ of the image of $J$. For arbitrary $\alpha \in \R$ we have $J u \in C \CC^\alpha$ whenever $u \in C \CC^{\alpha-2}$, see for example Lemma~A.9 in~\cite{Gubinelli2012}. Similarly we derive the formal regularities of the remaining driving terms: $X^{\zztwo} \in \LL^{1 / 2 -}$, $X^{\zzthree} \in \LL^{1 / 2 -}$, and
$X^{\zzfour} \in \LL^{1 -}$. In terms of $\gamma \nocomma$, we can encode this
as
\[ X \in C \CC^{\gamma - 1}, \hspace{1em} X^{\zzone} \in C \CC^{2 \gamma - 1},
   \hspace{1em} X^{\zztwo} \in \LL^{\gamma}, \hspace{1em} X^{\zzthree} \in
   \LL^{\gamma}, \hspace{1em} X^{\zzfour} \in \LL^{2 \gamma} . \]
Under these regularity assumptions the term $2 \partial_x (X^{\zzone} (u^Q +
X^{\zztwo})) + \partial_x ((u^Q + X^{\zztwo})^2)$ is well defined and the only
problematic term in~(\ref{eq:uQ-equation}) is $\partial_x (u^Q X)$. Using the
paracontrolled structure of $u^Q$, we can make sense of $\partial_x (u^Q X)$
as a bounded operator provided that $Q \reso X \in C \CC^{2 \gamma - 1}$ is
given. In other words, the right hand side of~(\ref{eq:uQ-equation}) is well
defined for paracontrolled distributions.

Next, we should specify how to choose $Q$ and which form $u'$ will take for
the solution $u^Q$. We have formally
\begin{align*}
   \LL u^Q & = \LL X^{\zzfour} + 4 \LL X^{\zzthree} + 2 \partial_x (u^Q X) + 2  \partial_x (X^{\zzone} (u^Q + 2 X^{\zztwo})) + \partial_x ((u^Q + 2  X^{\zztwo})^2) \\
   & = 4 \partial_x (X^{\zztwo} X) + 2 \partial_x (u^Q X) + C \CC^{2 \gamma - 2}\\
   & = 4 X^{\zztwo} \lpara \partial_x X + 2 u^Q \lpara \partial_x X + C \CC^{2\gamma - 2},
\end{align*}
where we assumed that not only $\LL X^{\zzthree} \in C \CC^{\gamma - 2}$, but
that $\partial_x (X^{\zztwo} \reso X) \in C \CC^{2 \gamma - 1}$ (which implies
$\LL X^{\zzthree} \in C \CC^{\gamma - 2}$, but also the stronger statement
$\LL X^{\zzthree} - X^{\zztwo} \lpara \partial_x X \in C \CC^{2 \gamma - 2}$).
By Theorem~\ref{thm:schauder paracontrolled}, $u^Q$ is paracontrolled by $J
(\partial_x X)$, and in other words we should set $Q = J (\partial_x X)$. The
derivative $u'$ of the solution $u^Q$ will then be given by $u' = 4 X^{\zztwo}
+ 2 u^Q$.

Unlike for {\tmname{pam}}, here we do not need to introduce a renormalization.
This is due to the fact that we differentiate after taking the square: to
construct $u^2$, we would have to subtract an infinite constant and formally
consider $u^{\renorm 2} = u^2 - \infty$, or at the level of the approximation
$u_n^2 - c_n$. But then
\[ \partial_x u^{\renorm 2} = \lim_{n \rightarrow \infty} \partial_x (u_n^2 -
   c_n) = \lim_{n \rightarrow \infty} \partial_x u_n^2 = \partial_x u^2 . \]
So we obtain the following description of the driving data for the stochastic
Burgers equation.

\begin{definition}
  \label{def:burgers rough distribution}({\tmname{sbe}}--enhancement) Let
  $\gamma \in (1 / 3, 1 / 2)$ and let
  \[ \mathcal{X}_{\tmop{sbe}} \subseteq C \CC^{\gamma - 1} \times C \CC^{2
     \gamma - 1} \times \LL^{\gamma} \times \LL^{2 \gamma} \times C \CC^{2
     \gamma - 1} \times C \CC^{2 \gamma - 1} \]
  be the closure of the image of the map $\Theta_{\tmop{sbe}} : C
  (\mathbb{R}_+, C^{\infty} (\mathbb{T})) \rightarrow
  \mathcal{X}_{\tmop{sbe}}$ given by
  \begin{equation}
    \Theta_{\tmop{sbe}} (\theta) = (X (\theta), X^{\zzone} (\theta),
    X^{\zztwo} (\theta), X^{\zzfour} (\theta), (X^{\zztwo} \reso X) (\theta),
    (Q \reso X) (\theta)), \label{eq:enhanced-theta}
  \end{equation}
  where
  \begin{equation}
    \begin{array}{rll}
      X (\theta) & = & J (\partial_x \theta),\\
      X^{\zzone} (\theta) & = & B (X (\theta), X (\theta)),\\
      X^{\zztwo} (\theta) & = & B (X^{\zzone} (\theta), X (\theta)),\\
      X^{\zzfour} (\theta) & = & B (X^{\zzone} (\theta), X^{\zzone}
      (\theta)),\\
      Q (\theta) & = & J (\partial_x X (\theta)) .
    \end{array} \label{eq:pol-X}
  \end{equation}
  We will call $\Theta_{\tmop{sbe}} (\theta)$ the {\tmname{sbe}}--enhancement
  of the driving distribution $\theta$. For $T > 0$ we define
  $\mathcal{X}_{\tmop{sbe}} (T) =\mathcal{X}_{\tmop{sbe}} |_{[0, T]}$ and we
  write $\| \mathbb{X} \|_{\mathcal{X}_{\tmop{sbe}} (T)}$ for the norm of
  $\mathbb{X}$ in the Banach space $C_T \CC^{\gamma - 1} \times C_T \CC^{2
  \gamma - 1} \times \LL^{\gamma}_T \times \LL_T^{2 \gamma} \times C_T \CC^{2
  \gamma - 1} \times C_T \CC^{2 \gamma - 1}$. Moreover, we define the distance
  $d_{\mathcal{X}_{\tmop{sbe}} (T)} (\mathbb{X}, \tilde{\mathbb{X}}) = \|
  \mathbb{X}- \tilde{\mathbb{X}} \|_{\mathcal{X}_{\tmop{sbe}} (T)}$.
\end{definition}

For every $\mathbb{X} \in \mathcal{X}_{\tmop{sbe}}$, there is an associated
space of paracontrolled distributions:

\begin{definition}
  Let $\mathbb{X} \in \mathcal{X}_{\tmop{sbe}}$. Then the space of
  paracontrolled distributions $\CD^{\gamma} (\mathbb{X})$ is defined as the
  set of all $(u, u') \in C \CC^{\gamma - 1} \times \LL^{\gamma}$ with
  \[ u = X + X^{\zzone} + 2 X^{\zztwo} + u' \lpara Q + u^{\sharp}, \]
  where $u^{\sharp} \in \LL^{2 \gamma}$. For $T > 0$ we define
  \[ \| u \|_{\CD^{\gamma}_T} = \| u' \|_{\LL^{\gamma}_T} + \| u^{\sharp}
     \|_{C_T \CC^{2 \gamma}} . \]
  If $\tilde{\mathbb{X}} \in \mathcal{X}_{\tmop{sbe}}$ and $(\tilde{u},
  \tilde{u}') \in \DD^{\gamma} (\tilde{\mathbb{X}}) \nocomma$, then we also
  write
  \[ d_{\DD^{\gamma}_T} (u, \tilde{u}) = \| u' - \tilde{u}'
     \|_{\LL^{\gamma}_T} + \|u^{\sharp} - \tilde{u}^{\sharp} \|_{C_T \CC^{2
     \gamma}_T} . \]
\end{definition}

We now have everything in place to solve {\tmname{sbe}} driven by $\mathbb{X}
\in \mathcal{X}_{\tmop{sbe}}$.

\begin{theorem}
  \label{thm:rbe}Let $\gamma \in (1 / 3, 1 / 2)$. Let $\mathbb{X} \in
  \mathcal{X}_{\tmop{sbe}}$, write $\partial_x \theta = \LL X$, and let $u_0
  \in \CC^{2 \gamma}$. Then there exists a unique solution $u \in \CD^{\gamma}
  (\mathbb{X})$ to the equation
  \begin{equation}
    \label{eq:rbe theorem} \LL u = \partial_x u^2 + \partial_x \theta,
    \hspace{2em} u (0) = u_0,
  \end{equation}
  up to the (possibly finite) explosion time $\tau = \tau (u) = \inf \{ t
  \geqslant 0 : \| u \|_{\CD^{\gamma}_t} = \infty \} > 0$.
  
  Moreover, $u$ depends on $(u_0, \mathbb{X}) \in \CC^{2 \gamma} \times
  \mathcal{X}_{\tmop{sbe}}$ in a locally Lipschitz continuous way: if $M, T >
  0$ are such that for all $(u_0, \mathbb{X})$ with $\| u_0 \|_{2 \gamma}
  \vee \| \mathbb{X} \|_{\mathcal{X}_{\tmop{sbe}} (T)} \leqslant M$, the
  solution $u$ to the equation driven by $(u_0, \mathbb{X})$ satisfies $\tau
  (u) > T$, and if $(\tilde{u}_0, \tilde{\mathbb{X}})$ is another set of data
  bounded in the above sense by $M$, then there exists $C (M) > 0$ for which
  \[ d_{\CD^{\gamma}_T} (u, \tilde{u}) \leqslant C (M) (\| u_0 - \tilde{u}_0
     \|_{2 \gamma} + d_{\mathcal{X}_{\tmop{sbe}} (T)} (\mathbb{X},
     \tilde{\mathbb{X}})) . \]
\end{theorem}

\begin{proof}
  By definition of the term $\partial_x u^2$, the distribution $u \in
  \CD^{\gamma} (\mathbb{X})$ solves~(\ref{eq:rbe theorem}) if and only if
  $u^Q = u - X - X^{\zzone} - 2 X^{\zztwo}$ solves
  \[ \LL u^Q = \LL X^{\zzfour} + 4 \partial_x (X^{\zztwo} X) + 2 \partial_x
     (u^Q X) + 2 \partial_x (X^{\zzone} (u^Q + 2 X^{\zztwo})) + \partial_x
     ((u^Q + 2 X^{\zztwo})^2) \]
  with initial condition $u^Q (0) = u_0$. This equation is structurally very
  similar to {\tmname{pam}}~(\ref{eq:pam-eta}) and can be solved using the
  same arguments, which we do not reproduce here.
\end{proof}

For this result to be of any use we still have to show that if $\xi$ is the
space-time white noise, then there is almost surely an element of
$\mathcal{X}_{\tmop{sbe}}$ associated to $\partial_x \xi$. While for
{\tmname{pam}} we needed to construct only one term, here we have to construct
five terms: $X^{\zzone}, X^{\zztwo}, X^{\zzfour}, X^{\zztwo} \reso X, Q \reso
X$. For details we refer to~{\cite{Gubinelli2014}}. Alternatively we can
simply differentiate the extended data which Hairer constructed for the KPZ
equation in Chapter 5 of~{\cite{Hairer2013b}}.

The same approach allows us to solve the KPZ equation $\LL h = (\partial_x
h)^{\renorm 2} + \xi$, and if we are careful how to interpret the product $w
\renorm \xi$, then also the linear heat equation $\LL w = w \renorm \xi$. In
both cases the solution depends continuously on some suitably extended
data that is constructed from $\xi$ in a similar way as described in
Definition~\ref{def:burgers rough distribution}. Moreover, the formal links
between the three equations that we discussed in Section~\ref{sec:rbe,kpz,she}
can be made rigorous. These results are included in~{\cite{Gubinelli2014}}.

\ \


\begin{thebibliography}{BCD11}
  \bibitem[BCD11]{Bahouri2011}Hajer~Bahouri, Jean-Yves~Chemin  and 
  Raphael~Danchin.{\newblock} \tmtextit{Fourier analysis and nonlinear partial
  differential equations}.{\newblock} Springer, Berlin, 2011.{\newblock}
  
  \bibitem[BG97]{Bertini1997}Lorenzo~Bertini  and 
  Giambattista~Giacomin.{\newblock} Stochastic Burgers and KPZ equations from
  particle systems.{\newblock} \tmtextit{Comm. Math. Phys.}, 183(3):571--607,
  1997.{\newblock}
  
  \bibitem[BG13]{bal_limiting_2013}Guillaume~Bal  and  Yu~Gu.{\newblock} Limiting
  models for equations with large random potential; a review.{\newblock}
  \tmtextit{Commun. Math. Sci.}, 13(3), 729--748, 2015.{\newblock}
  
  \bibitem[Bon81]{Bony1981}Jean-Michel~Bony.{\newblock} Calcul symbolique et
  propagation des singularites pour les {\'E}quations aux d{\'e}riv{\'e}es partielles non lin{\'e}aires.{\newblock}
  \tmtextit{Ann. Sci. {\'E}c. Norm. Sup{\'e}r. (4)}, 14:209--246,
  1981.{\newblock}
  
  \bibitem[Cha00]{Chan2000}Terence~Chan.{\newblock} Scaling limits of Wick
  ordered KPZ equation.{\newblock} \tmtextit{Comm. Math. Phys.},
  209(3):671--690, 2000.{\newblock}

   \bibitem[Cor12]{CorwinSurvey2012}
Ivan Corwin. {\newblock} The {K}ardar-{P}arisi-{Z}hang equation and universality class. {\newblock} \tmtextit{Random Matrices Theory Appl.}, 1(1), 2012.
  
  \bibitem[Ech82]{Echeverria1982}Echeverr{\'{\i}}a, Pedro.{\newblock}
   A criterion for invariant measures of {M}arkov processes.{\newblock}\tmtextit{Z. Wahrsch. Verw. Gebiete}, 61(1):1--16, 1982.{\newblock}

  \bibitem[FH14]{Friz2014}Peter~Friz  and  Martin~Hairer.{\newblock} \tmtextit{A
  Course on Rough Paths}.{\newblock} Springer, Berlin, 2014.{\newblock}
  
  \bibitem[FV10]{Friz2010}Peter~Friz  and  Nicolas~Victoir.{\newblock}
  \tmtextit{Multidimensional stochastic processes as rough paths. Theory and
  applications}.{\newblock} Cambridge University Press, Cambridge, 2010.{\newblock}
  
  \bibitem[GIP15]{Gubinelli2012}Massimiliano~Gubinelli, Peter~Imkeller  and 
  Nicolas~Perkowski.{\newblock} Paracontrolled distributions and singular
  PDEs.{\newblock} \tmtextit{Forum Math. Pi}, 3(6), 2015.{\newblock}
  
  \bibitem[GJ10]{goncalves_universality_2010}Patricia~Goncalves  and 
  Milton~Jara.{\newblock} Universality of KPZ equation.{\newblock}
  \tmtextit{ArXiv:1003.4478}, 2010.{\newblock}
  
  \bibitem[GJ13]{gubinelli_regularization_2013}Massimiliano~Gubinelli  and 
  Milton~Jara.{\newblock} Regularization by noise and stochastic Burgers
  equations.{\newblock} \tmtextit{Stochastic Partial Differential Equations:
  Analysis and Computations}, 1(2):325--350, 2013.{\newblock}
  
  \bibitem[GJ14]{goncalves_kpznew_2014}Patricia Goncalves and Milton Jara.{\newblock} Nonlinear fluctuations of weakly asymmetric interacting particle systems. {\newblock} \tmtextit{Arch. Ration. Mech. Anal.}, 212(2), 597--644, 2014.{\newblock}
  
  \bibitem[GP15]{Gubinelli2014}Massimiliano~Gubinelli  and 
  Nicolas~Perkowski.{\newblock} KPZ reloaded.{\newblock} \tmtextit{ArXiv:1508.03877},  2015.{\newblock}

  \bibitem[GP15b]{Gubinelli2015b}Massimiliano~Gubinelli  and 
  Nicolas~Perkowski.{\newblock} Energy solutions of KPZ are unique.{\newblock} 	\tmtextit{ArXiv:1508.07764},  2015.{\newblock}
  
  \bibitem[Gub04]{Gubinelli2004}Massimiliano~Gubinelli.{\newblock} Controlling
  rough paths.{\newblock} \tmtextit{J. Funct. Anal.}, 216(1):86--140, 
  2004.{\newblock}
  
  \bibitem[Gub10]{Gubinelli2010}Massimiliano~Gubinelli.{\newblock}
  Ramification of rough paths.{\newblock} \tmtextit{J. Differential
  Equations}, 248(4):693--721, 2010.{\newblock}
  
  \bibitem[Hai13]{Hairer2013b}Martin~Hairer.{\newblock} Solving the KPZ
  equation.{\newblock} \tmtextit{Ann. Math.}, 178(2):559--664,
  2013.{\newblock}
  
  \bibitem[Hai14]{Hairer2014}Martin~Hairer.{\newblock} A theory of regularity
  structures.{\newblock} \tmtextit{Invent. math.}, 198(2):269--504, 2014.{\newblock}
  
  \bibitem[HPP13]{hairer_random_2013}Martin~Hairer, Etienne~Pardoux  and 
  Andrey~Piatnitski.{\newblock} Random homogenisation of a highly oscillatory
  singular potential.{\newblock} \tmtextit{Stochastic Partial Differential
  Equations: Analysis and Computations}, 1(4):571--605, 2013.{\newblock}
  
  \bibitem[Hu02]{Hu2002}Yaozhong~Hu.{\newblock} Chaos expansion of heat
  equations with white noise potentials.{\newblock} \tmtextit{Potential
  Anal.}, 16(1):45--66, 2002.{\newblock}
  
  \bibitem[Jan97]{Janson1997}Svante~Janson.{\newblock} \tmtextit{Gaussian
  Hilbert spaces},  volume  129  of \tmtextit{Cambridge Tracts in
  Mathematics}.{\newblock} Cambridge University Press, Cambridge,
  1997.{\newblock}
  
  \bibitem[KS98]{karazas} Ioannis~Karatzas and Steven~Shreve. \newblock 
  \tmtextit{Brownian Motion and Stochastic Calculus}. \newblock Springer-Verlag New York, 1998.
  
  \bibitem[KPZ86]{Kardar1986}Mehran~Kardar, Giorgio~Parisi  and 
  Yi-Cheng~Zhang.{\newblock} Dynamic scaling of growing interfaces.{\newblock}
  \tmtextit{Physical Review Letters}, 56(9):889--892, 1986.{\newblock}
  
  \bibitem[LCL07]{Lyons2007}Terry J.~Lyons, Michael~Caruana  and 
  Thierry~L{\'e}vy.{\newblock} \tmtextit{Differential equations driven
  by rough paths},  volume  1908  of \tmtextit{Lecture Notes in
  Mathematics}.{\newblock} Springer, Berlin, 2007.{\newblock}
  
  \bibitem[LQ02]{Lyons2002}Terry J.~Lyons  and  Zhongmin~Qian.{\newblock}
  \tmtextit{System control and rough paths}.{\newblock} Oxford University
  Press, 2002.{\newblock}
  
  \bibitem[Lyo98]{Lyons1998}Terry J.~Lyons.{\newblock} Differential equations
  driven by rough signals.{\newblock} \tmtextit{Rev. Mat. Iberoam.},
  14(2):215--310, 1998.{\newblock}
  
   \bibitem[Qua12]{QuastelSurvey2012}
Jeremy Quastel. {\newblock} Introduction to {KPZ}. {\newblock} \tmtextit{Current developments in
  mathematics, 2011}, Int. Press, Somerville, MA, 2012. {\newblock}

   \bibitem[QS15]{Quastel2015}
Jeremy Quastel and Herbert Spohn.{\newblock} The One-Dimensional {KPZ} Equation and Its Universality Class. {\newblock} \tmtextit{J. Stat. Phys.}, 160(4):965--984, 2015. {\newblock}
  
  \bibitem[RVW01]{Russo2001}Francesco Russo, Pierre Vallois and Jochen Wolf.{\newblock} A generalized class of Lyons-Zheng processes.{\newblock} \tmtextit{Bernoulli},
  7(2):363--379, 2001.{\newblock}  
  
  \bibitem[ST87]{Schmeisser1987}Hans-J{\"u}rgen~Schmeisser  and  Hans~Triebel.{\newblock}
  \tmtextit{Topics in Fourier analysis and function spaces}.{\newblock} Akademische
  Verlagsgesellschaft Geest \& Portig K.-G., Leipzig, 1987.{\newblock}
  
  \bibitem[You36]{Young1936}Laurence C.~Young.{\newblock} An inequality of the
  H{\"o}lder type, connected with Stieltjes integration.{\newblock}
  \tmtextit{Acta Math.}, 67(1):251--282, 1936.{\newblock}
\end{thebibliography}
\end{document}